\documentclass[12pt,oneside]{amsart}
\usepackage{amsthm,amsmath,amssymb,enumerate}
\usepackage{mathrsfs}
\usepackage[letterpaper,margin=1in]{geometry}
\usepackage{microtype}
\usepackage{amscd}
\usepackage{mathtools}
\usepackage{cases}
\usepackage{xfrac}
\usepackage{color}
\usepackage{enumitem}
\usepackage{xr-hyper}
\usepackage[colorlinks=true,allcolors=black]{hyperref}
\usepackage[nameinlink,capitalize,noabbrev]{cleveref}

\theoremstyle{plain}
\newtheorem{thm}{Theorem}[section]
\newtheorem{theorem}[thm]{Theorem}

\newtheorem{corollary}[thm]{Corollary}

\newtheorem{lemma}[thm]{Lemma}

\newtheorem{proposition}[thm]{Proposition}

\newtheorem{introtheorem}{Theorem}

\theoremstyle{definition}
\newtheorem{definition}[thm]{Definition}

\theoremstyle{plain}

\crefname{assumption}{assumption}{assumptions}
\Crefname{assumption}{Assumption}{Assumptions}

\theoremstyle{definition}
\newtheorem{remark}[thm]{Remark}
\newtheorem{example}[thm]{Example}

\numberwithin{equation}{section}

\makeatletter
\def\input@path{{./}{output/}}
\makeatother

\newcommand{\ZZ}{\mathbb{Z}}
\newcommand{\QQ}{\mathbb{Q}}
\newcommand{\RR}{\mathbb{R}}

\newcommand{\ad}{_\mathrm{ad}}            

\newcommand{\barint}{\mbox{$ave \int$}}  
\def\barint_#1{\mathchoice
	{\mathop{\vrule width 6pt
			height 3 pt depth -2.5pt
			\kern -8.8pt
			\intop}\nolimits_{#1}}%
	{\mathop{\vrule width 5pt height
			3 pt depth -2.6pt
			\kern -6.5pt
			\intop}\nolimits_{#1}}%
	{\mathop{\vrule width 5pt height
			3 pt depth -2.6pt
			\kern -6pt
			\intop}\nolimits_{#1}}%
	{\mathop{\vrule width 5pt height
			3 pt depth -2.6pt
			\kern -6pt \intop}\nolimits_{#1}}}

\newcommand{\BP}{{\mathbb {P}}}

\newcommand{\rank}{{\mathrm{rank}}}

\newcommand{\Pic}{\mathrm{Pic}}

	\newcommand{\Spec}{{\mathrm{Spec \,}}}
	\newcommand{\Sym}{{\mathrm{Sym}}}

	\newcommand{\Td}{{\mathrm{Td}}}

	\newcommand{\norm}[1]{\|{#1}\|}

	\newcommand{\ch}{\mathrm{ch}}
	
	\newcommand{\Num}{\mathrm{Num}}

\newcommand{\intg}{\mathrm{int}}
\newcommand{\nef}{\mathrm{nef}}

\newcommand{\whdeg}{\widehat{\deg}}
\newcommand{\whc}{\widehat{c}}
\newcommand{\whs}{\widehat{s}}
\newcommand{\whch}{\widehat{\ch}}
\newcommand{\whTd}{\widehat{\Td}}
\newcommand{\Pj}{\BP}

\title[Characteristic classes of adelic vector bundles and applications]
{Characteristic classes of adelic vector bundles and applications.}
\author{Jiahui Gao}


\begin{document}

\begin{abstract}
We define numerical adelic characteristic classes for vector bundles on
projective varieties over number fields.  The metric data are carried by the
tautological quotient line bundle on the
associated projective bundle.  The construction uses the established
intersection theory of adelic line bundles.  Higher characteristic classes
are multilinear intersection functionals.  They are continuous for
simultaneously controlled model sequences that are Cauchy for the supremum
norm.  We construct
the resulting tautological numerical intersection algebra.  We also prove
controlled numerical Bogomolov--Gieseker inequalities.  We first treat the case of a curve over
$K$,
where the numerical inequality is combined with the Deligne pairing and the
adelic Hodge index theorem.  We then pass to higher dimension, using
Moriwaki's model-level dimension induction before taking the controlled
Zhang limit.  The curve theorem includes equality and uniform-gap
criteria.
\end{abstract}
\maketitle

\setcounter{tocdepth}{1}
\tableofcontents

\section*{Introduction}

This paper constructs numerical characteristic classes of adelic vector
bundles on projective varieties over number fields.  The construction uses
intersection numbers of the tautological quotient line bundle on a
projective bundle.  We use these classes to prove a numerical
Bogomolov--Gieseker inequality under explicit conditions on the
approximating arithmetic models.

Gillet and Soul\'e constructed arithmetic characteristic classes for
Hermitian vector bundles on regular arithmetic varieties
\cite{GSCharI,GSCharII}.  Zhang developed intersection theory for adelic
line bundles on projective varieties by approximation with model metrics
\cite{Zhang}.  We also use the intersection and Deligne pairing formalism
of Yuan and Zhang \cite{YZ}.  These theories suggest the following
question: which characteristic numbers of vector bundles can we define
using only adelic intersections of line bundles, and which inequalities
on arithmetic models pass to those numbers?

Throughout the paper, $K$ is a number field.  Let $X$ be a projective
integral $K$-variety of dimension $d$.  Let $E$ be a vector bundle of
rank $r>0$ on $X$.
Put $\pi:Y:=\mathbb P_X(E)\to X$ and
$\xi_E:=\mathcal O_Y(1)$, with the convention that $Y$ parametrizes line
quotients of $E$.  We equip $\xi_E$ with an integrable adelic metric
$\overline\xi_E$ and write $\overline E=(E,\overline\xi_E)$.
In particular, an integrable adelic metric on $E$ gives such data through
its quotient metric.

When $X$ is normal, the numerical Segre class of degree $i$ is the
functional
$$
\left\langle\widehat s_i(\overline E),\overline L_1,\ldots,\overline L_{d+1-i}\mid X\right\rangle
:=\widehat{\deg}_Y\!\left(\widehat c_1(\overline\xi_E)^{r-1+i}\prod_{j=1}^{d+1-i}\widehat c_1(\pi^*\overline L_j)\right),
$$
where $0\leq i\leq d+1$ and the $\overline L_j$ are integrable adelic
line bundles on $X$.  For each integral closed subvariety of $X$, we use
the same formula on its normalization.  Fibre products of projective
bundles define the corresponding Segre monomial numbers.  Universal
polynomials in these monomials then define numerical Chern, Chern
character, and Todd classes.

\begin{introtheorem}[\Cref{thm:main}]
\label{intro:characteristic-classes}
The numerical characteristic functionals of $\overline E$ are well
defined and depend only on $\overline\xi_E$.  They are symmetric and
multilinear in the adelic line bundles used in the pairing.  Their model evaluations converge
when the factors on every fibre product in the defining expansion satisfy
the simultaneous control condition of
\cref{def:controlled-adelic-convergence}.  The limits do not depend on the
models used to compute these evaluations.  The functionals satisfy the
finite normalized base change formula and the generically finite
projection formula.  Segre monomial numbers are nonnegative when the
tautological line bundles and the adelic line bundles used in the pairing
are all nef.
\end{introtheorem}

We also construct a commutative graded real algebra
$\widehat N^\bullet\ad(X)$ that contains the divisor and tautological
Segre classes.  Intersections on fibre products of projective bundles
define its degree functional.  We quotient the formal algebra by the
elements that pair to zero with every element of complementary degree.
The resulting pairings are nondegenerate
(\cref{prop:complementary-duality}).  This algebra defines mixed products
of characteristic classes of different vector bundles.

Our main application concerns the discriminant.  Assume now that $X$ is
smooth and geometrically integral, with $d\geq1$, and that $r\geq2$.
Let $L$ be ample and let $\overline L$ be a nef adelic metric on $L$.
Define the numerical discriminant by
$$
\Delta^\tau_{\overline L}(E,\overline\xi_E)
:=\left\langle 2r\widehat c_2(\overline E)-(r-1)\widehat c_1(\overline E)^2,\overline L^{d-1}\mid X\right\rangle.
$$
Here $\overline L^{d-1}$ denotes $d-1$ copies of $\overline L$.
By \cref{lem:bg-numerical-bridge}, this is the tautological intersection
number in \cref{eq:bg-tautological-discriminant}.

For this application, we require a model system as in
\cref{def:bg-admissible-model-system}.  The models are regular and carry
smooth Hermitian vector bundles extending $E$.  After pullback to suitable
dominating models, each metrized tautological quotient line bundle is a
difference of two nef Hermitian rational line bundles.
For each sign, the generic fibre stays fixed.
The induced model adelic metrics converge in the sense of
\cref{def:cauchy-boundary-topology}.
When $d\geq2$, the polarizations also come from a convergent sequence of
arithmetically ample Hermitian rational model line bundles.

\begin{introtheorem}[\Cref{thm:curve-controlled-bg,thm:controlled-numerical-bg}]
\label{intro:bg}
Assume that $E_{\overline K}$ is slope semistable with respect to
$L_{\overline K}$.  If $(\overline L,\overline\xi_E)$ admits a model
system satisfying \cref{def:bg-admissible-model-system}, then
$$
\Delta^\tau_{\overline L}(E,\overline\xi_E)\geq0.
$$
\end{introtheorem}

When $d=1$, the discriminant contains no polarization factor.  We write
$C=X$ and $\Delta_C^\tau=\Delta^\tau_{\overline L}$ in this case.
The Deligne pairing gives an additional conclusion on curves.

\begin{introtheorem}[\Cref{thm:curve-pb}]
\label{intro:curve-refinement}
Under the assumptions of the preceding theorem, suppose that $X=C$ is
a curve and $\det E\simeq\mathcal O_C$.  Let
$\overline A:=\langle\overline\xi_E,\ldots,\overline\xi_E\rangle_{Y/C}$
be the Deligne pairing with $r$ factors.
Write $\overline\xi_E^{\,r+1}$ for the normalized top intersection number
on $Y$ and $\overline A^2$ for the normalized self-intersection number on $C$.
Then
$$
\overline\xi_E^{\,r+1}
=\frac{r+1}{2r}\overline A^2-\frac{1}{2r}\Delta_C^\tau(E,\overline\xi_E)\leq0.
$$
The number $\overline\xi_E^{\,r+1}$ is zero if and only if both $\overline A^2$ and
$\Delta_C^\tau(E,\overline\xi_E)$ vanish.
\end{introtheorem}

The proof of the discriminant inequality has two steps.  On each regular
arithmetic model, Mourougane's calculation \cite{Mourougane} compares the
tautological Segre classes with the corrected arithmetic classes.
The individual classes differ by Bott--Chern terms.  We prove that the
first two correction terms cancel in the discriminant
(\cref{prop:bg-bott-chern-cancellation}).  This identifies its
tautological expression with the Gillet--Soul\'e discriminant degree.
Moriwaki's inequalities give nonnegativity of this degree on arithmetic
surfaces and in higher dimension
\cite{MoriwakiSurface,MoriwakiHigherBG}.  We then pass to the adelic limit
using the fixed nef decompositions.  The resulting limit is the numerical
discriminant in Theorem~\ref{intro:bg} (\cref{prop:bg-controlled-transfer}).

For the curve refinement, the underlying line bundle $A$ of
$\overline A$ is $\det E$.  The determinant hypothesis gives $\deg A=0$.
The adelic Hodge index theorem \cite{YuanZhangHodge} therefore gives
$\overline A^2\leq0$.  The identity in Theorem~\ref{intro:curve-refinement} then proves the
nonpositivity of $\overline\xi_E^{\,r+1}$ and characterizes its vanishing.

Section~1 constructs the numerical characteristic functionals.
Section~2 constructs their tautological numerical intersection algebra.
Section~\ref{sec:curve-pb} proves the model comparison, the passage to
the limit, and the curve results.  Section~\ref{sec:controlled-numerical-bg}
applies the same comparison and limit argument in higher dimension.

\label{chap:numerical-ad-char}

\section{Numerical adelic characteristic classes}

Let $K$ be a number field.
All absolute values and degrees are normalized so that heights are invariant
under finite extension of $K$.  Let $X$ be a projective integral $K$-variety
of dimension $d$.  Throughout this section, every vector bundle has positive
rank.  Let
$$
  \widehat{\Pic}(X)_{\intg,\RR}
$$
be the real vector space of integrable adelic line bundles, and let
$\widehat{\Pic}(X)_{\nef,\RR}$ be its nef cone.

Let the absolute adelic intersection pairing be
\begin{equation}
  (\overline L_0,\ldots,\overline L_d)
  \longmapsto
  \whdeg_X\!\left(
    \whc_1(\overline L_0)\cdots\whc_1(\overline L_d)
  \right)\in\RR.
  \label{eq:adelic-intersection}
\end{equation}
Zhang constructs the classical adelic intersection pairing using
semipositive model approximations. He extends this pairing to integrable
metrics by multilinearity; see
\cite[Theorem~1.4 and Section~1.5]{Zhang}.
We state Zhang's convergence condition for model metrics in
\cref{def:cauchy-boundary-topology} below.
For projective varieties over number fields, this construction agrees with
the adelic intersection pairing of \cite[Proposition~4.1.1]{YZ}.
The following properties hold throughout.

\begin{enumerate}[label=\textup{(I\arabic*)},leftmargin=2.2em]
\item The pairing is symmetric and multilinear.
\item It depends only on the isometry classes of the adelic line bundles.
\item It is continuous when the varying model sequences are simultaneously
controlled in the sense of \cref{def:controlled-adelic-convergence} below.
\item Pullback and proper pushforward satisfy the projection formula.
\item The intersection number of nef adelic line bundles is nonnegative.
\item Normalized intersections are invariant under finite extension of $K$.
\end{enumerate}

\begin{definition}
\label{def:cauchy-boundary-topology}
Let $X$ be a projective integral variety over the number field $K$.
Let $N$ be a rational line bundle on $X$.
Let $(\overline N_m)_{m\geq1}$ be a sequence of model adelic line bundles
on $N$, with fixed identifications on generic fibre.
Write $\|\cdot\|_{m,v}$ for the metric at a place $v$ of $K$.
We call this sequence a \emph{Cauchy sequence for the supremum norm}
if the following two conditions hold.

First, choose a finite set $S$ of places of $K$ containing all archimedean
places.  Let $U\subseteq\operatorname{Spec}\mathcal O_K$ be the complement
of the finite primes in $S$.
Choose a projective integral flat model $\mathcal X_U$ of $X$ over $U$
and a rational line bundle $\mathcal N_U$ extending $N$.
For every $m$ and every $v\notin S$, require $\|\cdot\|_{m,v}$ to equal
the model metric induced by $(\mathcal X_U,\mathcal N_U)$.
Keep $S$ and this model fixed as $m$ varies.

Second, require the metrics to be uniformly Cauchy at each place in $S$.
More precisely, let $X_v^{\mathrm{an}}$ be the analytic space of $X$ at $v$.
For an ordinary line bundle $N$, define
$$
 g_{j,m,v}(x):=-\log\frac{\|s(x)\|_{j,v}}{\|s(x)\|_{m,v}},
 \qquad x\in X_v^{\mathrm{an}},
$$
where $s$ is a nowhere vanishing local section of $N$ near $x$.
The ratio does not depend on the choice of $s$.
For a rational line bundle $N$, choose a positive integer $q$ such that
$N^{\otimes q}$ is an ordinary line bundle.
Compute the negative logarithmic ratio for the induced metrics on $N^{\otimes q}$
and divide by $q$.  This defines $g_{j,m,v}$ independently of $q$.
Put $g_{j,m}:=(g_{j,m,v})_{v\in S}$.
Choose positive rational numbers $\epsilon_m\to0$ such that
\begin{equation}
 \|g_{j,m}\|_{\sup,S}
 :=\max_{v\in S}\sup_{x\in X_v^{\mathrm{an}}}|g_{j,m,v}(x)|
 \leq\epsilon_m
 \qquad (j\geq m\geq1).
 \label{eq:cauchy-boundary-bounds}
\end{equation}

The limit metric equals the fixed model metric at every $v\notin S$.
For each $v\in S$, the functions $g_{m,1,v}$ converge uniformly to a
continuous function $h_v$, since $X_v^{\mathrm{an}}$ is compact.
Define the limit metric at $v$ by
$$
 \|s(x)\|_v:=e^{-h_v(x)}\|s(x)\|_{1,v}.
$$
We use this meaning of convergence throughout the paper.
\end{definition}

This is Zhang's convergence condition on a projective variety
\cite[Section~1.2]{Zhang}.
Under the fixed model condition outside $S$, the supremum norm condition
also agrees with the Cauchy condition for the boundary topology;
see \cite[Theorem~3.5.2 and its proof]{YZ}.

We use the positivity conventions for projective varieties in
\cite[Appendices~A.4.1 and~A.5.2]{YZ}.
Let $V$ be a projective integral variety over $K$.
A model adelic rational line bundle on $V$ is \emph{nef} if a nef
Hermitian rational line bundle on a projective arithmetic model over
$\mathcal O_K$ induces it.
Let $\overline N$ be an adelic rational line bundle on $V$.
We call $\overline N$ \emph{nef} if a sequence of nef model adelic
rational line bundles satisfies \cref{def:cauchy-boundary-topology} and
converges to $\overline N$.
We call $\overline N$ \emph{integrable} if it is the difference of two
nef adelic rational line bundles.
For projective varieties over number fields, this nefness also agrees with
strong nefness in \cite[Section~3.5.2, after Theorem~3.5.2]{YZ}.

\begin{definition}
\label{def:controlled-adelic-convergence}
Let $V$ be a projective integral variety over $K$, and let $A$ be a finite
index set.  For each $\alpha\in A$, let
$$
 \bigl(\overline N_{\alpha,m}\bigr)_{m\geq1}
$$
be a sequence of model adelic $\mathbb Q$-line bundles on a fixed
line bundle $N_\alpha$ on $V$.  We say that these sequences are
\emph{simultaneously controlled} if the following data exist.  For every
$\alpha\in A$ there are fixed $\mathbb Q$-line bundles $N_\alpha^+$ and
$N_\alpha^-$ and nef model adelic line bundles
$\overline N_{\alpha,m}^\pm$ such that
\begin{equation}
 \overline N_{\alpha,m}
 =\overline N_{\alpha,m}^+-\overline N_{\alpha,m}^-,
 \qquad
 N_\alpha=N_\alpha^+-N_\alpha^-.
 \label{eq:controlled-fixed-decomposition}
\end{equation}
We fix the identifications on generic fibre.
For every $\alpha$ and each sign, require
$(\overline N_{\alpha,m}^\pm)_{m\geq1}$ to be a Cauchy sequence for the
supremum norm in the sense of \cref{def:cauchy-boundary-topology}.
Choose the same finite set $S$ and sequence $(\epsilon_m)$ for all these
sequences.  Over the resulting open subset $U$ of
$\operatorname{Spec}\mathcal O_K$, choose one projective integral flat
model carrying all the fixed rational line bundles.
Require these model line bundles to induce the fixed metrics outside $S$.
Write $\overline N_\alpha^\pm$ for the limit of
$(\overline N_{\alpha,m}^\pm)_{m\geq1}$.
These limit adelic line bundles are nef by the preceding definition.
We then write
$$
 \overline N_{\alpha,m}\xrightarrow{\mathrm{ctrl}}
 \overline N_\alpha
 :=\overline N_\alpha^+-\overline N_\alpha^-.
$$
We call the chosen decompositions, model data, finite set $S$, and sequence
$(\epsilon_m)$ the control data.
\end{definition}

We use the same notation when some factors are fixed integrable adelic line
bundles.  For each such factor, choose an expression as a difference of two
nef adelic line bundles.
For each summand, choose a sequence of nef model adelic line bundles that
converges to it and is Cauchy for the supremum norm.
Together with the varying model sequences, these sequences form one
simultaneously controlled family.
Pullback preserves this simultaneous control.
Enlarge $S$ so that the morphism extends to the fixed models over $U$.
Pullback then preserves the fixed model metrics outside $S$.
It does not increase the supremum of a logarithmic metric ratio.
Zhang proves convergence for semipositive model approximations
\cite[Theorem~1.4]{Zhang}.
For the nef model metrics used here, \cite[Proposition~4.1.1]{YZ} gives
this convergence, including for continuous Hermitian model metrics.
A finite multilinear expansion gives
\begin{equation}
 \lim_{m\to\infty}
 \whdeg_V\left(\prod_{j=0}^{\dim V}
 \whc_1(\overline N_{j,m})\right)
 =
 \whdeg_V\left(\prod_{j=0}^{\dim V}
 \whc_1(\overline N_j)\right)
 \label{eq:controlled-intersection-continuity}
\end{equation}
whenever the varying model sequences are simultaneously controlled.
The fixed nef decompositions are part of this continuity statement; see
\cite[Appendix~A.5]{YZ}.

For an integral closed subvariety $Z\subseteq X$, all intersections are taken
on its normalization $\nu_Z\colon Z^\nu\to Z$.  Thus we denote
$\nu_Z^*(E|_Z)$ by $E|_Z$ below.

\subsection{Adelic vector bundles on projective varieties}

Let $\Pj(E):=\operatorname{Proj}_X(\operatorname{Sym}E)$ be the projective
bundle associated with $E$, and let $\pi:\Pj(E)\to X$ be its projection.
Let $\mathcal O_{\Pj(E)}(1)$ be the tautological quotient line bundle.  Let
\begin{equation}
\phi:  \pi^*E\twoheadrightarrow\mathcal O_{\Pj(E)}(1)
  \label{eq:taut-quotient}
\end{equation}
be the canonical surjection.

\begin{definition}[Quotient metric]
Let $v$ be a place of $K$.  A continuous norm on $E$ induces a norm on
$\mathcal O_{\Pj(E)}(1)$ by
\begin{equation}
  \norm{\ell}_v^Q
  =\inf\{\|e\|_{E,v}: \phi(e)=\ell\}.
  \label{eq:quotient-metric}
\end{equation}
We call the induced metric the \emph{quotient metric}.  Denote the resulting
metrized tautological quotient line bundle by
$$
  \overline\xi_E:=\overline{\mathcal O}_{\Pj(E)}(1).
$$
\end{definition}

\begin{definition}
\label{def:adelic-vector-bundle}
Let $E$ be a metrized vector bundle on $X$.

\begin{enumerate}[label=\textup{(\roman*)}]
	\item The metric on $E$ is called \emph{adelic} if the induced metrized
	line bundle $\overline\xi_E$
	is adelic.
	
	\item A \emph{model of $(X,E)$} consists of data
	$(\mathcal X,\mathcal E,h)$ as follows.
	The scheme $\mathcal X$ is integral, projective, and flat over
	$\operatorname{Spec}\mathcal O_K$.
	The sheaf $\mathcal E$ is a vector bundle on $\mathcal X$.
	We fix identifications $\mathcal X_K\simeq X$ and
	$\mathcal E_K\simeq E$.
	For each embedding $\sigma:K\hookrightarrow\mathbb C$, put
	$X_\sigma:=X\times_{K,\sigma}\mathbb C$ and let $E_\sigma$ be the
	pullback of $E$ to $X_\sigma$.
	The family $h=(h_\sigma)_\sigma$ consists of continuous positive definite
	Hermitian metrics on $E_\sigma$ over $X_\sigma(\mathbb C)$.

	Write $\bar\sigma$ for the conjugate embedding.
	Let $c_\sigma:X_\sigma(\mathbb C)\to X_{\bar\sigma}(\mathbb C)$ be
	complex conjugation, and denote its induced antilinear map on vector
	bundle fibres by the same symbol.
	For every $x\in X_\sigma(\mathbb C)$ and $e\in E_{\sigma,x}$, require
	$$
	\|c_\sigma(e)\|_{h_{\bar\sigma},c_\sigma(x)}
	=\|e\|_{h_\sigma,x}.
	$$

	At a nonarchimedean place $v$, the model vector bundle $\mathcal E$
	defines the norm on $E_v^{\mathrm{an}}$.
	Let $x\in X_v^{\mathrm{an}}$, and let $\mathcal H(x)$ be its completed
	residue field.
	Use the absolute value on $\mathcal H(x)$ defined by $x$, and denote it
	by $|\cdot|_v$.
	Properness of $\mathcal X$ gives $x$ a unique centre on $\mathcal X$.
	We call this centre the reduction of $x$.
	Choose a local basis $e_1,\ldots,e_r$ of $\mathcal E$ near the reduction
	of $x$, where $r=\operatorname{rk}E$.
	The model norm is
	$$
	\left\|\sum_{i=1}^r a_i e_i(x)\right\|_{v,x}
	=\max_{1\leq i\leq r}|a_i|_v,
	\qquad a_i\in\mathcal H(x).
	$$
	A change of model basis and its inverse have coefficients of absolute
	value at most one at $x$.
	Hence this formula does not depend on the chosen model basis.
	At an archimedean place represented by $\sigma$, use the norm from
	$h_\sigma$.
	We call the resulting family of norms on $E$ a \emph{model metric}.

	Let $\pi_{\mathcal X}:\mathbb P_{\mathcal X}(\mathcal E)\to\mathcal X$
	be the projection.
	At every place, give the tautological line bundle the quotient metric
	of \cref{eq:quotient-metric} through the canonical surjection
	$$
	\pi_{\mathcal X}^*\mathcal E
	\twoheadrightarrow
	\mathcal O_{\Pj_{\mathcal X}(\mathcal E)}(1).
	$$
	At finite places, this is the model metric induced by
	$\mathcal O_{\Pj_{\mathcal X}(\mathcal E)}(1)$.
	At archimedean places, it is the Hermitian quotient metric induced by
	$h_\sigma$.
	The metrics $h_\sigma$ need not be smooth in this definition.
	We impose smoothness explicitly when we use differential forms.
	
	\item An adelic metric on $E$ is called \emph{integrable}, respectively
	\emph{nef}, if $\overline\xi_E$ is an integrable, respectively nef, adelic
	line bundle on $\mathbb P_X(E)$ in the sense of
	\cite[Appendix~A.5.2]{YZ}.
\end{enumerate}
\end{definition}

\begin{definition}
We say that a metric on $E$ is \emph{controlled by vector-bundle models} if
we can choose the following data.
Choose a sequence of models $(\mathcal X_m,\mathcal E_m,h_m)$ in the sense
of \cref{def:adelic-vector-bundle}\textup{(ii)}.
Require their induced quotient metrics on $\mathcal O_{\mathbb P(E)}(1)$
to converge to $\overline\xi_E$ under the control condition of
\cref{def:controlled-adelic-convergence}.
\end{definition}

All constructions below depend on the metric on $E$ only through the induced
metrized tautological line bundle $\overline\xi_E$.  Hence we may start
directly with an integrable adelic metric on $\mathcal O_{\Pj(E)}(1)$.
When a result requires model control, this metric must satisfy the
corresponding hypothesis.

\begin{lemma}[Pullback]
\label{lem:pullback}
Let $Y$ be a projective integral $K$-variety, and let
$f\colon Y\to X$ be a morphism.  Equip $f^*E$ with the pullback metric.
Then there is a cartesian diagram
$$
\begin{array}{ccc}
\Pj(f^*E) & \xrightarrow{\ \widetilde f\ } & \Pj(E)\\
\pi_Y\downarrow\phantom{\pi_Y} && \phantom{\pi_X}\downarrow\pi_X\\
Y & \xrightarrow{\ f\ } & X
\end{array}
$$
and a canonical isometry
\begin{equation}
  \overline\xi_{f^*E}\simeq\widetilde f^*\overline\xi_E.
  \label{eq:pullback-taut}
\end{equation}
Hence integrability and nefness are preserved by pullback.
\end{lemma}

\begin{proof}
The canonical isomorphism of graded $\mathcal O_Y$-algebras
$$
\operatorname{Sym}(f^*E)\simeq f^*\operatorname{Sym}(E)
$$
and base change for relative Proj give
$$
\Pj(f^*E)\simeq Y\times_X\Pj(E).
$$
Let $\phi_X$ and $\phi_Y$ denote the tautological quotient maps on
$\Pj(E)$ and $\Pj(f^*E)$, respectively.  Under the above isomorphism, the
canonical identifications fit into the commutative diagram
$$
\begin{array}{ccc}
\pi_Y^*f^*E & \xrightarrow{\ \phi_Y\ } &
\mathcal O_{\Pj(f^*E)}(1)\\
\downarrow && \downarrow\\
\widetilde f^*\pi_X^*E &
\xrightarrow{\ \widetilde f^*\phi_X\ } &
\widetilde f^*\mathcal O_{\Pj(E)}(1).
\end{array}
$$
Fix a place $v$ of $K$, and let $\mathbb C_v$ be the completion of an
algebraic closure of $K_v$.  For a point
$z\in\Pj(f^*E)(\mathbb C_v)$, the left vertical identification is an
isometry because both metrics are the pullback of the metric on $E$ along
$f\circ\pi_Y=\pi_X\circ\widetilde f$.  After identifying the two line-bundle
fibers by the right vertical arrow, both quotient norms are computed in the
same normed $\mathbb C_v$-vector space and with the same quotient map.
Thus \cref{eq:quotient-metric} gives, for every $\ell$ in that fiber,
$$
\begin{aligned}
\|\ell\|_{\overline\xi_{f^*E},v,z}
&=\inf_{\phi_{Y,z}(e)=\ell}
  \|e\|_{\pi_Y^*f^*E,v,z}\\
&=\inf_{(\widetilde f^*\phi_X)_z(e)=\ell}
  \|e\|_{\widetilde f^*\pi_X^*E,v,z}
 =\|\ell\|_{\widetilde f^*\overline\xi_E,v,z}.
\end{aligned}
$$
The $\mathbb C_v$-points are dense in
$\Pj(f^*E)_v^{\mathrm{an}}$, and both metrics are continuous.  The equality
therefore holds at every analytic point, so the right vertical identification
is an isometry.  Pullback preserves integrable and nef adelic line bundles.
The final assertion now follows from the definition of integrable and nef
adelic metrics on vector bundles.
\end{proof}

\begin{lemma}[Twisting]
\label{lem:twisting}
For an integrable adelic line bundle $\overline M$ on $X$, equip
$E\otimes M$ with the tensor product metric.  Let
$\tau\colon\Pj(E\otimes M)\xrightarrow{\sim}\Pj(E)$ be the canonical
isomorphism, and let $\pi_{E\otimes M}\colon\Pj(E\otimes M)\to X$ be the
projection.  There is a canonical isometry
\begin{equation}
  \overline\xi_{E\otimes M}
  \simeq\tau^*\overline\xi_E\otimes
  \pi_{E\otimes M}^*\overline M.
  \label{eq:twisting-taut}
\end{equation}
\end{lemma}

\begin{proof}
Let $p\colon T\to X$ be an $X$-scheme.  Tensoring a line-bundle quotient
$q\colon p^*E\twoheadrightarrow Q$ with $p^*M$ gives
$$
q\otimes\operatorname{id}_{p^*M}\colon
p^*(E\otimes M)\twoheadrightarrow Q\otimes p^*M.
$$
This construction is functorial in $T$, and tensoring with $p^*M^{-1}$ gives
its inverse.  Since $\Pj(E)$ represents the functor of line-bundle quotients
of $E$, the construction defines $\tau^{-1}$, whose inverse is $\tau$.
Applying it to the universal quotient yields a canonical isomorphism
$$
\mathcal O_{\Pj(E\otimes M)}(1)
\simeq
\tau^*\mathcal O_{\Pj(E)}(1)\otimes\pi_{E\otimes M}^*M.
$$
Fix a place $v$ of $K$, let $\mathbb C_v$ be the completion of an algebraic
closure of $K_v$, and take
$z\in\Pj(E\otimes M)(\mathbb C_v)$.  Put $x=\pi_{E\otimes M}(z)$, and write
$q_z\colon E_x\twoheadrightarrow L_z$ for the quotient represented by
$\tau(z)$.  The quotient represented by $z$ is
$q_z\otimes\operatorname{id}_{M_x}$.  If $m\in M_x$ is nonzero, then
$e\mapsto e\otimes m$ identifies $E_x$ with $E_x\otimes M_x$.  Hence, for
every $\ell\in L_z$, the definition of the quotient metric gives
$$
\begin{aligned}
\|\ell\otimes m\|_{\overline\xi_{E\otimes M},v,z}
&=\inf_{q_z(e)=\ell}
  \|e\otimes m\|_{E\otimes M,v,x}\\
&=\left(\inf_{q_z(e)=\ell}\|e\|_{E,v,x}\right)
  \|m\|_{\overline M,v,x}\\
&=\|\ell\|_{\overline\xi_E,v,\tau(z)}
  \|m\|_{\overline M,v,x}.
\end{aligned}
$$
The last expression is the tensor product norm on the right-hand side of
\cref{eq:twisting-taut}.  Since the $\mathbb C_v$-points are dense in the
corresponding analytification and both metrics are continuous, the algebraic
isomorphism above is an isometry at every place.
\end{proof}

\subsubsection{Numerical adelic Segre classes}

\begin{definition}
For $i\geq0$, define $\Num^i\ad(X)$ to be the product, over all integral
closed subvarieties $Z\subseteq X$ of dimension $e\geq i-1$, of the spaces of
symmetric multilinear maps satisfying controlled continuity in the sense of
\cref{def:controlled-adelic-convergence}
\begin{equation}
  \widehat{\Pic}(Z^\nu)_{\intg,\RR}^{\,e+1-i}\longrightarrow\RR.
  \label{eq:num-space}
\end{equation}
Since $Z$ has arithmetic dimension $e+1$, the corresponding functional takes
$e+1-i$ integrable adelic line bundles as arguments.
When $e+1-i=0$, the corresponding factor is $\RR$.  An element of
$\Num^i\ad(X)$ is called a \emph{numerical adelic class of codimension
$i$}.
\end{definition}

\begin{definition}[Numerical tautological Segre classes]\label{def:segre}
		Let $\overline E$ be an integrable adelic vector bundle of rank $r$ on $X$.
		For each integral closed subvariety $Z\subseteq X$, write $e:=\dim Z$, let
		$\nu_Z:Z^\nu\to Z$ be its normalization and set
	$$
	P_Z:=\mathbb P_{Z^\nu}(\nu_Z^*E),\qquad
	\pi_Z:P_Z\to Z^\nu,\qquad  \overline\xi_Z:=\overline\xi_{E|_{Z^\nu}}.
$$

		For every $i\geq0$, we define $\widehat s_i(\overline E)$ by specifying its
		component on each integral closed subvariety $Z\subseteq X$ with
		$e=\dim Z\geq i-1$.  For integrable adelic line bundles
		$\overline L_1,\ldots,\overline L_{e+1-i}$ on $Z^\nu$, set
			\begin{equation}
			\left\langle \widehat s_i(\overline E),\overline L_1,\ldots,\overline L_{e+1-i}\mid Z\right\rangle
			:=\widehat{\deg}_{P_Z}\!\left(\widehat c_1(\overline\xi_Z)^{r-1+i}
			\prod_{j=1}^{e+1-i}\widehat c_1(\pi_Z^*\overline L_j)\right).
		\label{eq:segre-definition}
		\end{equation}

	Equivalently,
	\[
	\widehat s_i(\overline E)
	:=
	\left[
	(Z;\overline L_1,\ldots,\overline L_{e+1-i})
	\longmapsto
		\widehat{\deg}_{P_Z}\left(
		\widehat c_1(\overline\xi_Z)^{r-1+i}
		\prod_{j=1}^{e+1-i}
		\widehat c_1\bigl(\pi_Z^*\overline L_j\bigr)
		\right)
	\right]
	\in \operatorname{Num}_{\mathrm{ad}}^i(X).
	\]
	It is called the $i$-th numerical tautological Segre class of
	$\overline E$.
\end{definition}

The number of factors in \ref{eq:segre-definition} is
$$
  (r-1+i)+(e+1-i)=e+r=\dim P_Z+1,
$$
Thus the right-hand side is a top adelic intersection number.

\begin{theorem}
\label{thm:segre-basic}
The class in \cref{def:segre} is well defined and has the following
properties.

\begin{enumerate}[label=\textup{(\alph*)}]
\item It depends only on the isometry class of $\overline\xi_E$.
\item It is symmetric and multilinear in the adelic line bundles used in
the pairing.  It is continuous for simultaneously controlled model sequences.
\item Let $\overline\xi_{E,n}$ be model metrics on the fixed line bundle
$\mathcal O_{\Pj(E)}(1)$.  If
$\overline\xi_{E,n}\xrightarrow{\mathrm{ctrl}}\overline\xi_E$, then, for
every fixed $Z$ and every fixed collection of integrable adelic line bundles
$\overline L_1,\ldots,\overline L_{e+1-i}$ on $Z^\nu$,
\[
\left\langle
\widehat s_i(\overline E_n),
\overline L_1,\ldots,\overline L_{e+1-i}\mid Z
\right\rangle
\longrightarrow
\left\langle
\widehat s_i(\overline E),
\overline L_1,\ldots,\overline L_{e+1-i}\mid Z
\right\rangle .
\]
\item If $\overline E$ and the adelic line bundles
$\overline L_1,\ldots,\overline L_{e+1-i}$ used in the pairing are nef, then
  \cref{eq:segre-definition} is nonnegative.
\item Its degree-zero component is the identity:
\begin{equation}
\left\langle\whs_0(\overline E),\overline L_1,\ldots,\overline L_{e+1}\mid Z\right\rangle
=\whdeg_{Z^\nu}\!\left(\prod_{j=1}^{e+1}\whc_1(\overline L_j)\right).
\label{eq:s-zero}
\end{equation}
\end{enumerate}
\end{theorem}

\begin{proof}
	Fix an integral closed subvariety $Z\subseteq X$ of dimension
	$e\geq i-1$, and put
	$$
	Y:=Z^\nu,\qquad
	P_Z:=\mathbb P_Y(\nu_Z^*E),\qquad
	m:=r-1+i,\qquad
	q:=e+1-i.
	$$
	Then
	$$
	m+q=e+r=\dim P_Z+1.
	$$
	Thus the expression in \cref{def:segre} has the correct number of
	factors to define a top adelic intersection number on $P_Z$.
	
	Let
	$$
	\widetilde\nu_Z:P_Z\longrightarrow\mathbb P(E)
	$$
	be the morphism induced by $Y\to Z\hookrightarrow X$. By the
	base-change property of projective bundles and \cref{lem:pullback},
	there is a canonical isometry
	$$
	\overline\xi_Z
	\simeq
	\widetilde\nu_Z^*\overline\xi_E.
	$$
	Thus $\overline\xi_Z$ is integrable. If
	$\overline L_1,\ldots,\overline L_q$ are integrable adelic line
	bundles on $Y$, then each $\pi_Z^*\overline L_j$ is integrable.
	Consequently,
	$$
	F_{Z,i}(\overline L_1,\ldots,\overline L_q)
	:=
	\widehat{\deg}_{P_Z}
	\left(
	\widehat c_1(\overline\xi_Z)^m
	\prod_{j=1}^q
	\widehat c_1(\pi_Z^*\overline L_j)
	\right)
	$$
	is well defined.
	
	The adelic intersection pairing is symmetric and multilinear.  Moreover,
	$$
	\pi_Z^*:
	\widehat{\operatorname{Pic}}(Y)_{\mathrm{int},\mathbb R}
	\longrightarrow
	\widehat{\operatorname{Pic}}(P_Z)_{\mathrm{int},\mathbb R}
	$$
	is linear.  Therefore $F_{Z,i}$ is symmetric and multilinear.
	
	Pullback preserves simultaneous control.  Therefore
	\cref{eq:controlled-intersection-continuity} shows that $F_{Z,i}$ has the
	asserted controlled continuity.  Hence the collection
	$$
	\bigl(F_{Z,i}\bigr)_{
		\substack{Z\subseteq X\ \mathrm{integral}\\
			\dim Z\geq i-1}}
	$$
	defines an element of $\operatorname{Num}_{\mathrm{ad}}^i(X)$.
	
	If $\overline\xi_E'\simeq\overline\xi_E$ is an isometry, then
	$$
	\widetilde\nu_Z^*\overline\xi_E'
	\simeq
	\widetilde\nu_Z^*\overline\xi_E.
	$$
	Since adelic intersection numbers depend only on isometry classes,
	the resulting numerical class is unchanged.  This proves~(a).  The
	preceding multilinearity and controlled-continuity statement prove~(b).
	
	For~(c), let
	$$
	\overline\xi_{E,n}\xrightarrow{\mathrm{ctrl}}\overline\xi_E
	$$
	and set
	$$
	\overline\xi_{Z,n}:=
	\widetilde\nu_Z^*\overline\xi_{E,n}.
	$$
	Pulling back the fixed nef summands in the control data gives
	$$
	\overline\xi_{Z,n}\xrightarrow{\mathrm{ctrl}}\overline\xi_Z.
	$$
	\Cref{eq:controlled-intersection-continuity} then implies
		$$
		\widehat{\deg}_{P_Z}\!\left(\widehat c_1(\overline\xi_{Z,n})^m
		\prod_{j=1}^q\widehat c_1(\pi_Z^*\overline L_j)\right)
		\longrightarrow\widehat{\deg}_{P_Z}\!\left(\widehat c_1(\overline\xi_Z)^m
		\prod_{j=1}^q\widehat c_1(\pi_Z^*\overline L_j)\right).
		$$
	This is precisely the assertion in~(c).
	
	Suppose that $\overline E$ and
	$\overline L_1,\ldots,\overline L_q$ are nef.  Then
	$\overline\xi_E$ is nef by definition. Nefness is preserved by
	restriction and pullback, so $\overline\xi_Z$ and all
	$\pi_Z^*\overline L_j$ are nef. The nonnegativity of top
	intersections of nef adelic line bundles therefore gives
	$$
	F_{Z,i}(\overline L_1,\ldots,\overline L_q)\geq 0,
	$$
	which proves~(d).
	
	It remains to prove~(e).
	Put $i=0$.
	Let $p:Y\to\operatorname{Spec}K$ be the structure morphism.
	The morphisms $p$ and $p\circ\pi_Z$ are projective and flat.
	The schemes $Y$ and $P_Z$ are integral.
	The base $\operatorname{Spec}K$ is normal.
	First take integrable adelic line bundles
	$\overline L_1,\ldots,\overline L_{e+1}\in\widehat{\Pic}(Y)_{\mathrm{int}}$.
	We extend to real coefficients at the end.

	Let $\eta$ be the generic point of $Y$.
	The fibre $(P_Z)_\eta$ is $\mathbb P^{r-1}_{K(Y)}$.
	The restriction of the underlying line bundle $\xi_Z$ to this fibre is
	$\mathcal O_{\mathbb P^{r-1}_{K(Y)}}(1)$.
	Hence its intersection number on the generic fibre is
	$$
	\deg_{(P_Z)_\eta}\!\left(c_1(\xi_Z|_{(P_Z)_\eta})^{r-1}\right)=1.
	$$
	For $r=1$, the morphism $\pi_Z$ is an isomorphism.
	In this case, the displayed number is the degree of the point
	$\operatorname{Spec}K(Y)$ over $K(Y)$.

	Write $\langle\cdots\rangle_{P_Z/K}$ and
	$\langle\cdots\rangle_{Y/K}$ for the adelic Deligne pairings associated
	with $p\circ\pi_Z$ and $p$, respectively.
	Apply \cite[Lemma~4.6.1(3)]{YZ} to these morphisms.
	We use its extension to $\operatorname{Spec}K$ by the direct limit over
	projective and flat models in \cite[Section~4.5]{YZ}.
	Use $r-1$ copies of the original integrable adelic line bundle
	$\overline\xi_Z$ as the relative inputs.
	The generic fibre calculation above gives the coefficient $1$ in that formula.
	We obtain an isometry of adelic line bundles on $\operatorname{Spec}K$:
	$$
	\left\langle
	\underbrace{\overline\xi_Z,\ldots,\overline\xi_Z}_{r-1\ \mathrm{factors}},
	\pi_Z^*\overline L_1,\ldots,\pi_Z^*\overline L_{e+1}
	\right\rangle_{P_Z/K}
	\simeq
	\left\langle\overline L_1,\ldots,\overline L_{e+1}\right\rangle_{Y/K}.
	$$
	The degree of a Deligne pairing equals the corresponding top intersection
	number; see \cite[Lemma~4.4.3(1) and Section~4.5]{YZ}.
	Taking normalized adelic degrees therefore gives
	$$
	\widehat{\deg}_{P_Z}
	\left(
	\widehat c_1(\overline\xi_Z)^{r-1}
	\prod_{j=1}^{e+1}
	\widehat c_1(\pi_Z^*\overline L_j)
	\right)
	=
	\widehat{\deg}_{Y}
	\left(
	\prod_{j=1}^{e+1}
	\widehat c_1(\overline L_j)
	\right).
	$$
	Both sides are multilinear in the adelic line bundles
	$\overline L_1,\ldots,\overline L_{e+1}$ on $Y$.
	Finite real linear expansion therefore gives the same identity for
	$\overline L_j\in\widehat{\Pic}(Y)_{\mathrm{int},\RR}$.
	Since $Y=Z^\nu$, this proves \cref{eq:s-zero} and completes the proof.
\end{proof}


\begin{proposition}
\label{prop:rank-one}
If $\overline E=\overline M$ has rank one, then
\begin{equation}
\left\langle\whs_i(\overline M),\overline L_1,\ldots,\overline L_{e+1-i}\mid Z\right\rangle
=\whdeg_{Z^\nu}\!\left(\whc_1(\overline M|_{Z^\nu})^i\prod_j\whc_1(\overline L_j)\right).
\label{eq:rank-one}
\end{equation}
\end{proposition}

\begin{proof}
	Put
	$$
	Y:=Z^\nu,
	\qquad
	\overline M_Y:=\nu_Z^*\overline M.
	$$
	Since $M_Y$ has rank one, relative Proj gives the canonical isomorphism
	$$
	\pi_Z:\mathbb P_Y(M_Y)\longrightarrow Y
	$$
	
	Under this identification, the canonical surjection
	$$
	\pi_Z^*M_Y\twoheadrightarrow
	\mathcal O_{\mathbb P_Y(M_Y)}(1)
	$$
	is an isomorphism of line bundles.  We claim that it is also an
	isometry.  Indeed, let $v$ be a place of $K$, and let
	$$
	q:(M_Y)_y\longrightarrow
	\mathcal O_{\mathbb P_Y(M_Y)}(1)_y
	$$
	be the induced map at an analytic point $y$.  Since $q$ is an
	isomorphism, the defining formula for the quotient metric gives
	$$
	\|\ell\|_{\xi_Z,v}
	=
	\inf\bigl\{
	\|m\|_{M_Y,v}:q(m)=\ell
	\bigr\}
	=
	\|q^{-1}(\ell)\|_{M_Y,v}.
	$$
	Thus, under the canonical identification
	$\mathbb P_Y(M_Y)\simeq Y$, one has a canonical isometry
	$$
	\overline\xi_Z
	\simeq
	\overline M_Y.
	$$
	
	Since $\operatorname{rk}(M)=1$, the exponent appearing in
	\cref{def:segre} is
	$$
	r-1+i=i.
	$$
	Moreover, under $\pi_Z\simeq\operatorname{id}_Y$, one has
	$$
	\pi_Z^*\overline L_j\simeq\overline L_j.
	$$
	Therefore \cref{def:segre} gives
		$$
		\left\langle\widehat s_i(\overline M),\overline L_1,\ldots,
		\overline L_{e+1-i}\mid Z\right\rangle
		=\widehat{\deg}_{Y}\!\left(\widehat c_1(\overline M_Y)^i
		\prod_{j=1}^{e+1-i}\widehat c_1(\overline L_j)\right).
		$$
	Since $Y=Z^\nu$ and
	$\overline M_Y=\overline M|_{Z^\nu}$, this is exactly
	\cref{eq:rank-one}.
\end{proof}

\begin{proposition}[Finite normalized base change]
\label{prop:field-extension}
Let $K'/K$ be a finite extension, and let $p:X_{K'}\to X$ be the canonical
projection.  Let $Z'\subseteq X_{K'}$ be an integral closed subvariety, put
$Z:=p(Z')$ and $e:=\dim Z$, and let
$$
\rho:(Z')^\nu\longrightarrow Z^\nu
$$
be the induced finite morphism.  For $0\leq i\leq e+1$ and integrable adelic
line bundles $\overline L_1,\ldots,\overline L_{e+1-i}$ on $Z^\nu$, one has
\begin{equation}
\begin{aligned}
&\left\langle
\widehat s_i(p^*\overline E),
\rho^*\overline L_1,\ldots,\rho^*\overline L_{e+1-i}\mid Z'
\right\rangle_{K'}\\
&\qquad=
\frac{\deg(\rho)}{[K':K]}
\left\langle
\widehat s_i(\overline E),
\overline L_1,\ldots,\overline L_{e+1-i}\mid Z
\right\rangle_K.
\end{aligned}
\label{eq:segre-base-change}
\end{equation}
Thus numerical Segre pairings are compatible with finite normalized base
change.
\end{proposition}

\begin{proof}
The base-change property of relative projective bundles gives a canonical
isomorphism
$$
P_{Z'}\simeq P_Z\times_{Z^\nu}(Z')^\nu.
$$
Let $\widetilde\rho:P_{Z'}\to P_Z$ be the induced finite morphism.  It has
degree $\deg(\rho)$ and satisfies
$$
\widetilde\rho^*\overline\xi_Z\simeq\overline\xi_{Z'},
\qquad
\widetilde\rho^*\pi_Z^*\overline L_j
\simeq\pi_{Z'}^*\rho^*\overline L_j.
$$
Applying the finite projection formula to the top intersection defining the
left-hand side of \eqref{eq:segre-base-change}, and using the normalization
of adelic degrees over $K'$, gives the factor
$\deg(\rho)/[K':K]$.  This proves \eqref{eq:segre-base-change}.
\end{proof}

\begin{proposition}[Generically finite projection formula]
\label{prop:genfinite}
Let $f:Y\to X$ be a dominant generically finite morphism of integral
projective $K$-varieties.  Let $V\subseteq Y$ be integral, put
$Z:=\overline{f(V)}$, and suppose that $f|_V:V\to Z$ is generically finite
of degree $\delta_V$.  Denote the induced morphism by
$$
g:V^\nu\longrightarrow Z^\nu,
$$
and put $e:=\dim V=\dim Z$.  For $0\leq i\leq e+1$ and integrable adelic
line bundles $\overline L_1,\ldots,\overline L_{e+1-i}$ on $Z^\nu$, one has
\begin{equation}
\begin{aligned}
&\left\langle
\widehat s_i(f^*\overline E),
g^*\overline L_1,\ldots,g^*\overline L_{e+1-i}\mid V
\right\rangle\\
&\qquad=
\delta_V
\left\langle
\widehat s_i(\overline E),
\overline L_1,\ldots,\overline L_{e+1-i}\mid Z
\right\rangle.
\end{aligned}
\label{eq:segre-projection}
\end{equation}
\end{proposition}

\begin{proof}
The equality $(f|_V)\circ\nu_V=\nu_Z\circ g$ gives a canonical identification
$$
(f^*E)|_V\simeq g^*(E|_Z).
$$
Thus, if $P_V$ is formed from $(f^*E)|_V$ and $P_Z$ is formed from $E|_Z$,
projective-bundle base change gives
$$
P_V\simeq P_Z\times_{Z^\nu}V^\nu.
$$
The morphism $g$ therefore induces a proper generically finite morphism
$$
\widetilde g:P_V\longrightarrow P_Z
$$
of degree $\delta_V$.  Projective-bundle base change and
Lemma~\ref{lem:pullback} give canonical isometries
$$
\widetilde g^*\overline\xi_Z\simeq\overline\xi_V,
\qquad
\widetilde g^*\pi_Z^*\overline L_j
\simeq\pi_V^*g^*\overline L_j.
$$
The adelic projection formula applied to the defining top intersection on
$P_V$ now gives \eqref{eq:segre-projection}.
\end{proof}

\subsubsection{Products of Segre classes}
\label{sec:segre-products}
Let $X$ be a projective integral variety over $K$, and let $\overline E$ be an
integrable adelic vector bundle of rank $r$ on $X$.  Let
$Z\subseteq X$ be an integral closed subvariety of dimension $e$.  Write its
normalization as
$$
\nu_Z:Z^\nu\longrightarrow Z.
$$
Set
$$
P_Z:=\mathbb P_{Z^\nu}(\nu_Z^*E),
\qquad
\pi_Z:P_Z\longrightarrow Z^\nu,
$$
and let $\overline\xi_Z=\overline{\mathcal O}_{P_Z}(1)$ be the tautological
quotient line bundle equipped with the quotient metric induced by
$\nu_Z^*\overline E$.

Let $\lambda$ be a partition of $i=|\lambda|$, choose an ordering
$\lambda=(\lambda_1,\ldots,\lambda_\ell)$ of its positive parts, and assume
that $i\leq e+1$.  If $\ell\geq1$, set
$$
P_{Z,\lambda}
:=
\underbrace{
	P_Z\times_{Z^\nu}\cdots\times_{Z^\nu}P_Z
}_{\ell\ \mathrm{factors}}.
$$
When $\ell\geq1$, denote by $q:P_{Z,\lambda}\to Z^\nu$ the structure
morphism, by $p_a:P_{Z,\lambda}\to P_Z$ the $a$-th projection, and put
$\overline\xi_a:=p_a^*\overline\xi_Z$.
For the empty partition, set
$$
P_{Z,\varnothing}:=Z^\nu,
\qquad
q:=\operatorname{id}_{Z^\nu};
$$
there are then no maps $p_a$ and no line bundles $\overline\xi_a$.

\begin{definition}[Segre monomial]
\label{def:segre-monomial}
For every collection of integrable adelic line bundles
$$
\overline L_1,\ldots,\overline L_{e+1-i}
\in
\widehat{\operatorname{Pic}}(Z^\nu)_{\mathrm{int},\mathbb R},
$$
define
\begin{equation}
	\label{eq:segre-monomial}
	\begin{aligned}
		\left\langle
		\prod_{a=1}^{\ell}\whs_{\lambda_a}(\overline E),
		\overline L_1,\ldots,\overline L_{e+1-i}\mid Z
		\right\rangle
		&:=
		\whdeg_{P_{Z,\lambda}}\!\left(
		\prod_{a=1}^{\ell}
		\whc_1(\overline\xi_a)^{r-1+\lambda_a}
		\prod_{j=1}^{e+1-i}\whc_1(q^*\overline L_j)
		\right).
	\end{aligned}
\end{equation}

For the fixed subvariety $Z$, Equation~\eqref{eq:segre-monomial} defines the
numerical pairing of the Segre monomial
$\prod_{a=1}^{\ell}\widehat s_{\lambda_a}(\overline E)$ with the adelic
line bundles $\overline L_1,\ldots,\overline L_{e+1-i}$ on $Z^\nu$.

For the empty partition, the right-hand side is
$$
\widehat{\deg}_{Z^\nu}
\left(
\prod_{j=1}^{e+1}\widehat c_1(\overline L_j)
\right).
$$
\end{definition}

The right-hand side contains
$$
  \ell(r-1)+i+e+1-i
  =\dim P_{Z,\lambda}+1
$$
first Chern classes.  Thus it is a top intersection number.

\begin{lemma}
\label{lem:fiber-product}
With the notation of Definition~\ref{def:segre-monomial}, the number in
\eqref{eq:segre-monomial} is invariant under every permutation of
$\lambda_1,\ldots,\lambda_\ell$.
\end{lemma}

\begin{proof}
If $\ell=0$, the assertion is vacuous.  Let $\sigma\in S_\ell$.
Permuting the factors gives a canonical $Z^\nu$-automorphism
$$
\tau_\sigma:P_{Z,\lambda}\longrightarrow P_{Z,\lambda}.
$$
It satisfies $q\circ\tau_\sigma=q$ and permutes the metrized line bundles
$\overline\xi_1,\ldots,\overline\xi_\ell$.  Pullback by $\tau_\sigma$
therefore identifies the top adelic intersection in
\eqref{eq:segre-monomial} with the one obtained by applying the same
permutation to $\lambda_1,\ldots,\lambda_\ell$.  Invariance of adelic
intersection numbers under isomorphisms proves the assertion.
\end{proof}

\begin{proposition}
\label{prop:segre-monomial-properties}
With the notation of Definition~\ref{def:segre-monomial}, the following
properties hold.

\begin{enumerate}[label=\textup{(\roman*)},leftmargin=2.2em]
\item The pairing in \eqref{eq:segre-monomial} is symmetric and multilinear
in $\overline L_1,\ldots,\overline L_{e+1-i}$.  Let
$\overline\xi_{Z,m}\to\overline\xi_Z$ and
$\overline L_{j,m}\to\overline L_j$ be model approximations.  If the
sequences
$$
p_a^*\overline\xi_{Z,m}\quad(1\leq a\leq\ell),
\qquad
q^*\overline L_{j,m}\quad(1\leq j\leq e+1-i)
$$
form a simultaneously controlled family on $P_{Z,\lambda}$, then the
corresponding model intersection numbers converge to
\eqref{eq:segre-monomial}.

\item Let $K'/K$ be finite, let $p:X_{K'}\to X$, and let
$Z'\subseteq X_{K'}$ be integral with $p(Z')=Z$.  If
$$
\rho:(Z')^\nu\longrightarrow Z^\nu
$$
is the induced finite morphism, then
\begin{equation}
\begin{aligned}
&\left\langle
\prod_{a=1}^{\ell}\widehat s_{\lambda_a}(p^*\overline E),
\rho^*\overline L_1,\ldots,\rho^*\overline L_{e+1-i}\mid Z'
\right\rangle_{K'}\\
&\qquad=
\frac{\deg(\rho)}{[K':K]}
\left\langle
\prod_{a=1}^{\ell}\widehat s_{\lambda_a}(\overline E),
\overline L_1,\ldots,\overline L_{e+1-i}\mid Z
\right\rangle_K.
\end{aligned}
\label{eq:segre-monomial-base-change}
\end{equation}

\item Let $f:Y\to X$ be a morphism of integral projective $K$-varieties.
Let $V\subseteq Y$ be integral with $\overline{f(V)}=Z$, and suppose that
$f|_V:V\to Z$ is generically finite of degree $\delta_V$.  If
$$
g:V^\nu\longrightarrow Z^\nu
$$
is the induced morphism, then
\begin{equation}
\begin{aligned}
&\left\langle
\prod_{a=1}^{\ell}\widehat s_{\lambda_a}(f^*\overline E),
g^*\overline L_1,\ldots,g^*\overline L_{e+1-i}\mid V
\right\rangle\\
&\qquad=
\delta_V
\left\langle
\prod_{a=1}^{\ell}\widehat s_{\lambda_a}(\overline E),
\overline L_1,\ldots,\overline L_{e+1-i}\mid Z
\right\rangle.
\end{aligned}
\label{eq:segre-monomial-projection}
\end{equation}

\item If $\overline\xi_Z$ and all $\overline L_j$ are nef, then the number
in \eqref{eq:segre-monomial} is nonnegative.
\end{enumerate}

Consequently, as $Z$ varies over the integral closed subvarieties of $X$ with
$\dim Z\geq i-1$, the pairings in \eqref{eq:segre-monomial} define an
element of $\Num^i\ad(X)$.
\end{proposition}

\begin{proof}
Pullback preserves integrability, so every factor in
\eqref{eq:segre-monomial} is integrable.  The degree count after
\cref{def:segre-monomial} shows that this is a top intersection number.
Symmetry and multilinearity of adelic intersections give the corresponding
properties in~\textup{(i)}.

In the model intersections, the factor $p_a^*\overline\xi_{Z,m}$ occurs
$r-1+\lambda_a$ times.  Each factor $q^*\overline L_{j,m}$ occurs once.
Finite repetition preserves the simultaneous control data.  Hence
\cref{eq:controlled-intersection-continuity} gives the convergence
in~\textup{(i)}.

For finite base change, relative projective bundles give a canonical
isomorphism
$$
P_{Z',\lambda}
\simeq
P_{Z,\lambda}\times_{Z^\nu}(Z')^\nu.
$$
Let $\widetilde\rho_\lambda:P_{Z',\lambda}\to P_{Z,\lambda}$ be the induced
map, and let $q':P_{Z',\lambda}\to(Z')^\nu$ be the structure morphism.  The
map $\widetilde\rho_\lambda$ is finite of degree $\deg(\rho)$.
By \cref{lem:pullback}, it pulls back each metrized tautological line bundle
on $P_{Z,\lambda}$ isometrically to the corresponding line bundle on
$P_{Z',\lambda}$.  For each adelic line bundle $\overline L_j$ on $Z^\nu$,
functoriality of pullback gives
$$
\widetilde\rho_\lambda^*q^*\overline L_j
\simeq
(q')^*\rho^*\overline L_j.
$$
The finite projection formula contributes the factor $\deg(\rho)$.
The normalization of degrees over $K'$ contributes $1/[K':K]$.
This proves~\textup{(ii)}.

For the generically finite formula, the canonical identification
$(f^*E)|_V\simeq g^*(E|_Z)$ and projective-bundle base change give
$$
P_{V,\lambda}
\simeq
P_{Z,\lambda}\times_{Z^\nu}V^\nu.
$$
The induced morphism
$\widetilde g_\lambda:P_{V,\lambda}\to P_{Z,\lambda}$ is generically finite
of degree $\delta_V$.  By \cref{lem:pullback}, it identifies the metrized
tautological line bundles by pullback.  Functoriality identifies the pullback
of $q^*\overline L_j$ along $\widetilde g_\lambda$ with the pullback of
$g^*\overline L_j$ to $P_{V,\lambda}$.  The adelic projection formula therefore gives the
factor $\delta_V$, which proves~\textup{(iii)}.

If the data are nef, their pullbacks are nef.  Nonnegativity of top adelic
intersections proves~\textup{(iv)}.  Finally, part~\textup{(i)} gives the
properties required in the definition of $\Num^i\ad(X)$.
\end{proof}

\begin{proposition}[Twisting formula]
\label{prop:twist-formula}
Let $\overline M$ be an integrable adelic line bundle on $X$.  Let
$Z\subseteq X$ be integral of dimension $e$, let $0\leq i\leq e+1$, and let
$\overline L_1,\ldots,\overline L_{e+1-i}$ be integrable adelic line bundles
on $Z^\nu$.  Then
\begin{equation}
\begin{aligned}
&\left\langle
\widehat s_i(\overline E\otimes\overline M),
\overline L_1,\ldots,\overline L_{e+1-i}\mid Z
\right\rangle\\
&\qquad=
\sum_{k=0}^{i}\binom{r-1+i}{i-k}
\left\langle
\widehat s_k(\overline E),
\underbrace{\overline M|_{Z^\nu},\ldots,\overline M|_{Z^\nu}}_{i-k\ \mathrm{times}},
\overline L_1,\ldots,\overline L_{e+1-i}\mid Z
\right\rangle.
\end{aligned}
\label{eq:twist-pairing}
\end{equation}
The repeated list is empty when $k=i$.  Equivalently, with every product on
the right interpreted by this pairing, we write
\begin{equation}
  \whs_i(\overline E\otimes\overline M)
  =
  \sum_{k=0}^{i}\binom{r-1+i}{i-k}
  \whc_1(\overline M)^{i-k}\whs_k(\overline E)
  \label{eq:twist-formula}
\end{equation}
\end{proposition}

\begin{proof}
Put $Y:=Z^\nu$ and $\overline M_Z:=\overline M|_Y$.  Under the canonical
identification
$$
\mathbb P_{Z^\nu}\bigl((E\otimes M)|_{Z^\nu}\bigr)
\simeq P_Z,
$$
Equation~\eqref{eq:twisting-taut} gives the isometry
$$
\overline\xi_{(E\otimes M)|_{Z^\nu}}
\simeq
\overline\xi_Z\otimes\pi_Z^*\overline M_Z.
$$

First assume $\overline L_j\in\widehat{\Pic}(Y)_{\intg}$ for every $j$.
For $0\leq t\leq r-1+i$, set
$$
J_t:=
\widehat{\deg}_{P_Z}\left(
\widehat c_1(\overline\xi_Z)^t
\widehat c_1(\pi_Z^*\overline M_Z)^{r-1+i-t}
\prod_{j=1}^{e+1-i}\widehat c_1(\pi_Z^*\overline L_j)
\right).
$$
Suppose $t<r-1$.
Apply \cite[Lemma~4.6.1(3), pp.~128--129]{YZ} to
$P_Z\xrightarrow{\pi_Z}Y\to\operatorname{Spec}K$.
We use its arithmetic version over $\operatorname{Spec}K$, obtained by
the direct limit over projective and flat models in \cite[Section~4.5]{YZ}.
The schemes $P_Z$ and $Y$ are integral and projective over $K$.
Their structure morphisms to $\operatorname{Spec}K$ are flat, and
$\operatorname{Spec}K$ is normal.
Take $t$ copies of the original integrable adelic line bundle
$\overline\xi_Z$ and $r-1-t$ copies of $\pi_Z^*\overline M_Z$ as the
$r-1$ relative inputs.
The remaining inputs are pulled back from $i$ copies of $\overline M_Z$
and the adelic line bundles $\overline L_1,\ldots,\overline L_{e+1-i}$ on $Y$.
Thus there are exactly $e+1$ inputs from $Y$.

Let $\eta$ be the generic point of $Y$.
The line bundle $\pi_Z^*M_Z$ is trivial on $(P_Z)_\eta$.
At least one relative input has this underlying line bundle.
Hence the intersection number of the relative inputs on $(P_Z)_\eta$ is zero.
The Deligne projection formula therefore identifies the Deligne pairing
of all the factors defining $J_t$ with the trivial metrized line bundle
on $\operatorname{Spec}K$.
Its section $1$ has norm $1$ at every place.
Taking normalized adelic degrees gives $J_t=0$.
Here we use the compatibility between the degree of a Deligne pairing and
the top intersection number; see
\cite[Lemma~4.4.3(1) and Section~4.5]{YZ}.

If $t=r-1+k$ with $0\leq k\leq i$, then the definition of the Segre
pairing gives
$$
J_{r-1+k}
=\left\langle\widehat s_k(\overline E),
\underbrace{\overline M_Z,\ldots,\overline M_Z}_{i-k\ \mathrm{times}},
\overline L_1,\ldots,\overline L_{e+1-i}\mid Z\right\rangle.
$$
Expand the tautological tensor product using multilinearity.
The term $J_t$ has coefficient $\binom{r-1+i}{t}$.
The terms with $t<r-1$ vanish, and
$$
\binom{r-1+i}{r-1+k}=\binom{r-1+i}{i-k}.
$$
This proves \cref{eq:twist-pairing}.
For $r=1$, there are no terms with $t<r-1$, so the same expansion applies.
Finite real linear expansion in $\overline L_1,\ldots,\overline L_{e+1-i}$
gives the formula for real coefficients.
Equation~\eqref{eq:twist-formula} is the notation for this pairing identity.
\end{proof}

\subsubsection{Universal characteristic polynomials}

Let
$$
\mathcal U:=\QQ[S_1,S_2,\ldots],
\qquad
\deg S_j=j,
$$
and denote by $\mathcal U_i$ its homogeneous component of weighted degree
$i$.  Here $S_i$ is a formal variable of weighted degree $i$ and represents the
$i$-th Segre class.
For a partition
$$
\lambda=(\lambda_1,\ldots,\lambda_\ell),
\qquad
\lambda_1\geq\cdots\geq\lambda_\ell\geq1,
$$
write
$$
|\lambda|:=\sum_{a=1}^{\ell}\lambda_a,
\qquad
S_\lambda:=\prod_{a=1}^{\ell}S_{\lambda_a}.
$$
Then the monomials $S_\lambda$, with $\lambda$ ranging over the partitions of
$i$, form a $\QQ$-basis of $\mathcal U_i$.  Consequently, every
$P\in\mathcal U_i$ admits a unique expression
$$
P=\sum_{\lambda\vdash i}a_\lambda S_\lambda,
\qquad
a_\lambda\in\QQ.
$$

Let $\overline E$ be an integrable adelic vector bundle on $X$, let
$Z\subseteq X$ be an integral closed subvariety of dimension $e$, and assume
$0\leq i\leq e+1$.  For integrable adelic line bundles
$$
\overline L_1,\ldots,\overline L_{e+1-i}
\in
\widehat{\operatorname{Pic}}(Z^\nu)_{\mathrm{int},\mathbb R},
$$
define
\begin{equation}
\left\langle P\bigl(\whs(\overline E)\bigr),\overline L_1,\ldots,
\overline L_{e+1-i}\mid Z\right\rangle
:=\sum_{\lambda\vdash i}a_\lambda\left\langle
\prod_{a=1}^{\ell(\lambda)}\whs_{\lambda_a}(\overline E),
\overline L_1,\ldots,\overline L_{e+1-i}\mid Z\right\rangle.
\label{eq:polynomial-evaluation}
\end{equation}
For each partition $\lambda\vdash i$, the pairing in the corresponding
summand is the top adelic intersection number on $P_{Z,\lambda}$ defined by
\eqref{eq:segre-monomial}.

\begin{proposition}
\label{prop:universal-polynomial-evaluation}
The evaluation in Equation~\eqref{eq:polynomial-evaluation} is well defined
and $\QQ$-linear in $P$.  For fixed $P$, it is symmetric and multilinear in
$\overline L_1,\ldots,\overline L_{e+1-i}$.

Let $\overline\xi_{Z,m}\to\overline\xi_Z$ and
$\overline L_{j,m}\to\overline L_j$ be model approximations.  Assume that,
for every partition $\lambda$ with $a_\lambda\neq0$, the sequences
$$
p_a^*\overline\xi_{Z,m}\quad(1\leq a\leq\ell(\lambda)),
\qquad
q^*\overline L_{j,m}\quad(1\leq j\leq e+1-i)
$$
form a simultaneously controlled family on $P_{Z,\lambda}$.  Then the model
evaluations of $P$ converge to \eqref{eq:polynomial-evaluation}.  The
evaluation also satisfies the finite-base-change formula
\eqref{eq:segre-monomial-base-change} and the generically finite projection
formula \eqref{eq:segre-monomial-projection}, with every Segre-monomial
expression replaced by the corresponding evaluation of $P$.  Consequently,
as $Z$ varies over the integral closed subvarieties with $\dim Z\geq i-1$,
these pairings define an element of $\Num^i\ad(X)$.
\end{proposition}

\begin{proof}
The coefficients $a_\lambda$ in \cref{eq:polynomial-evaluation} are unique
because the monomials $S_\lambda$ form a basis of $\mathcal U_i$.
By \cref{lem:fiber-product}, each monomial pairing is independent of the
ordering of the parts of $\lambda$.
For $i=0$, use the empty partition in \cref{def:segre-monomial}.
Thus the evaluation is well defined and $\QQ$-linear in $P$.
\Cref{prop:segre-monomial-properties}\textup{(i)} gives symmetry and
multilinearity for each summand, and hence for their sum.

For each partition $\lambda\vdash i$, let $I_{\lambda,m}$ denote the
model intersection on $P_{Z,\lambda}$ and let $I_\lambda$ denote the adelic
intersection.  For $a_\lambda\neq0$, the control hypothesis and
\cref{prop:segre-monomial-properties}\textup{(i)} give
$I_{\lambda,m}\to I_\lambda$.
There are finitely many such partitions, so
$$
\lim_{m\to\infty}\sum_{\lambda\vdash i}a_\lambda I_{\lambda,m}=\sum_{\lambda\vdash i}a_\lambda I_\lambda.
$$
This proves the asserted continuity.

\Cref{prop:segre-monomial-properties}\textup{(ii)--(iii)} gives the two
projection formulas for each summand.
Their respective factors $\deg(\rho)/[K':K]$ and $\delta_V$ do not depend
on $\lambda$.  Summing with coefficients $a_\lambda$ proves both formulas
for $P$.  The symmetry, multilinearity, and continuity just proved give an
element of $\Num^i\ad(X)$.
\end{proof}

The adjective \emph{universal} means that the coefficients of $P$ are
independent of $K$, $X$, and $\overline E$.  Thus the same polynomial defines
a numerical characteristic class for every integrable adelic vector bundle.
For example,
$$
\mathcal U_1=\QQ S_1,\qquad
\mathcal U_2=\QQ S_1^2\oplus\QQ S_2,
$$
and
$$
\mathcal U_3
=
\QQ S_1^3\oplus\QQ S_1S_2\oplus\QQ S_3.
$$
We evaluate the universal polynomials defining the Chern classes, Chern
character, and Todd class in this sense.

\subsubsection{Chern classes}

For the projective bundle defined above, let the universal polynomials $C_j$
be determined by
\begin{equation}
  (1+C_1t+C_2t^2+\cdots)
  (1-S_1t+S_2t^2-S_3t^3+\cdots)=1.
  \label{eq:chern-segre}
\end{equation}
Equivalently,
$$
  C_m=\sum_{j=1}^m(-1)^{j-1}C_{m-j}S_j,
  \qquad C_0=1.
$$

\begin{definition}[Tautological Chern classes]
If $\overline E$ has rank $r$, define, for $0\leq i\leq r$,
\begin{equation}
  \whc_i(\overline E)
  :=
  C_i\bigl(\whs_1(\overline E),\ldots,
           \whs_i(\overline E)\bigr)
  \in\Num^i\ad(X).
  \label{eq:taut-chern}
\end{equation}
Set $\whc_0=1$ and $\whc_i=0$ for $i>r$.
\end{definition}

For indices not exceeding the rank $r$, the first formulas are
\begin{equation}
\whc_1=\whs_1,\qquad \whc_2=(\whs_1)^2-\whs_2,\qquad
\whc_3=(\whs_1)^3-2\whs_1\whs_2+\whs_3.
\label{eq:first-chern}
\end{equation}
Here products mean the fiber-product evaluation of
\cref{def:segre-monomial}.

\begin{remark}[Sign convention]
For a line bundle $M$, one has
$\whs_i(\overline M)=\whc_1(\overline M)^i$.  Hence the correct
formal identity is
$$
  c_t(E)s_{-t}(E)=1,
$$
not $c_t(E)s_t(E)=1$.
\end{remark}

\subsubsection{Chern character and Todd class}

Let $x_1,\ldots,x_r$ be formal Chern roots.  Define
$$
  \ch_m(c_1,\ldots,c_m)=\frac{1}{m!}\sum_{a=1}^r x_a^m,
  \qquad \ch_0=r,
$$
and define the Todd polynomials by
$$
  \prod_{a=1}^r\frac{x_a}{1-e^{-x_a}}
  =\sum_{m\geq0}\Td_m(c_1,\ldots,c_m).
$$

\begin{definition}
Set
\begin{equation}
\whch_m(\overline E):=\ch_m(\whc_1,\ldots,\whc_m),\qquad
\whTd_m(\overline E):=\Td_m(\whc_1,\ldots,\whc_m).
\label{eq:ch-td}
\end{equation}
where $\whc_j=0$ for $j>r$ and all monomials are evaluated through
\cref{eq:polynomial-evaluation}.  Only $m\leq d+1$ contributes to numerical
pairings on $X$.
\end{definition}

For ranks large enough for the displayed indices,
\begin{align}
 \whch_0&=r,\notag\\
 \whch_1&=\whs_1,\notag\\
 \whch_2&=\whs_2
 -\frac12(\whs_1)^2,\notag\\
 \whch_3&=\frac16\left(
 (\whs_1)^3
 -3\whs_1\whs_2
 +3\whs_3\right),
 \label{eq:first-ch}
\end{align}
and
\begin{align}
 \whTd_0&=1,\notag\\
 \whTd_1&=\frac12\whs_1,\notag\\
 \whTd_2&=\frac1{12}
 \left(2(\whs_1)^2-\whs_2\right),\notag\\
 \whTd_3&=\frac1{24}
 \left((\whs_1)^3
 -\whs_1\whs_2\right).
 \label{eq:first-td}
\end{align}

\begin{theorem}[Numerical adelic characteristic classes]
\label{thm:main}
Let $K$ be a number field.
Let $X$ be a projective integral $K$-variety.
Let $\overline E$ be an integrable adelic vector bundle on $X$.  All numerical pairings
associated with
$$
  \whs_i(\overline E),\qquad
  \whc_i(\overline E),\qquad
  \whch_i(\overline E),\qquad
  \whTd_i(\overline E)
$$
are well defined.  They depend only on the induced tautological adelic
metric.  They are continuous along model approximations whose pullbacks to
every fiber product $P_{Z,\lambda}$ occurring in the defining finite
expansion form a simultaneously controlled family.  They satisfy the finite
normalized base-change formula in \cref{eq:segre-monomial-base-change} and the
generically finite projection formula.  Segre monomial numbers are
nonnegative for nef data.
\end{theorem}

\begin{proof}
Put $r:=\operatorname{rank}E$.
The recursion in \cref{eq:chern-segre} gives $C_j\in\mathcal U_j$.
Use $C_j$ for $j\leq r$ and zero for $j>r$, as in
\cref{eq:taut-chern}.
Substituting these expressions into the Chern character and Todd
polynomials in \cref{eq:ch-td} gives an element of $\mathcal U_i$ in each
degree $i$.  The degree zero polynomials are $r$ and $1$, respectively.
\Cref{prop:universal-polynomial-evaluation} therefore proves that all the
stated characteristic pairings are well defined.
For the Segre classes of positive degree, use $S_i$.
\Cref{thm:segre-basic}\textup{(e)} identifies $\whs_0$ with the unit
defined by the empty partition.

Fix an integral closed subvariety $Z\subseteq X$.
Let $\iota_Z:P_Z\to\mathbb P_X(E)$ be the natural morphism.
Projective bundle base change and \cref{lem:pullback} give
$$
P_Z\simeq\mathbb P_X(E)\times_X Z^\nu,\qquad \overline\xi_Z\simeq\iota_Z^*\overline\xi_E.
$$
For every nonempty partition $\lambda$, the tautological input on its
$a$-th factor is therefore
$$
\overline\xi_a\simeq p_a^*\iota_Z^*\overline\xi_E.
$$
Thus, for fixed adelic line bundles $\overline L_j$ in the pairing,
\cref{eq:segre-monomial} depends only on the metrized line
bundle $\overline\xi_E$.
The same holds for its finite linear combinations in
\cref{eq:polynomial-evaluation}.

The theorem assumes simultaneous control on every $P_{Z,\lambda}$ that
occurs in each defining polynomial.
These are precisely the continuity hypotheses of
\cref{prop:universal-polynomial-evaluation}.
That proposition gives the asserted convergence and both projection
formulas.  In particular, the factors are $\deg(\rho)/[K':K]$ in
\cref{eq:segre-monomial-base-change} and $\delta_V$ in
\cref{eq:segre-monomial-projection}.
Finally, \cref{prop:segre-monomial-properties}\textup{(iv)} gives
nonnegativity for Segre monomial pairings with nef inputs.
\end{proof}

\begin{remark}
	Nef inputs give nonnegative Segre monomial numbers.
	The universal expressions for the Chern classes, the Chern character,
	and the Todd class in the Segre classes can have coefficients of both signs.
\end{remark}

\subsubsection{Approximation by models}

\begin{proposition}[Model independence]
	\label{prop:model-independence}
	Retain the notation of
	Definition~\ref{def:segre-monomial}.  Let
	$\overline\xi_{E,n}$ be integrable adelic line bundles on $\Pj(E)$ induced by
	model line bundles, and, for $1\leq j\leq e+1-i$, let
	$\overline L_{j,n}$ be integrable adelic line bundles on $Z^\nu$ induced by
	model line bundles.  Let
	$$
	\iota_Z:P_Z\longrightarrow\Pj(E)
	$$
	be the natural morphism and set
	$$
	\overline\xi_{Z,n}:=\iota_Z^*\overline\xi_{E,n},
	\qquad
	\overline\xi_{a,n}
	:=p_a^*\overline\xi_{Z,n}.
	$$
	Assume that
	$$
	\overline\xi_{Z,n}\longrightarrow\overline\xi_Z,
	\qquad
	\overline L_{j,n}\longrightarrow\overline L_j,
	$$
	and that the sequences
	$$
	(\overline\xi_{a,n})_{n\geq1}\quad(1\leq a\leq\ell),
	\qquad
	(q^*\overline L_{j,n})_{n\geq1}\quad(1\leq j\leq e+1-i)
	$$
	form a simultaneously controlled family on $P_{Z,\lambda}$, with respective
	limits $\overline\xi_a$ and $q^*\overline L_j$.
	Then
	\begin{align*}
		&\lim_{n\to\infty}
		\whdeg_{P_{Z,\lambda}}\left(
		\prod_{a=1}^{\ell}
		\whc_1(\overline\xi_{a,n})^{r-1+\lambda_a}
		\prod_{j=1}^{e+1-i}
		\whc_1(q^*\overline L_{j,n})
		\right)                                                     \\
		&\qquad=
		\left\langle
		\prod_{a=1}^{\ell}\whs_{\lambda_a}(\overline E),
		\overline L_1,\ldots,\overline L_{e+1-i}\mid Z
		\right\rangle .
	\end{align*}
	For each $n$, the intersection number on the left may be computed on any
	projective model on which the finitely many model line bundles involved are
	simultaneously represented.  Its value, and hence its limit, is independent
	of all such choices.
\end{proposition}

\begin{proof}
	Put
	$$
	P:=P_{Z,\lambda}.
	$$
	Since
	$$
	\dim P=e+\ell(r-1),
	$$
	the number of first Chern class factors occurring in the intersection is
	$$
	\sum_{a=1}^{\ell}(r-1+\lambda_a)+(e+1-i)
	=
	\ell(r-1)+i+e+1-i
	=
	\dim P+1.
	$$
	Thus the expression is a top adelic intersection
	number on $P$.
	
	We first verify that the model intersection is well defined.  Fix $n$ and
	choose projective models representing $\overline\xi_{E,n}$ and the
	$\overline L_{j,n}$.  Starting from
	any projective model of $P$, take the normalization of the closure of the
	graph of the morphisms
	$$
		\iota_Z\circ p_a:P\longrightarrow\Pj(E),
		\qquad
		q:P\longrightarrow Z^\nu,
	$$
	with respect to these finitely many models.  This gives a projective model
	$\mathcal P_n$ of $P$ that carries all these pullbacks.
	
	Suppose that $\mathcal P_n'$ is another such model.  There exists a
	projective model $\mathcal P_n''$ dominating both $\mathcal P_n$ and
	$\mathcal P_n'$.  The two morphisms from $\mathcal P_n''$ to each of the
	models carrying the original line bundles agree on the generic fibre and
	therefore agree everywhere, since the target models are separated.
	Consequently, the corresponding model line bundles have isometric
	pullbacks to $\mathcal P_n''$.  The projection formula for the proper
	birational morphisms
	$$
	\mathcal P_n''\longrightarrow\mathcal P_n,
	\qquad
	\mathcal P_n''\longrightarrow\mathcal P_n'
	$$
	shows that the two model intersection numbers coincide; see
	\cite[Proposition~4.1.2]{YZ}.
	
	By hypothesis, the sequences $p_a^*\overline\xi_{Z,n}$ and
	$q^*\overline L_{j,n}$ form a simultaneously controlled family on
	$P_{Z,\lambda}$.  Proposition~\ref{prop:segre-monomial-properties}
	therefore identifies the limit of the model intersection numbers with
	$$
	\left\langle
	\prod_{a=1}^{\ell}\whs_{\lambda_a}(\overline E),
	\overline L_1,\ldots,\overline L_{e+1-i}\mid Z
	\right\rangle .
	$$
	This proves the assertion.
\end{proof}

\begin{corollary}
\label{cor:model-polynomial-convergence}
Fix an integral closed subvariety $Z\subseteq X$ of dimension $e$, an
integer $0\leq i\leq e+1$, and let
$$
P=\sum_{\lambda\vdash i}a_\lambda S_\lambda\in\mathcal U_i.
$$
Let $\overline\xi_{E,n}\to\overline\xi_E$ be model approximations induced
by vector-bundle models, and let
$\overline L_{j,n}\to\overline L_j$ $(1\leq j\leq e+1-i)$ be model
approximations of integrable adelic line bundles on $Z^\nu$.
Assume that, for every $\lambda$ with $a_\lambda\neq0$, these approximations
satisfy the hypotheses of Proposition~\ref{prop:model-independence} on
$P_{Z,\lambda}$.  Denote the model intersection in that proposition by
$I_{\lambda,n}$.  Then
$$
\lim_{n\to\infty}
\sum_{\substack{\lambda\vdash i\\a_\lambda\neq0}}
a_\lambda I_{\lambda,n}
=
\left\langle
P\bigl(\widehat s(\overline E)\bigr),
\overline L_1,\ldots,\overline L_{e+1-i}\mid Z
\right\rangle.
$$
\end{corollary}

\begin{proof}
Proposition~\ref{prop:model-independence} makes each $I_{\lambda,n}$
independent of the model used to compute it and gives
$$
I_{\lambda,n}\longrightarrow
I_\lambda:=\left\langle
\prod_{a=1}^{\ell(\lambda)}\whs_{\lambda_a}(\overline E),
\overline L_1,\ldots,\overline L_{e+1-i}\mid Z
\right\rangle
$$
for every $\lambda$ with $a_\lambda\neq0$.  Since there are only finitely
many such partitions,
$$
\begin{aligned}
\lim_{n\to\infty}
\sum_{\substack{\lambda\vdash i\\a_\lambda\neq0}}
a_\lambda I_{\lambda,n}
&=\sum_{\substack{\lambda\vdash i\\a_\lambda\neq0}}
a_\lambda I_\lambda\\
&=\left\langle
P\bigl(\widehat s(\overline E)\bigr),
\overline L_1,\ldots,\overline L_{e+1-i}\mid Z
\right\rangle,
\end{aligned}
$$
by Equation~\eqref{eq:polynomial-evaluation}.
\end{proof}


\subsection{The Bott--Chern boundary}
\label{sec:bott-chern}

The classes constructed above are numerical adelic functionals.
Let $\mathcal X$ be a regular projective arithmetic variety and let
$\overline{\mathcal E}$ be a smooth hermitian vector bundle of rank $r$.
Put
$$
  \widehat\pi\colon\Pj(\mathcal E)\longrightarrow\mathcal X,
  \qquad
  \widehat\xi=\whc_1(\overline{\mathcal O}(1)).
$$
For a Hermitian vector bundle $\overline{\mathcal F}$ on $\mathcal X$, let
$\widehat c_i^{\mathrm{GS}}(\overline{\mathcal F})$ be the arithmetic Chern
class of degree $i$ defined by Gillet and Soul\'e.  Set
$\widehat c_0^{\mathrm{GS}}(\overline{\mathcal F})=1$.  Define
\begin{equation*}
 \widehat c_t^{\mathrm{GS}}(\overline{\mathcal F})
 :=\sum_{i=0}^{\operatorname{rk}\mathcal F}
 \widehat c_i^{\mathrm{GS}}(\overline{\mathcal F})t^i.
\end{equation*}
We equip $\mathcal E^{\vee}$ with the dual metric.  Arithmetic Chern classes
satisfy
$$
\widehat c_i^{\mathrm{GS}}(\overline{\mathcal E}^{\vee})=(-1)^i\widehat c_i^{\mathrm{GS}}(\overline{\mathcal E}),\qquad \widehat c_t^{\mathrm{GS}}(\overline{\mathcal E}^{\vee})=\widehat c_{-t}^{\mathrm{GS}}(\overline{\mathcal E});
$$
see \cite{GSCharI,GSCharII}.  We use the same sign convention for arithmetic
Segre classes as for the numerical classes.  Thus we define
$\widehat s_t^{\mathrm{GS}}(\overline{\mathcal E})
:=\sum_{i\geq0}\widehat s_i^{\mathrm{GS}}(\overline{\mathcal E})t^i$ by
$$
\widehat c_t^{\mathrm{GS}}(\overline{\mathcal E})\widehat s_{-t}^{\mathrm{GS}}(\overline{\mathcal E})=1.
$$
The proposition below computes
$$
\widehat s_i^{\mathrm{GS}}(\overline{\mathcal E})
-\widehat\pi_*(\widehat\xi^{r-1+i})
\quad\text{in }\widehat{\mathrm{CH}}^i(\mathcal X)_{\QQ}.
$$
The hermitian relative Euler sequence is generally not orthogonally split,
and its Bott--Chern class contributes a secondary term.
Let $F_\infty$ denote complex conjugation on $\mathcal X(\mathbb C)$.
For $p\geq0$, let $A^{p,p}(\mathcal X)$ be the space of smooth real
differential forms $\eta$ of type $(p,p)$ on $\mathcal X(\mathbb C)$ satisfying
$$
F_\infty^*\eta=(-1)^p\eta.
$$
We use this convention throughout; see \cite[Section~3.2.1]{GS}.  Define
$$
\widetilde A^{p,p}(\mathcal X):=A^{p,p}(\mathcal X)\big/\bigl(A^{p,p}(\mathcal X)\cap(\operatorname{im}\partial+\operatorname{im}\overline\partial)\bigr).
$$
We take the two images among smooth complex differential forms of bidegree
$(p,p)$.  For $i\geq1$ and
$\eta\in\widetilde A^{i-1,i-1}(\mathcal X)$, let
$a(\eta)\in\widehat{\mathrm{CH}}^i(\mathcal X)_{\QQ}$ be the class represented
by $(0,\eta)$.

\begin{proposition}
\label{prop:bott-chern}
For every $i\geq1$, there is a canonical class
$\eta_i(\overline{\mathcal E})\in\widetilde A^{i-1,i-1}(\mathcal X)$.  It is
functorial under smooth pullback and satisfies
\begin{equation}
  \widehat s_i^{\mathrm{GS}}(\overline{\mathcal E})
  =
  \widehat\pi_*(\widehat\xi^{r-1+i})
  +a\bigl(\eta_i(\overline{\mathcal E})\bigr)
  \quad\text{in }\widehat{\mathrm{CH}}^i(\mathcal X)_{\QQ}.
  \label{eq:bott-chern}
\end{equation}
With this correction, arithmetic Segre and Chern classes satisfy the inverse
relation in the arithmetic Chow ring.  The correction need not vanish.
\end{proposition}

\begin{proof}
Set $\mathcal Y=\Pj(\mathcal E)$.  The metric on $\mathcal E$ induces the
submetric and quotient metric in the following relative Euler sequence:
$$
\overline\Sigma(1)\colon 0\longrightarrow\overline{\mathcal O}_{\mathcal Y}\longrightarrow\widehat\pi^*\overline{\mathcal E}^{\vee}\otimes\overline{\mathcal O}_{\mathcal Y}(1)\longrightarrow\overline T_{\mathcal Y/\mathcal X}\longrightarrow0.
$$
The first term has the flat metric, and the last term has the quotient metric
induced by the middle term.  Let
$\widetilde c_r(\overline\Sigma(1))\in
\widetilde A^{r-1,r-1}(\mathcal Y)$ be the top Bott--Chern class.  Its
transgression equation is
$$
 c_r\bigl(\widehat\pi^*\mathcal E^{\vee}\otimes\mathcal O_{\mathcal Y}(1)\bigr)=-\frac{\sqrt{-1}}{2\pi}\partial\overline\partial\,\widetilde c_r(\overline\Sigma(1)).
$$
Equivalently, the defining relation for arithmetic Chern classes gives
$$
\sum_{p+q=r}\widehat\pi^*\widehat c_p^{\mathrm{GS}}(\overline{\mathcal E}^{\vee})\widehat\xi^q=-a\bigl(\widetilde c_r(\overline\Sigma(1))\bigr).
$$
These two identities follow from the Bott--Chern formalism of
\cite{BGSI,GSCharI,GSCharII}.  Put
$\alpha=c_1(\mathcal O_{\mathcal Y}(1),h_Q)$, and let $\pi_{\mathbb C}$ denote the
projection of complex fibres.  Define
$$
S_{m+1}:=\pi_{\mathbb C,*}\bigl(\alpha^m\widetilde c_r(\overline\Sigma(1))\bigr),\qquad S_t:=\sum_{m\geq0}S_{m+1}t^m.
$$
Set
$\widehat s_i'=\widehat\pi_*(\widehat\xi^{r-1+i})$ and
$\widehat s_t'=\sum_{i\geq0}\widehat s_i't^i$.  Multiplication of the last
arithmetic identity by $\widehat\xi^m$, followed by the projection formula,
gives a relation between the coefficients.  The duality identity above
rewrites this relation as
$$
\widehat c_t^{\mathrm{GS}}(\overline{\mathcal E})\widehat s_{-t}'=1+a(tS_{-t}).
$$
Mourougane defines the forms $R_i$ by
$$
R_t:=\sum_{m\geq0}R_{m+1}t^m=c_{-t}(\mathcal E_{\mathbb C},h)^{-1}S_t.
$$
The identity $x\,a(\gamma)=a(\omega(x)\gamma)$ in the arithmetic Chow ring
then implies
$$
\widehat c_t^{\mathrm{GS}}(\overline{\mathcal E})\bigl(\widehat s_{-t}'-a(tR_{-t})\bigr)=1.
$$
The inverse series is unique.  Therefore
$\widehat s_i^{\mathrm{GS}}(\overline{\mathcal E})
=\widehat s_i'+a(R_i)$, and \cref{eq:bott-chern} holds with
$\eta_i(\overline{\mathcal E})=R_i$; see
\cite[Theorem~3 in Section~7.2, Proposition~6 in Section~7.1,
and Theorem~4 in Section~8.1]{Mourougane}.

Bott--Chern classes commute with pullback, and fibre integration on a
projective bundle commutes with smooth base change.  Thus, for every smooth
morphism $g$, one has
$R_i(g^*\overline{\mathcal E})=g^*R_i(\overline{\mathcal E})$.  Finally,
Mourougane computes, in \cite[Section~8.2]{Mourougane},
$$
R_1=-\kappa_r,\qquad \kappa_r:=\sum_{j=1}^{r-1}\sum_{\ell=1}^j\frac1\ell.
$$
For $r\geq2$, this correction is nonzero.  For example, on
$\mathcal X=\operatorname{Spec}\mathbb Z$ one has
$\widehat{\deg}\,a(R_1)=-\kappa_r/2\neq0$.
\end{proof}

\subsection{Examples and consistency checks}

\begin{example}[A line bundle]
For a metrized line bundle $\overline M$,
\cref{prop:rank-one} gives
$$
  \whs_i(\overline M)=\whc_1(\overline M)^i.
$$
The rank convention gives $\whc_1(\overline M)
=\whc_1(\overline M)$ and $\whc_i=0$ for $i>1$.  Consequently,
at the level of numerical polynomial evaluations,
$$
  \whch(\overline M)=e^{\whc_1(\overline M)},
  \qquad
  \whTd(\overline M)
  =\frac{\whc_1(\overline M)}
  {1-e^{-\whc_1(\overline M)}}.
$$
\end{example}

\begin{example}[The arithmetic point]
Let $X=\operatorname{Spec}K$.  For an integrable adelic vector bundle
$\overline E$ of rank $r$ on $X$,
the positive-codimension Segre number is
\begin{equation}
  \whs_1(\overline E)
  =
  \whdeg_{\Pj(E)}
  \whc_1(\overline{\mathcal O}_{\Pj(E)}(1))^r.
  \label{eq:point-example}
\end{equation}
It is the adelic height of the metrized projective space.  It is nonnegative
when $\overline{\mathcal O}(1)$ is nef.  Its possible discrepancy from the
degree of $\det\overline E$ is the simplest instance of
\cref{sec:bott-chern}.
\end{example}

\begin{example}[Twisting in low degree]
\Cref{eq:twist-formula} gives
\begin{align*}
 \whs_1(\overline E\otimes\overline M)
 &=\whs_1(\overline E)+r\whc_1(\overline M),\\
 \whs_2(\overline E\otimes\overline M)
 &=\whs_2(\overline E)
 +(r+1)\whc_1(\overline M)\whs_1(\overline E)
 +\binom{r+1}{2}\whc_1(\overline M)^2.
\end{align*}
\end{example}

\begin{proposition}
\label{prop:algebraic-component}
Let $E$ be a vector bundle of rank $r\geq1$ on $X$.
Let $\pi:\Pj(E)\to X$ be the projective bundle that parametrizes
line bundle quotients of $E$.  Put
$$
\xi=c_1\!\left(\mathcal O_{\Pj(E)}(1)\right).
$$
For $i\geq0$, define the Segre operator on $\alpha\in\mathrm{CH}_k(X)_{\QQ}$ by
\begin{equation}
  s_i(E)(\alpha):=\pi_*\!\left(\xi^{r-1+i}\cap\pi^*\alpha\right)
  \in\mathrm{CH}_{k-i}(X)_{\QQ}.
  \label{eq:algebraic-segre}
\end{equation}
Let $C_j$ be the universal polynomials defined by \cref{eq:chern-segre}.
For every $j\geq1$, we have
$$
C_j\!\left(s_1(E),\ldots,s_j(E)\right)(\alpha)=c_j(E)\cap\alpha.
$$
Here products of operators mean composition.
The class on the right is the ordinary algebraic Chern class acting on
Chow homology.
\end{proposition}

\begin{proof}
We use only ordinary Chow groups in this proof.
The projection $\pi$ is flat of relative dimension $r-1$, so its flat
pullback is defined.
The projection is also proper, so its pushforward is defined.

The projective bundle formula gives
$$
s_0(E)(\alpha)=\alpha,\qquad
\pi_*\!\left(\xi^a\cap\pi^*\alpha\right)=0\quad(0\leq a<r-1).
$$
See \cite[Tag~02TW]{Stacks}.
The ordinary Chern classes satisfy the relation
$$
\sum_{j=0}^{r}(-1)^j\xi^{r-j}\cap\pi^*\!\left(c_j(E)\cap\alpha\right)=0.
$$
This is the relation for the projective bundle that parametrizes quotients;
see \cite[Tag~02U6]{Stacks}.

Fix $m\geq1$.
Apply $\xi^{m-1}\cap-$ to this relation.
Then apply $\pi_*$.
The projection formula moves each Chern operator outside the pushforward.
The vanishing above removes the terms with $j>m$.
We obtain
$$
\sum_{j=0}^{\min(m,r)}(-1)^j c_j(E)\cap s_{m-j}(E)(\alpha)=0.
$$
This recurrence expresses each $s_m(E)$ as a polynomial in the ordinary
Chern operators.
The Chern operators commute, so the operators $s_m(E)$ also commute.

Write
$$
s_t(E):=\sum_{i\geq0}s_i(E)t^i,\qquad
c_t(E):=\sum_{j=0}^{r}\bigl(c_j(E)\cap-\bigr)t^j.
$$
The recurrence and the identity $s_0(E)=\operatorname{id}$ give
$$
c_{-t}(E)s_t(E)=\operatorname{id}.
$$
Replace $t$ by $-t$.
This gives
$$
c_t(E)s_{-t}(E)=\operatorname{id}.
$$
Equation~\eqref{eq:chern-segre} defines the polynomials $C_j$ by the same
inverse identity.
A formal power series with constant term equal to the identity has a unique
inverse.
Comparing coefficients therefore proves the assertion.
\end{proof}

\section{Tautological numerical adelic intersection algebra}
\label{sec:numerical-adelic-chow-ring}

Section~1 defines numerical characteristic functionals by adelic intersection
pairings on fiber products of projective bundles.  We now organize these
functionals and their universal polynomial operations into a graded algebra.
Multiplication in this algebra is defined formally.  The degree-$(d+1)$
functional in \cref{def:top-degree} is the corresponding adelic intersection
number.

Let $K$ be a number field and let $X$ be a projective integral
$K$-variety of dimension $d$.  We retain the notation of the preceding
section.  Thus $\widehat{\Pic}(X)_{\intg,\RR}$ denotes the real vector space
of integrable adelic line bundles.  For an integrable adelic vector bundle
$\overline E$ on $X$, let
$\overline\xi_E:=\overline{\mathcal O}_{\BP(E)}(1)$ denote its tautological
adelic quotient line bundle.

This construction uses adelic intersection numbers on fiber products
of projective bundles.  The resulting ring is the maximal
numerically separated quotient of the graded algebra
generated by adelic divisor classes and tautological Segre symbols.  In
particular, it provides mixed products of characteristic classes attached to
different adelic vector bundles.
The required adelic intersection theory is supplied by
\cite{YZ}; our projective-bundle and Segre sign conventions follow
\cite[Chapter~3]{Fulton}.
For the classical theory of arithmetic Chow rings, see
\cite{BurgosChow,BKK,GS}.

\begin{definition}[Tautological algebra]
	\label{def:tautological-algebra}
	Let $\mathscr T^\bullet\ad(X)$ be the commutative graded $\RR$-algebra
	generated by
	\begin{enumerate}[label=\textup{(\roman*)},leftmargin=2.2em]
		\item $\lambda(\overline L)$ of degree $1$, for
		$\overline L\in\widehat{\Pic}(X)_{\intg,\RR}$;
		\item $\sigma_i(\overline E)$ of degree $i$, for $i\geq0$ and every
		integrable adelic vector bundle $\overline E$ of positive rank on $X$.
	\end{enumerate}
	We impose the relations
	\begin{align}
		\lambda(a\overline L+b\overline M)
		&=a\lambda(\overline L)+b\lambda(\overline M),
		\label{eq:lambda-linearity}\\
		\sigma_0(\overline E)&=1,
		\label{eq:sigma-zero}\\
		\sigma_i(\overline M)&=\lambda(\overline M)^i
		\qquad\text{if $\rank M=1$},
		\label{eq:rank-one-relation}
	\end{align}
	and identify symbols attached to isometric adelic data.
\end{definition}

The generator $\sigma_i(\overline E)$ is a formal lift of the numerical Segre
functional $\widehat s_i(\overline E)$ defined in Section~1.  We define
products of these symbols in the formal algebra.
Its image in the numerical quotient will again be denoted by
$\widehat s_i(\overline E)$.

\subsection{The top-degree functional}

It suffices to define the degree on a monomial
\begin{equation}
	M=
	\prod_{a=1}^{m}\sigma_{i_a}(\overline E_a)
	\prod_{j=1}^{n}\lambda(\overline L_j),
	\qquad
	\sum_{a=1}^{m}i_a+n=d+1.
	\label{eq:top-monomial}
\end{equation}
Write $r_a=\rank E_a$ and form
\begin{equation}
u:	Y_M=
	\BP(E_1)\times_X\cdots\times_X\BP(E_m)
	\to X.
	\label{eq:ring-fiber-product}
\end{equation}
Let $q_a\colon Y_M\to\BP(E_a)$ be the projection and put
$\overline\xi_a=q_a^*\overline\xi_{E_a}$.

\begin{definition}[Top degree]
	\label{def:top-degree}
	For the monomial in \cref{eq:top-monomial}, define
	\begin{equation}
		\widehat{\deg}_X(M):=
		\widehat{\deg}_{Y_M}\!\left(
		\prod_{a=1}^{m}
		\widehat c_1(\overline\xi_a)^{r_a-1+i_a}
		\prod_{j=1}^{n}
		\widehat c_1(u^*\overline L_j)
		\right).
		\label{eq:degree-monomial}
	\end{equation}
	For a finite formal sum of degree-$(d+1)$ monomials, write
	$$
	F=\sum_{\nu=1}^{N}a_\nu M_\nu,
	\qquad a_\nu\in\RR,
	$$
	and set
	$$
	\widetilde{\deg}_X(F):=
	\sum_{\nu=1}^{N}a_\nu\widehat{\deg}_X(M_\nu).
	$$
	\end{definition}

\begin{lemma}
	\label{lem:degree-dimension}
	The right-hand side of \cref{eq:degree-monomial} is a well-defined absolute
	adelic intersection number.
\end{lemma}

\begin{proof}
	The fiber product $Y_M$ is projective over $K$.  Since $X$ is integral
	and a projective bundle is Zariski locally a product with projective
	space, the iterated fiber product $Y_M$ is integral.  Moreover,
	$$
	\dim Y_M=d+\sum_{a=1}^{m}(r_a-1).
	$$
	The number of first Chern factors in \cref{eq:degree-monomial} is
	$$
	\sum_{a=1}^{m}(r_a-1+i_a)+n
	=\sum_{a=1}^{m}(r_a-1)+(d+1)
	=\dim Y_M+1.
	$$
	Every metrized line bundle in the intersection is integrable, because
	integrability is preserved by pullback.  Hence the absolute adelic
	intersection pairing of \cite{YZ} applies.
\end{proof}

\begin{proposition}
	\label{prop:degree-well-defined}
	The value $\widetilde{\deg}_X(F)$ is unchanged if the factors of the
	monomials are reordered or if $F$ is modified using the relations in
	Definition~\ref{def:tautological-algebra}.  Hence it depends only on the
	image of $F$ in $\mathscr T^{d+1}\ad(X)$ and defines a linear functional
	$$
	\widehat{\deg}_X\colon
	\mathscr T^{d+1}\ad(X)\longrightarrow\RR
	$$
	on the tautological algebra.
\end{proposition}

\begin{proof}
Let $\tau$ be a permutation of the Segre class factors in
\cref{eq:top-monomial}, and let $M^\tau$ denote the resulting monomial.
There is a canonical $X$-isomorphism
$$
\varphi_\tau\colon Y_M\xrightarrow{\sim}Y_{M^\tau}
$$
which permutes the corresponding factors and satisfies
$$
u_{M^\tau}\circ\varphi_\tau=u_M,
\qquad
q_{\tau(a)}^{M^\tau}\circ\varphi_\tau=q_a^M.
$$
Consequently, $\varphi_\tau$ identifies the metrized tautological line
bundles occurring in the two intersection products.  Invariance of adelic
intersection numbers under isomorphisms, together with their symmetry,
therefore shows that $\widehat{\deg}_X(M)$ is independent of the ordering
of all Segre and divisor factors.

For compatibility with the defining relations, let $M_0$ be a monomial of
degree $d$.  For
$\overline L,\overline M\in\widehat{\Pic}(X)_{\intg,\RR}$ and
$a,b\in\RR$, one has
$$
\widehat c_1\!\left(
u^*(a\overline L+b\overline M)
\right)
=
a\,\widehat c_1(u^*\overline L)
+
b\,\widehat c_1(u^*\overline M).
$$
The multilinearity of the adelic intersection pairing hence gives
$$
\widehat{\deg}_X\!\left(
M_0\lambda(a\overline L+b\overline M)
\right)
=
a\,\widehat{\deg}_X\!\left(M_0\lambda(\overline L)\right)
+
b\,\widehat{\deg}_X\!\left(M_0\lambda(\overline M)\right),
$$
which proves compatibility with \cref{eq:lambda-linearity}.  An
isometry $\overline E\simeq\overline E'$ induces an $X$-isomorphism
$\BP(E)\simeq\BP(E')$ identifying the metrized tautological quotient line
bundles, while an isometry $\overline L\simeq\overline L'$ identifies their
pulled-back metrized line bundles.  The resulting intersection numbers are
therefore unchanged, so the isometry relations are also respected.
	
	If $\overline M$ has rank one, then $\BP(M)=X$ and its tautological quotient
	line bundle is $\overline M$.  Thus replacing
	$\sigma_i(\overline M)$ by $\lambda(\overline M)^i$ leaves
	\cref{eq:degree-monomial} unchanged.

	It remains to check the relation $\sigma_0(\overline E)=1$.
	Let $H$ be a monomial of degree $d+1$.
	Form $u:B\to X$ from its Segre factors as in
	\cref{eq:ring-fiber-product}, and put $b:=\dim B$.
	If $H$ has no Segre factors, take $B=X$ and $u=\operatorname{id}_X$.
	List the integrable adelic line bundles in the intersection defining
	$\widehat{\deg}_X(H)$ as $\overline N_1,\ldots,\overline N_{b+1}$.
	First assume $\overline N_j\in\widehat{\Pic}(B)_{\intg}$ for every $j$.
	Let $r:=\operatorname{rk}E$ and set
	$$
	p:P:=B\times_X\mathbb P(E)\longrightarrow B.
	$$
	Let $\overline\xi$ be the pullback of the original metrized line bundle
	$\overline\xi_E$ to $P$.
	The schemes $B$ and $P$ are integral and projective over $K$.
	Their structure morphisms to $\operatorname{Spec}K$ are flat, and
	$\operatorname{Spec}K$ is normal.
	For the generic point $\eta$ of $B$, we have
	$$
	P_\eta\simeq\mathbb P^{r-1}_{K(B)},\qquad
	\xi|_{P_\eta}\simeq\mathcal O_{\mathbb P^{r-1}_{K(B)}}(1),\qquad
	\deg_{P_\eta}\!\left(c_1(\xi|_{P_\eta})^{r-1}\right)=1.
	$$
	For $r=1$, the last number is the degree of
	$\operatorname{Spec}K(B)$ over $K(B)$.
	Apply \cite[Lemma~4.6.1(3), pp.~128--129, and Section~4.5]{YZ} to
	$P\xrightarrow{p}B\to\operatorname{Spec}K$.
	Take $r-1$ copies of $\overline\xi$ as the relative inputs.
	The generic fibre coefficient is $1$.
	Writing the Deligne pairings over $K$ with subscripts $P/K$ and $B/K$,
	we obtain the isometry
	$$
	\left\langle\underbrace{\overline\xi,\ldots,\overline\xi}_{r-1\ \mathrm{times}},
	p^*\overline N_1,\ldots,p^*\overline N_{b+1}\right\rangle_{P/K}
	\simeq\left\langle\overline N_1,\ldots,\overline N_{b+1}\right\rangle_{B/K}.
	$$
	Taking normalized adelic degrees and using
	\cite[Lemma~4.4.3(1) and Section~4.5]{YZ} gives
	$$
	\widehat{\deg}_X\!\left(H\sigma_0(\overline E)\right)
	=\widehat{\deg}_X(H).
	$$
	Finite real linear expansion gives the same equality for real divisor
	factors and real linear combinations of monomials.
	Thus the relation $\sigma_0(\overline E)=1$ preserves the degree.
\end{proof}

\begin{proposition}
	\label{prop:degree-continuity-positivity}
	Fix a monomial $M$ as in \cref{eq:top-monomial}.  Let
	$\overline E_{a,\nu}$ and $\overline L_{j,\nu}$ be model approximations of
	$\overline E_a$ and $\overline L_j$.  Assume that all their pullbacks to
	$Y_M$ form a simultaneously controlled family.  Then the intersection
	numbers in \cref{eq:degree-monomial} converge to
	$\widehat{\deg}_X(M)$.  If every $\overline E_a$ and every
	$\overline L_j$ is nef, then $\widehat{\deg}_X(M)\geq0$.
\end{proposition}

\begin{proof}
	For controlled continuity, fix the underlying vector bundles
	$E_1,\ldots,E_m$ and line bundles $L_1,\ldots,L_n$, so that the
	fiber product
	$$
	Y_M=\BP(E_1)\times_X\cdots\times_X\BP(E_m)
	$$
	is fixed. For each $\nu\geq1$, let
	$\overline E_{a,\nu}$ and $\overline L_{j,\nu}$ be integrable adelic
	model metrics.  Assume that the sequences
	$$
	\bigl(q_a^*\overline\xi_{E_{a,\nu}}\bigr)_{\nu\geq1}
	\quad(1\leq a\leq m),
	\qquad
	\bigl(u^*\overline L_{j,\nu}\bigr)_{\nu\geq1}
	\quad(1\leq j\leq n)
	$$
	form a simultaneously controlled family on $Y_M$, with respective limits
	$q_a^*\overline\xi_{E_a}$ and $u^*\overline L_j$.
	
	Set
	$$
	N=\dim Y_M
	=d+\sum_{a=1}^m(r_a-1).
	$$
	By \cref{eq:top-monomial}, the intersection product defining
	$\widehat{\deg}_X(M)$ contains
	$$
	\sum_{a=1}^m(r_a-1+i_a)+n=N+1
	$$
	metrized line-bundle factors.  \Cref{eq:controlled-intersection-continuity}
	on $Y_M$ therefore gives
	\begin{align*}
		&\widehat{\deg}_{Y_M}\!\left(
		\prod_{a=1}^m
		\widehat c_1(q_a^*\overline\xi_{E_{a,\nu}})^{r_a-1+i_a}
		\prod_{j=1}^n
		\widehat c_1(u^*\overline L_{j,\nu})
		\right)\\
		&\qquad\longrightarrow
		\widehat{\deg}_{Y_M}\!\left(
		\prod_{a=1}^m
		\widehat c_1(q_a^*\overline\xi_{E_a})^{r_a-1+i_a}
		\prod_{j=1}^n
		\widehat c_1(u^*\overline L_j)
		\right).
	\end{align*}
	Thus the expression has the asserted controlled continuity.
	
	Assume that every $\overline E_a$ is nef and every
	$\overline L_j$ is nef. By definition,
	$\overline\xi_{E_a}$ is a nef adelic line bundle on $\BP(E_a)$.
	Since nefness is preserved under pullback, all the adelic line bundles
	$$
	q_a^*\overline\xi_{E_a}
	\qquad\text{and}\qquad
	u^*\overline L_j
	$$
	are nef on $Y_M$. The nonnegativity theorem for top intersections of
	nef adelic line bundles consequently yields
	$$
	\widehat{\deg}_{Y_M}\!\left(
	\prod_{a=1}^m
	\widehat c_1(q_a^*\overline\xi_{E_a})^{r_a-1+i_a}
	\prod_{j=1}^n
	\widehat c_1(u^*\overline L_j)
	\right)\geq0.
	$$
	By \cref{eq:degree-monomial}, the left-hand side is
	$\widehat{\deg}_X(M)$.  This proves nonnegativity.
\end{proof}

\subsection{The numerical radical and the quotient ring}

\begin{definition}[Numerical radical]
	\label{def:numerical-radical}
	For $0\leq p\leq d+1$, set
	\begin{equation}
		\mathscr I^p_{\mathrm{num}}(X)=
		\left\{
		\alpha\in\mathscr T^p\ad(X):
		\widehat{\deg}_X(\alpha\beta)=0
		\text{ for every }
		\beta\in\mathscr T^{d+1-p}\ad(X)
		\right\}.
		\label{eq:numerical-radical}
	\end{equation}
	For $p>d+1$, put
	$\mathscr I^p_{\mathrm{num}}(X)=\mathscr T^p\ad(X)$, and write
	$$
	\mathscr I^\bullet_{\mathrm{num}}(X)
	=\bigoplus_{p\geq0}\mathscr I^p_{\mathrm{num}}(X).
	$$
\end{definition}

\begin{lemma}[Ideal property]
	\label{lem:radical-ideal}
	The graded vector space $\mathscr I^\bullet_{\mathrm{num}}(X)$ is a homogeneous
	ideal of $\mathscr T^\bullet\ad(X)$.
\end{lemma}

\begin{proof}
	The homogeneity of
	$\mathscr I^\bullet_{\mathrm{num}}(X)$ follows from its
	definition.  For the ideal property, let
	$\alpha\in\mathscr I^p_{\mathrm{num}}(X)$ and
	$\gamma\in\mathscr T^q\ad(X)$.  If $p+q>d+1$, the required containment
	holds by definition.  Otherwise, for every
	$\beta\in\mathscr T^{d+1-p-q}\ad(X)$,
	$$
	\widehat{\deg}_X((\alpha\gamma)\beta)
	=\widehat{\deg}_X(\alpha(\gamma\beta))=0,
	$$
	because $\gamma\beta$ has degree $d+1-p$.
\end{proof}

\begin{definition}[Tautological numerical adelic ring]
	\label{def:numerical-adelic-chow-ring}
	The \emph{tautological numerical adelic ring} of $X$ is the graded
	quotient
	\begin{equation}
		\widehat N^\bullet\ad(X)
		:=
		\mathscr T^\bullet\ad(X)\big/
		\mathscr I^\bullet_{\mathrm{num}}(X).
		\label{eq:numerical-adelic-chow-ring}
	\end{equation}
	For $\alpha\in\mathscr T^\bullet\ad(X)$, denote its image by
	$[\alpha]_{\mathrm{num}}$.  By \cref{lem:radical-ideal}, multiplication
	is induced from $\mathscr T^\bullet\ad(X)$.  Thus, for
	$\alpha\in\mathscr T^p\ad(X)$ and $\beta\in\mathscr T^q\ad(X)$,
	$$
	[\alpha]_{\mathrm{num}}[\beta]_{\mathrm{num}}
	:=
	[\alpha\beta]_{\mathrm{num}}
	\in\widehat N^{p+q}\ad(X).
	$$
\end{definition}

\begin{proposition}[Numerical duality]
	\label{prop:complementary-duality}
	For every $0\leq p\leq d+1$, multiplication followed by the
	top-degree functional induces a nondegenerate pairing
	$$
	\widehat N^p\ad(X)\times
	\widehat N^{d+1-p}\ad(X)
	\longrightarrow \RR,
	\qquad
	([\alpha]_{\mathrm{num}},[\beta]_{\mathrm{num}})
	\longmapsto
	\widehat{\deg}_X(\alpha\beta).
	$$
\end{proposition}

\begin{proof}
	The pairing is well defined by the definition of the numerical radical.
	Indeed, replacing either $\alpha$ or $\beta$ by an element with the same
	numerical class does not change $\widehat{\deg}_X(\alpha\beta)$.
	
	Let $[\alpha]_{\mathrm{num}}$ belong to the left radical.  Then
	$$
	\widehat{\deg}_X(\alpha\beta)=0
	\qquad
	\text{for every }
	\beta\in\mathscr T^{d+1-p}\ad(X).
	$$
	Hence $\alpha\in\mathscr I^p_{\mathrm{num}}(X)$, and therefore
	$[\alpha]_{\mathrm{num}}=0$. The same argument applies to the second
	factor. Thus the pairing is nondegenerate.
\end{proof}

\noindent\textbf{The algebra without metrics.}
Put $\Pic(X)_{\RR}:=\Pic(X)\otimes_{\ZZ}\RR$.
Let $\mathscr T^\bullet_{\mathrm{alg}}(X)$ be the commutative graded
$\RR$-algebra generated by symbols
$$
\lambda_{\mathrm{alg}}(L)\quad\text{of degree }1,
\qquad
\sigma_{i,\mathrm{alg}}(E)\quad\text{of degree }i.
$$
Here $L$ ranges over $\Pic(X)_{\RR}$.
The symbol $E$ ranges over algebraic vector bundles of positive rank on $X$,
and $i\geq0$.
We identify symbols attached to isomorphic algebraic data.
For $a,b\in\RR$ and $L,M\in\Pic(X)_{\RR}$, impose the relation
$$
\lambda_{\mathrm{alg}}(aL+bM)
=a\lambda_{\mathrm{alg}}(L)+b\lambda_{\mathrm{alg}}(M).
$$
For every vector bundle $E$ and every algebraic line bundle $M$, impose
$$
\sigma_{0,\mathrm{alg}}(E)=1,
\qquad
\sigma_{i,\mathrm{alg}}(M)=\lambda_{\mathrm{alg}}(M)^i.
$$
We impose no other relations at this stage.

\noindent\textbf{The algebraic degree.}
First consider a monomial
$$
A=\prod_{a=1}^{m}\sigma_{i_a,\mathrm{alg}}(E_a)
\prod_{j=1}^{n}\lambda_{\mathrm{alg}}(L_j),
\qquad \sum_{a=1}^{m}i_a+n=d,
$$
where each $L_j$ is an algebraic line bundle.
Write $r_a:=\rank E_a$ and set
$$
u_A:Y_A:=\BP(E_1)\times_X\cdots\times_X\BP(E_m)\longrightarrow X.
$$
For $m=0$, set $Y_A:=X$ and $u_A:=\operatorname{id}_X$.
Let $q_a:Y_A\to\BP(E_a)$ be the projections and put
$\xi_a:=q_a^*\mathcal O_{\BP(E_a)}(1)$.
Define
$$
\deg_X(A):=\deg_{Y_A}\!\left(
\left(\prod_{a=1}^{m}c_1(\xi_a)^{r_a-1+i_a}
\prod_{j=1}^{n}c_1(u_A^*L_j)\right)\cap[Y_A]\right).
$$
Each first Chern class acts as an operator on Chow homology.
The symbol $\deg_{Y_A}$ denotes the degree over $K$ of a zero-cycle.
The scheme $Y_A$ is projective and integral.
The number of Chern operators is
$$
\sum_{a=1}^{m}(r_a-1+i_a)+n
=d+\sum_{a=1}^{m}(r_a-1)=\dim Y_A.
$$
Thus the displayed formula is defined even when $X$ is singular.
Extend it linearly to $L_j\in\Pic(X)_{\RR}$ and to finite real sums
of monomials.

We check the defining relations.
First Chern operators are linear in line bundle classes and commute with
one another.
Hence the formula respects the linearity relations and permutations of
the divisor factors.
Permuting the projective bundle factors induces an isomorphism of $Y_A$
that identifies their tautological line bundles.
This proves invariance under permutations of the Segre factors.
The same argument proves invariance under isomorphisms of the algebraic data.
For a line bundle $M$, the identification $\BP(M)=X$ identifies the
tautological line bundle with $M$.
It therefore proves the relation for a rank one bundle.
Finally, for any projective bundle $h:\BP_V(G)\to V$ of rank $s$,
the projective bundle formula gives
$$
h_*\!\left(c_1(\mathcal O_{\BP_V(G)}(1))^{s-1}\cap h^*\alpha\right)
=\alpha\qquad\text{for every }\alpha\in\mathrm{CH}_*(V)_{\RR};
$$
see \cite[Tag~02TW]{Stacks}.
Apply this formula with $V$ equal to the fibre product of the remaining
projective bundle factors.
It removes a factor $\sigma_{0,\mathrm{alg}}(E)$ without changing the degree.
We have therefore defined a linear functional
$$
\deg_X:\mathscr T^d_{\mathrm{alg}}(X)\longrightarrow\RR.
$$

\noindent\textbf{The algebraic numerical radical.}
For $0\leq p\leq d$, let $\mathscr I^p_{\mathrm{num},\mathrm{alg}}(X)$
consist of those $\alpha\in\mathscr T^p_{\mathrm{alg}}(X)$ such that
$$
\deg_X(\alpha\beta)=0
\qquad\text{for every }\beta\in\mathscr T^{d-p}_{\mathrm{alg}}(X).
$$
For $p>d$, put
$\mathscr I^p_{\mathrm{num},\mathrm{alg}}(X):=\mathscr T^p_{\mathrm{alg}}(X)$.
Set
$$
\mathscr I^\bullet_{\mathrm{num},\mathrm{alg}}(X)
:=\bigoplus_{p\geq0}\mathscr I^p_{\mathrm{num},\mathrm{alg}}(X).
$$
This is a homogeneous ideal.
Indeed, take $\alpha\in\mathscr I^p_{\mathrm{num},\mathrm{alg}}(X)$
and $\gamma\in\mathscr T^q_{\mathrm{alg}}(X)$.
If $p+q>d$, the ideal condition follows from the definition.
Otherwise, every $\beta\in\mathscr T^{d-p-q}_{\mathrm{alg}}(X)$ satisfies
$$
\deg_X((\alpha\gamma)\beta)=\deg_X(\alpha(\gamma\beta))=0,
$$
because $\gamma\beta$ has degree $d-p$.

Define the algebraic numerical ring by
$$
N^\bullet_{\mathrm{alg}}(X)_{\RR}
:=\mathscr T^\bullet_{\mathrm{alg}}(X)
/\mathscr I^\bullet_{\mathrm{num},\mathrm{alg}}(X).
$$
Its multiplication is induced from the formal algebra.
In particular, $N^p_{\mathrm{alg}}(X)_{\RR}=0$ for $p>d$.
Taking $\beta=1$ in degree $d$ shows that $\deg_X$ descends to
$N^d_{\mathrm{alg}}(X)_{\RR}$.
Write $[\alpha]_{\mathrm{num},\mathrm{alg}}$ for the image of a formal element
$\alpha$ in this quotient.
Define
$$
c_1(L):=[\lambda_{\mathrm{alg}}(L)]_{\mathrm{num},\mathrm{alg}},
\qquad
s_i(E):=[\sigma_{i,\mathrm{alg}}(E)]_{\mathrm{num},\mathrm{alg}}.
$$
Put $s_t(E):=\sum_{i\geq0}s_i(E)t^i$.
For a vector bundle $E$ of rank $r$, define
$$
c_i(E):=[t^i]s_{-t}(E)^{-1}\quad(0\leq i\leq r),
\qquad c_i(E):=0\quad(i>r).
$$
These are classes in the algebraic numerical ring.
For a line bundle, this definition of $c_1$ agrees with the preceding one
by the rank one relation.

Forgetting metrics now defines a graded homomorphism
$$
F:\mathscr T^\bullet\ad(X)\longrightarrow\mathscr T^\bullet_{\mathrm{alg}}(X),
\qquad
\lambda(\overline L)\longmapsto\lambda_{\mathrm{alg}}(L),
\quad
\sigma_i(\overline E)\longmapsto\sigma_{i,\mathrm{alg}}(E).
$$
The assignments respect all the defining relations.
An adelic lift of $A\in\mathscr T^\bullet_{\mathrm{alg}}(X)$ means an element
$\overline A\in\mathscr T^\bullet\ad(X)$ with $F(\overline A)=A$.

\begin{lemma}[Normalized vertical class]
	\label{lem:vertical-class}
	There exists
	$\overline L\in\widehat{\Pic}(\Spec K)_{\intg,\RR}$ such that
	$\widehat{\deg}_K(\overline L)=1$. Let
	$v\colon X\to\Spec K$ be the structure morphism and set
	$$
	\varepsilon
	:=
	[\lambda(v^*\overline L)]_{\mathrm{num}}
	\in\widehat N^1\ad(X).
	$$
	Then, for every degree-$d$ algebraic tautological monomial $A$ and every
	adelic lift $\overline A\in\mathscr T^d\ad(X)$ of $A$, one has
	$$
	\widehat{\deg}_X(\varepsilon\overline A)=\deg_X(A).
	$$
	In particular, the left-hand side is independent of the metrics defining
	$\overline A$.  We interpret $\varepsilon\overline A$ as
	$\varepsilon[\overline A]_{\mathrm{num}}$ in the numerical adelic ring.
\end{lemma}

\begin{proof}
	To construct $\overline L$, choose an
	archimedean place $w_0$ and equip the trivial line $K$ with the standard
	norms away from $w_0$ and with
	$$
	\norm{1}_{w_0}=e^{-1/n_{w_0}},
	\qquad
	n_{w_0}=\frac{[K_{w_0}:\mathbb Q_{w_0}]}{[K:\mathbb Q]}.
	$$
	Then
	$$
	\widehat{\deg}_K(\overline L)
	=
	-\sum_w n_w\log\norm{1}_w
	=
	1.
	$$
	First take an adelic monomial $\overline A$ and put $A:=F(\overline A)$.
	Write
	$$
	A=
	\prod_{a=1}^m\sigma_{i_a,\mathrm{alg}}(E_a)
	\prod_{j=1}^n\lambda_{\mathrm{alg}}(L_j),
	\qquad
	\sum_{a=1}^m i_a+n=d.
	$$
	Put $r_a:=\rank E_a$ and set
	$$
	Y_A=
	\BP(E_1)\times_X\cdots\times_X\BP(E_m)
	\xrightarrow{\,u\,}X.
	$$
	Denote by $q_a:Y_A\to\BP(E_a)$ the projections, put
	$$
	\xi_a=q_a^*\mathcal O_{\BP(E_a)}(1),
	\qquad
	g=v\circ u,
	$$
	and let $\overline\xi_a$ be the corresponding metrized tautological line
	bundle. The definition of the adelic degree gives the first equality below.
	The projection formula for adelic Deligne pairings gives the second;
	see \cite[Lemma~4.6.1(3) and Section~4.5]{YZ}.
	\begin{align*}
		\widehat{\deg}_X(\varepsilon\overline A)
		&=
		\widehat{\deg}_{Y_A}\left(
		\widehat c_1(g^*\overline L)
		\prod_{a=1}^m
		\widehat c_1(\overline\xi_a)^{r_a-1+i_a}
		\prod_{j=1}^n
		\widehat c_1(u^*\overline L_j)
		\right)\\
		&=
		\widehat{\deg}_K(\overline L)\,
		\deg_{Y_A}\left(\left(
		\prod_{a=1}^m
		c_1(\xi_a)^{r_a-1+i_a}
		\prod_{j=1}^n
		c_1(u^*L_j)
		\right)\cap[Y_A]\right).
	\end{align*}
	The definition of the algebraic degree identifies the last factor with
	$\deg_X(A)$.
	Since $\widehat{\deg}_K(\overline L)=1$, we obtain
		$$
		\widehat{\deg}_X(\varepsilon\overline A)
		=\deg_X(A).
		$$
	Linearity now gives
	$$
	\widehat{\deg}_X\bigl(\varepsilon[\delta]_{\mathrm{num}}\bigr)
	=\deg_X(F(\delta))
	\qquad\text{for every }\delta\in\mathscr T^d\ad(X).
	$$
	Apply this identity to any adelic lift of $A$.
	The result depends only on $A$, not on the chosen lift.
\end{proof}

\begin{proposition}
	\label{prop:extreme-degrees}
	The unit class and the top-degree functional induce canonical isomorphisms
	$$
	\mathbb R
	\xrightarrow{\ \sim\ }
	\widehat N^0\ad(X),
	\qquad
	a\longmapsto a[1]_{\mathrm{num}},
	$$
	and
	$$
	\widehat{\deg}_X:
	\widehat N^{d+1}\ad(X)
	\xrightarrow{\ \sim\ }
	\mathbb R.
	$$
	Moreover,
	$$
	\widehat N^p\ad(X)=0
	\qquad\text{for every }p>d+1.
	$$
\end{proposition}

\begin{proof}
	By \cref{def:tautological-algebra}, every generator of
	$\mathscr T^\bullet\ad(X)$ has positive degree, except for
	the symbols $\sigma_0(E)$, all of which are identified with the unit.
	Consequently,
	$$
	\mathscr T^0\ad(X)=\mathbb R\cdot 1.
	$$
	
	Let $\overline L$ be as in \cref{lem:vertical-class}, and put
	$$
	\widetilde\varepsilon
	:=\lambda(v^*\overline L)
	\in\mathscr T^1\ad(X).
	$$
	Choose an ample line bundle $A$ on $X$ and equip it with an integrable
	model metric $\overline A$.  Then
	$$
	\Theta
	:=\widetilde\varepsilon\,\lambda(\overline A)^d
	\in\mathscr T^{d+1}\ad(X).
	$$
	\Cref{lem:vertical-class} gives
	$$
	\widehat{\deg}_X(\Theta)
	=
	\deg\bigl(c_1(A)^d\cap[X]\bigr)>0.
	$$
	In particular, $\widehat{\deg}_X$ is not identically zero on
	$\mathscr T^{d+1}\ad(X)$.
	
	Now let $a\in\mathbb R$ and suppose that
	$a\cdot1\in\mathscr I_{\mathrm{num}}^0(X)$.  By the definition of the
	numerical radical,
	$$
	0
	=
	\widehat{\deg}_X\bigl((a\cdot1)\Theta\bigr)
	=
	a\,\widehat{\deg}_X(\Theta).
	$$
	Since $\widehat{\deg}_X(\Theta)>0$, it follows that $a=0$.  Hence
	$$
	\mathscr I_{\mathrm{num}}^0(X)=0,
	$$
	and the unit map therefore induces the canonical isomorphism
	$\mathbb R\simeq\widehat N^0\ad(X)$.
	
	Since $\mathscr T^0\ad(X)=\mathbb R\cdot1$,
	\cref{def:numerical-radical}
	also gives
	$$
	\mathscr I_{\mathrm{num}}^{d+1}(X)
	=
	\ker\left(
	\widehat{\deg}_X:
	\mathscr T^{d+1}\ad(X)\longrightarrow\mathbb R
	\right).
	$$
	Thus $\widehat{\deg}_X$ induces an injective linear map
	$$
	\widehat N^{d+1}\ad(X)\longrightarrow\mathbb R.
	$$
	Its image contains the nonzero number
	$\widehat{\deg}_X(\Theta)$. Since the image is an
	$\mathbb R$-linear subspace of $\mathbb R$, it is all of $\mathbb R$.
	The induced degree map is therefore an isomorphism.
	
	Finally, for $p>d+1$, \cref{def:numerical-radical} gives
	$$
	\mathscr I_{\mathrm{num}}^p(X)
	=
	\mathscr T^p\ad(X),
	$$
	and hence $\widehat N^p\ad(X)=0$.
\end{proof}

\begin{proposition}[Universal numerical quotient]
	\label{prop:universal-numerical-quotient}
	Let
	$$
	q\colon\mathscr T^\bullet\ad(X)\twoheadrightarrow A^\bullet
	$$
	be a surjective homomorphism of graded $\RR$-algebras.  Suppose that there
	is an $\RR$-linear map
	$$
	\deg_A\colon A^{d+1}\longrightarrow\RR
	$$
	such that
	$$
	\deg_A(q(\alpha))=\widehat{\deg}_X(\alpha)
	\qquad
	\text{for every }\alpha\in\mathscr T^{d+1}\ad(X).
	$$
	Assume that $A^p=0$ for $p>d+1$.  For every $0\leq p\leq d+1$, suppose
	that the pairing
	$$
	A^p\times A^{d+1-p}\longrightarrow\RR,
	\qquad
	(a,b)\longmapsto\deg_A(ab),
	$$
	is nondegenerate.  Then
	$$
	\ker(q)=\mathscr I^\bullet_{\mathrm{num}}(X).
	$$
	There is therefore a unique isomorphism of graded $\RR$-algebras
	$$
	\overline q\colon
	\widehat N^\bullet\ad(X)\xrightarrow{\ \sim\ }A^\bullet,
	\qquad
	[\alpha]_{\mathrm{num}}\longmapsto q(\alpha).
	$$
\end{proposition}

\begin{proof}
	Write $J^\bullet:=\ker(q)$.  For $p>d+1$, one has
	$J^p=\mathscr T^p\ad(X)=\mathscr I^p_{\mathrm{num}}(X)$.
	Now let $0\leq p\leq d+1$.  If $\alpha\in J^p$, then for every
	$\beta\in\mathscr T^{d+1-p}\ad(X)$,
	$$
	\widehat{\deg}_X(\alpha\beta)
	=\deg_A(q(\alpha)q(\beta))=0.
	$$
	Hence $\alpha\in\mathscr I^p_{\mathrm{num}}(X)$, so
	$J^p\subseteq\mathscr I^p_{\mathrm{num}}(X)$.

	Conversely, let $\alpha\in\mathscr I^p_{\mathrm{num}}(X)$ and let
	$b\in A^{d+1-p}$.  Since $q$ is surjective, there is
	$\beta\in\mathscr T^{d+1-p}\ad(X)$ such that $q(\beta)=b$.  Then
	$$
	\deg_A(q(\alpha)b)
	=\widehat{\deg}_X(\alpha\beta)=0.
	$$
	Since $b$ was arbitrary, nondegeneracy gives $q(\alpha)=0$.  Therefore
	$\mathscr I^p_{\mathrm{num}}(X)\subseteq J^p$.  We conclude that
	$J^\bullet=\mathscr I^\bullet_{\mathrm{num}}(X)$, and the displayed map
	$\overline q$ is an isomorphism.
\end{proof}

\subsection{Characteristic classes in the numerical ring}

The definitions below are given by the standard universal relations among
Segre, Chern, Chern-character, and Todd polynomials; see
\cite[Chapter~3]{Fulton}.

\begin{definition}[Segre and Chern classes]
	\label{def:ring-characteristic-classes}
	For an integrable adelic vector bundle $\overline E$, set
	\begin{equation}
		\widehat s_i(\overline E)
		:=[\sigma_i(\overline E)]_{\mathrm{num}},
		\qquad
		\widehat s_t(\overline E)
		:=\sum_{i\geq0}\widehat s_i(\overline E)t^i.
		\label{eq:ring-segre-class}
	\end{equation}
	If $r=\rank E$, define
	\begin{equation}
		\widehat c_i(\overline E)
		:=[t^i]\,\widehat s_{-t}(\overline E)^{-1}
		\quad(0\leq i\leq r),
		\qquad
		\widehat c_i(\overline E):=0\quad(i>r).
		\label{eq:ring-chern-class}
	\end{equation}
	The numerical Chern character and Todd class are the usual universal
	polynomials in these Chern classes:
	\begin{align}
		\widehat{\ch}(\overline E)
		&=r+\widehat c_1(\overline E)
		+\frac12\bigl(
		\widehat c_1(\overline E)^2
		-2\widehat c_2(\overline E)
		\bigr)+\cdots,
		\label{eq:ring-chern-character}\\
		\widehat{\Td}(\overline E)
		&=1+\frac12\widehat c_1(\overline E)
		+\frac1{12}\bigl(
		\widehat c_1(\overline E)^2
		+\widehat c_2(\overline E)
		\bigr)+\cdots.
		\label{eq:ring-todd-class}
	\end{align}
	All expressions are truncated in degrees greater than $d+1$.
\end{definition}

For every integrable adelic line bundle $\overline L$ on $X$, one has
\[
\widehat c_1(\overline L)
=
[\lambda(\overline L)]_{\mathrm{num}}
\in\widehat N^1\ad(X).
\]
Indeed, by
\cref{def:ring-characteristic-classes,eq:rank-one-relation},
\[
\widehat c_1(\overline L)
=
\widehat s_1(\overline L)
=
[\sigma_1(\overline L)]_{\mathrm{num}}
=
[\lambda(\overline L)]_{\mathrm{num}}.
\]

\begin{remark}[Sign convention]
	For a line bundle $\overline M$, one has
	$\widehat s_i(\overline M)=\widehat c_1(\overline M)^i$.
	With this convention, the Chern--Segre relation takes the form
	$$
	\widehat c_t(\overline E)\widehat s_{-t}(\overline E)=1
	\pmod {t^{r+1}},
	$$
	rather than
	$\widehat c_t(\overline E)\widehat s_t(\overline E)=1
	\pmod {t^{r+1}}$.
\end{remark}

\begin{proposition}
	\label{prop:ring-rank-one-twisting}
	Let $\overline E$ have rank $r$, and let $\overline M$ be an integrable
	adelic line bundle.  If $r=1$, then
	$$
	\widehat s_i(\overline E)=\widehat c_1(\overline E)^i,
	\qquad
	\widehat c_t(\overline E)=1+\widehat c_1(\overline E)t.
	$$
	For every $i\geq0$,
	\begin{equation}
		\widehat s_i(\overline E\otimes\overline M)
		=
		\sum_{k=0}^{i}
		\binom{r-1+i}{i-k}
		\widehat s_k(\overline E)
		\widehat c_1(\overline M)^{i-k}.
		\label{eq:ring-twisting}
	\end{equation}
\end{proposition}

\begin{proof}
	If $r=1$, \cref{eq:rank-one-relation} gives
	$\widehat s_i(\overline E)=\widehat c_1(\overline E)^i$.  Hence
	$$
	\widehat s_{-t}(\overline E)
	=
	\bigl(1+\widehat c_1(\overline E)t\bigr)^{-1},
	$$
	and the defining relation
	$\widehat c_t(\overline E)\widehat s_{-t}(\overline E)=1$
	yields
	$$
	\widehat c_t(\overline E)
	=
	1+\widehat c_1(\overline E)t.
	$$
	
	For the twisting formula, identify
	$\mathbb P(E\otimes M)$ with $\mathbb P(E)$.  Under this identification,
	$$
	\overline\xi_{E\otimes M}
	\simeq
	\overline\xi_E\otimes\pi^*\overline M.
	$$
	If $i>d+1$, both sides of \cref{eq:ring-twisting} vanish by
	\cref{prop:extreme-degrees}.  Assume $0\leq i\leq d+1$.  By linearity,
	it suffices to prove that both sides have the same pairing with every monomial class
	$\gamma\in\widehat N^{d+1-i}\ad(X)$.
	Choose a formal monomial $\Gamma$ representing $\gamma$.
	Form $u:B\to X$ from the Segre factors of $\Gamma$ as in
	\cref{eq:ring-fiber-product}, and put $b:=\dim B$.
	If there are no Segre factors, take $B=X$ and $u=\operatorname{id}_X$.
	Let $F_a$ denote the vector bundles in the Segre factors of $\Gamma$.
	Put $r_a:=\operatorname{rk}F_a$.
	Then $b=d+\sum_a(r_a-1)$.
	Since $\Gamma$ has degree $d+1-i$, its intersection on $B$ has
	$\sum_a(r_a-1)+d+1-i=b+1-i$ line bundle factors.
	Denote this list, including repetitions, by
	$\overline N_1,\ldots,\overline N_{b+1-i}$.
	First take divisor factors $\lambda(\overline L)$ with
	$\overline L\in\widehat{\Pic}(X)_{\intg}$.
	Put $\overline M_B:=u^*\overline M$ and set
	$$
	p:P:=B\times_X\mathbb P(E)\longrightarrow B.
	$$
	Let $\overline\xi$ be the pullback of $\overline\xi_E$ to $P$.
	By \cref{eq:degree-monomial}, the pairing of the left side of
	\cref{eq:ring-twisting} with $\gamma$ is the top
	intersection on $P$ with $r-1+i$ copies of
	$\overline\xi\otimes p^*\overline M_B$ and the factors
	$p^*\overline N_1,\ldots,p^*\overline N_{b+1-i}$.

	Expand this intersection by multilinearity.
	The term with $t$ copies of $\overline\xi$ contains
	$r-1+i-t$ copies of $p^*\overline M_B$.
	If $t<r-1$, apply
	\cite[Lemma~4.6.1(3), pp.~128--129, and Section~4.5]{YZ} to
	$P\xrightarrow{p}B\to\operatorname{Spec}K$.
	Both $P$ and $B$ are integral and projective over $K$.
	Their structure morphisms are flat, and $\operatorname{Spec}K$ is normal.
	Take $t$ copies of $\overline\xi$ and $r-1-t$ copies of
	$p^*\overline M_B$ as the relative inputs.
	The remaining inputs are pulled back from $i$ copies of $\overline M_B$
	and the adelic line bundles $\overline N_1,\ldots,\overline N_{b+1-i}$ on $B$.
	This gives exactly $b+1$ inputs from $B$.
	At least one relative input has trivial underlying line bundle on the
	generic fibre of $p$.
	Hence the generic fibre intersection number is zero.
	The Deligne pairing is therefore isometric to the trivial metrized line
	bundle on $\operatorname{Spec}K$.
	Taking normalized adelic degrees shows that this term is zero.

	For $t=r-1+k$ with $0\leq k\leq i$, there remain $i-k$ copies of
	$p^*\overline M_B$.
	By \cref{eq:degree-monomial}, this term equals
	$\widehat{\deg}_X(\widehat s_k(\overline E)
	\widehat c_1(\overline M)^{i-k}\gamma)$.
	Its binomial coefficient is
	$\binom{r-1+i}{r-1+k}=\binom{r-1+i}{i-k}$.
	Thus
	$$
	\begin{aligned}
		\widehat{\deg}_X\!\left(
		\widehat s_i(\overline E\otimes\overline M)\gamma
		\right)
		&=
		\sum_{k=0}^{i}
		\binom{r-1+i}{i-k}
		\widehat{\deg}_X\!\left(
		\widehat s_k(\overline E)
		\widehat c_1(\overline M)^{i-k}\gamma
		\right).
	\end{aligned}
	$$
	For $r=1$, there are no terms with $t<r-1$.
	Finite real linear expansion in the divisor factors extends the equality
	to every monomial class $\gamma\in\widehat N^{d+1-i}\ad(X)$.
	The nondegeneracy of the pairing in
	\cref{prop:complementary-duality} now implies
	\cref{eq:ring-twisting}.
\end{proof}

\begin{remark}
	Write $\ell:=\widehat c_1(\overline M)$.  The cases $i=1$ and $i=2$ of
	\cref{eq:ring-twisting} are
	$$
	\begin{aligned}
		\widehat s_1(\overline E\otimes\overline M)
		&=\widehat s_1(\overline E)+r\ell,\\
		\widehat s_2(\overline E\otimes\overline M)
		&=\widehat s_2(\overline E)
		+(r+1)\widehat s_1(\overline E)\ell
		+\binom{r+1}{2}\ell^2.
	\end{aligned}
	$$
	Consequently,
	$$
	\begin{aligned}
		\widehat c_1(\overline E\otimes\overline M)
		&=\widehat c_1(\overline E)+r\ell,\\
		\widehat c_2(\overline E\otimes\overline M)
		&=\widehat c_2(\overline E)
		+(r-1)\widehat c_1(\overline E)\ell
		+\binom{r}{2}\ell^2.
	\end{aligned}
	$$
	If $\overline E=\overline L$ has rank one, then
	\cref{eq:ring-twisting} reduces to the binomial identity
	$$
	\widehat s_i(\overline L\otimes\overline M)
	=
	\bigl(
	\widehat c_1(\overline L)
	+\widehat c_1(\overline M)
	\bigr)^i.
	$$
\end{remark}

\subsection{The subring generated by divisor classes and the forgetful map}

\begin{definition}[Subring generated by divisor classes]
	\label{def:divisorial-subring}
	Let $\widehat N^\bullet_{\mathrm{div},\mathrm{ad}}(X)$ be the graded
	subalgebra of $\widehat N^\bullet\ad(X)$ generated by
	$$
	\widehat c_1(\overline L)
	:=[\lambda(\overline L)]_{\mathrm{num}},
	\qquad
	\overline L\in\widehat{\Pic}(X)_{\intg,\RR}.
	$$
\end{definition}

\begin{proposition}
	\label{prop:divisorial-quotient}
	Put
	$$
	S^\bullet_X
	:=\Sym^\bullet_{\RR}
	\bigl(\widehat{\Pic}(X)_{\intg,\RR}\bigr).
	$$
	For $a\in S_X^p$, let $\widetilde a\in\mathscr T^p\ad(X)$ denote its
	image under the homomorphism determined by
	$\overline L\mapsto\lambda(\overline L)$.
	For $0\leq p\leq d+1$, let $R_X^p$ consist of those
	$a\in S_X^p$ such that
	$$
	\widehat{\deg}_X(\widetilde a\beta)=0
	\quad\text{for every }\beta\in
	\mathscr T^{d+1-p}\ad(X),
	$$
	and put $R_X^p=S_X^p$ for $p>d+1$.  Set
	$R_X^\bullet:=\bigoplus_{p\geq0}R_X^p$.  Then
	\begin{equation}
		\widehat N^\bullet_{\mathrm{div},\mathrm{ad}}(X)
		\simeq S^\bullet_X/R^\bullet_X.
		\label{eq:divisorial-quotient}
	\end{equation}
\end{proposition}

\begin{proof}
	Let
	$$
	\Phi\colon S^\bullet_X
	\longrightarrow
	\widehat N^\bullet_{\mathrm{div},\mathrm{ad}}(X)
	$$
	be the graded homomorphism induced by
	$$
	\overline L\longmapsto\widehat c_1(\overline L).
	$$
	It is well defined by \cref{eq:lambda-linearity}, and it is surjective by
	Definition~\ref{def:divisorial-subring}.
	
	Let $0\leq p\leq d+1$ and $a\in S_X^p$.  By construction,
	$\Phi(a)=[\widetilde a]_{\mathrm{num}}$.  Thus $\Phi(a)=0$ precisely when
	$\widetilde a\in\mathscr I^p_{\mathrm{num}}(X)$.  By
	\cref{def:numerical-radical}, this means that
	$$
	\widehat{\deg}_X(\widetilde a\beta)=0
	\quad\text{for every }\beta\in\mathscr T^{d+1-p}\ad(X).
	$$
	This is the defining condition for $a\in R_X^p$.  Hence
	$\ker(\Phi)\cap S_X^p=R_X^p$.
	For $p>d+1$, \cref{prop:extreme-degrees} gives
	$\widehat N^p_{\mathrm{div},\mathrm{ad}}(X)=0$, while
	$R_X^p=S_X^p$ by definition.
	Thus $\ker(\Phi)=R_X^\bullet$, and the first isomorphism theorem gives the
	displayed quotient description.
\end{proof}

\begin{remark}
	For $0\leq p\leq d+1$, the definition of $R_X^p$ requires
	$$
	\widehat{\deg}_X(\widetilde a\beta)=0
	\quad\text{for every }\beta\in\mathscr T^{d+1-p}\ad(X).
	$$
	If this equality is required only for $\beta=\widetilde b$, with
	$b\in S_X^{d+1-p}$, the set of elements $a$ satisfying it can be larger
	than $R_X^p$.  The quotient of $S_X^\bullet$ by the resulting larger ideal
	can therefore be smaller than
	$\widehat N^\bullet_{\mathrm{div},\mathrm{ad}}(X)$.  The two quotients
	coincide exactly when vanishing for every $\widetilde b$ implies the
	displayed equality for every $\beta\in\mathscr T^{d+1-p}\ad(X)$.
\end{remark}

\begin{proposition}[Forgetful homomorphism]
	\label{prop:algebraic-shadow}
	Forgetting adelic metrics induces a canonical surjective homomorphism of
	graded $\RR$-algebras
	\begin{equation}
		\rho_{\mathrm{alg}}\colon
		\widehat N^\bullet\ad(X)
		\longrightarrow
		N^\bullet_{\mathrm{alg}}(X)_{\RR}.
		\label{eq:algebraic-shadow}
	\end{equation}
	It is uniquely characterized by
	$$
	\rho_{\mathrm{alg}}\bigl(
	\widehat c_1(\overline L)
	\bigr)
	=
	c_1(L)
	\qquad\text{and}\qquad
	\rho_{\mathrm{alg}}\bigl(
	\widehat s_i(\overline E)
	\bigr)
	=
	s_i(E).
	$$
	In particular,
	$$
	\rho_{\mathrm{alg}}\bigl(
	\widehat c_i(\overline E)
	\bigr)
	=
	c_i(E)
	\qquad(i\geq0).
	$$
\end{proposition}

\begin{proof}
	Use the formal homomorphism
	$F:\mathscr T^\bullet\ad(X)\to\mathscr T^\bullet_{\mathrm{alg}}(X)$
	defined above.
	Every algebraic line bundle on a projective variety admits an integrable
	adelic metric.
	Indeed, write the line bundle as $B\otimes A^{-1}$ with $A$ and $B$
	generated by global sections.
	Give $A$ and $B$ the pullbacks of the standard nef adelic metrics on
	$\mathcal O(1)$ from the associated projective spaces.
	Their quotient is an integrable metric on $B\otimes A^{-1}$.
	Apply this construction also to $\mathcal O_{\BP(E)}(1)$ for every
	algebraic vector bundle $E$.
	Our convention in Section~1 allows this tautological metric as the adelic
	data for $E$.
	Real linear combinations give lifts of all elements of $\Pic(X)_{\RR}$.
	Thus every generator of $\mathscr T^\bullet_{\mathrm{alg}}(X)$ has an
	adelic lift, and $F$ is surjective.
	
	It suffices to prove that
	$$
	F\bigl(\mathscr I^\bullet_{\mathrm{num}}(X)\bigr)
	\subseteq
	\mathscr I^\bullet_{\mathrm{num},\mathrm{alg}}(X).
	$$
	For $0\leq p\leq d$, let
	$\alpha\in\mathscr I^p_{\mathrm{num}}(X)$ and let
	$\gamma\in\mathscr T^{d-p}_{\mathrm{alg}}(X)$.
	Choose a lift $\overline\gamma\in\mathscr T^{d-p}\ad(X)$ with
	$F(\overline\gamma)=\gamma$.
	Let $v:X\to\Spec K$ be the structure morphism.
	Choose $\overline L$ as in \cref{lem:vertical-class} and put
	$\widetilde\varepsilon:=\lambda(v^*\overline L)$.
	The lemma and linearity give
	$$
	\deg_X\bigl(F(\alpha)\gamma\bigr)
	=
	\widehat{\deg}_X\bigl(
	\alpha\,\overline\gamma\,\widetilde\varepsilon
	\bigr)
	=
	0.
	$$
	The last equality holds because
	$\overline\gamma\widetilde\varepsilon$ has degree $d+1-p$ and
	$\alpha$ belongs to the adelic numerical radical.
	Thus $F(\alpha)$ belongs to
	$\mathscr I^p_{\mathrm{num},\mathrm{alg}}(X)$.  For $p>d$, the same inclusion
	follows from the definition of the algebraic numerical radical.
	
	Hence $F$ descends to the surjective homomorphism
	$\rho_{\mathrm{alg}}$.  The construction determines its action on the
	tautological generators, and the formula for Chern classes follows from
	the universal Segre--Chern relation.
\end{proof}

\begin{definition}[Vertical numerical ideal]
	\label{def:vertical-ideal}
	The ideal
	$$
	\widehat N^\bullet_{\mathrm{vert},\mathrm{ad}}(X)
	:=\ker(\rho_{\mathrm{alg}})
	$$
	is called the \emph{vertical numerical ideal}.  A class lies in this ideal
	exactly when its image in the algebraic numerical ring is zero.
\end{definition}

\subsection{Functoriality within the tautological category}

For a morphism $f\colon Y\to X$, pullback of adelic line bundles and vector
bundles induces a graded homomorphism
$$
f^\sharp\colon
\mathscr T^\bullet\ad(X)
\longrightarrow
\mathscr T^\bullet\ad(Y),
$$
characterized by
$$
f^\sharp\bigl(\lambda(\overline L)\bigr)
=
\lambda(f^*\overline L),
\qquad
f^\sharp\bigl(\sigma_i(\overline E)\bigr)
=
\sigma_i(f^*\overline E).
$$
This defines $f^\sharp$ on the tautological algebra.  The following condition
ensures that $f^\sharp$ descends to the numerical quotient.

\begin{definition}[Numerically admissible morphism]
	\label{def:numerically-admissible}
	A morphism $f\colon Y\to X$ is \emph{numerically admissible} if
	$$
	f^\sharp\bigl(\mathscr I^\bullet_{\mathrm{num}}(X)\bigr)
	\subseteq
	\mathscr I^\bullet_{\mathrm{num}}(Y).
	$$
\end{definition}

\begin{lemma}[Pullback]
	\label{lem:admissible-pullback}
	If $f\colon Y\to X$ is numerically admissible, then $f^\sharp$ induces a
	unique graded ring homomorphism
	\begin{equation}
		f^*\colon
		\widehat N^\bullet\ad(X)
		\longrightarrow
		\widehat N^\bullet\ad(Y),
		\qquad
		[\alpha]_{\mathrm{num}}
		\longmapsto
		[f^\sharp(\alpha)]_{\mathrm{num}}.
		\label{eq:admissible-pullback}
	\end{equation}
	Isomorphisms are numerically admissible, and numerically admissible
	morphisms are closed under composition.
\end{lemma}

\begin{proof}
	Let $\alpha,\alpha'\in\mathscr T^\bullet\ad(X)$ and suppose that
	$$
	[\alpha]_{\mathrm{num}}=[\alpha']_{\mathrm{num}}.
	$$
	Then
	$$
	\alpha-\alpha'\in\mathscr I^\bullet_{\mathrm{num}}(X).
	$$
	Since $f$ is numerically admissible,
	$$
	f^\sharp(\alpha-\alpha')
	\in\mathscr I^\bullet_{\mathrm{num}}(Y).
	$$
	Therefore
	$$
	[f^\sharp(\alpha)]_{\mathrm{num}}
	=
	[f^\sharp(\alpha')]_{\mathrm{num}}.
	$$
	Thus the formula is independent of the representative.  It defines a
	graded ring homomorphism.  Every class in $\widehat N^\bullet\ad(X)$ has
	a representative in $\mathscr T^\bullet\ad(X)$.  Hence this homomorphism
	is unique.
	
	If $f$ is an isomorphism, then $f^\sharp$ and
	$(f^{-1})^\sharp$ preserve all top-degree intersection numbers; hence they
	identify the numerical radicals.  Finally, for
	$Z\xrightarrow{g}Y\xrightarrow{f}X$, one has
	$$
	(f\circ g)^\sharp=g^\sharp\circ f^\sharp,
	$$
	and composition of the corresponding radical containments proves closure
	under composition.
\end{proof}

\begin{example}[The base field]
	\label{ex:base-field}
	For $X=\Spec K$, one has $d=0$ and
	$$
	\widehat N^\bullet\ad(\Spec K)
	\simeq\RR[\varepsilon]/(\varepsilon^2),
	\qquad
	\widehat{\deg}_K(\varepsilon)=1.
	$$
	Thus the arithmetic shift from top degree $d$ to $d+1$ is already visible
	over the base field.
\end{example}

\section{The adelic Bogomolov--Gieseker inequality over curves}
\label{sec:curve-pb}

We first prove two facts.  We use them here and again in
Section~\ref{sec:controlled-numerical-bg}.  First, on every model, an
intersection number on the projective bundle equals the arithmetic
discriminant.  This equality holds in every dimension.  Second, under the
conditions stated below, the model intersection numbers converge to the
adelic intersection number.
For curves, nonnegativity follows from Moriwaki's theorem for arithmetic
surfaces.  In the next section, it follows from Moriwaki's theorem in higher
dimensions.

Let $K$ be a number field.  Let $X$ be a smooth projective variety over $K$
of dimension $d\geq1$.  Assume that $X$ is geometrically integral.  Let $E$
be a vector bundle on $X$ of rank $r\geq2$.  If $d\geq2$, choose an ample line
bundle $L$ on $X$.  If $d=1$, choose any ample line bundle $L$.  This keeps
the notation uniform.  Since $d-1=0$, the curve invariants below do not
contain $L$.

Let
$Y:=\mathbb P_X(E):=\operatorname{Proj}_X(\operatorname{Sym}E)$ and let
$\pi:Y\to X$ be the projection.  Put $\xi:=\mathcal O_Y(1)$.  There is a
canonical surjection
\begin{equation}
 \pi^*E\twoheadrightarrow\xi.
 \label{eq:bg-projective-bundle}
\end{equation}
Let $W:=Y\times_XY$, let $p_1,p_2:W\to Y$ be the projections, and let
$q:=\pi p_1=\pi p_2$.  Thus $\dim Y=d+r-1$ and
$\dim W=d+2r-2$.

Let $\overline L$ be a nef adelic metric on $L$, and let
$\overline\xi$ be an integrable adelic metric on $\xi$.  All arithmetic
degrees and adelic intersections use the normalization fixed in
\cref{sec:numerical-adelic-chow-ring}; in particular, arithmetic degrees over
$\mathcal O_K$ are divided by $[K:\mathbb Q]$.  Only top intersection
numbers are used.  On each regular arithmetic model, we use the arithmetic
Chern classes of Gillet and Soul\'e.

We use \cref{def:cauchy-boundary-topology} for sequences that are Cauchy for
the supremum norm.
We compare model metrics through the fixed identifications on generic fibre.

Suppose that $t$ sequences satisfy this definition on the same variety $V$.
We can use one set of control data for all of them.
Let $S_a$ and $(\epsilon_{a,m})_m$ be the finite set of places and the
error bounds for the $a$-th sequence.  Set
$$
 S:=\bigcup_{a=1}^t S_a,
 \qquad \epsilon_m:=\max_{1\leq a\leq t}\epsilon_{a,m}.
$$
The set $S$ is finite, and $\epsilon_m\to0$.
For the $a$-th sequence, the logarithmic metric ratios vanish at every
place in $S\setminus S_a$.
Thus every sequence satisfies \cref{eq:cauchy-boundary-bounds} with the
same $S$ and $\epsilon_m$.
Restrict the fixed models to the common open subset $U$ of
$\operatorname{Spec}\mathcal O_K$ obtained by omitting the finite primes
in $S$.
Take the closure of the generic diagonal in their product over $U$, with
its reduced structure.
This gives an integral projective model dominating all the fixed models.
It is flat over $U$ because it is integral and dominates the Dedekind
scheme $U$.
Pull back the fixed rational line bundles to this model.
These pullbacks induce the same metrics outside $S$.

\noindent\textbf{Arithmetic ampleness.}
We use Zhang's notion of arithmetic ampleness \cite{ZhangPositive}.
We state its conditions in the form used by Moriwaki
\cite[p.~1326]{MoriwakiHigherBG}.  Let
$\mathcal X$ be a regular projective arithmetic model.  Let
$\overline{\mathcal H}$ be a smooth Hermitian line bundle on $\mathcal X$.
We call $\overline{\mathcal H}$ \emph{arithmetically ample} if the following
conditions hold.
\begin{enumerate}[label=\textup{(A\arabic*)},leftmargin=2.8em]
\item The line bundle $\mathcal H$ is relatively ample over
$\operatorname{Spec}\mathcal O_K$.
\item For every embedding $\sigma:K\hookrightarrow\mathbb C$, the Chern form
$c_1(\overline{\mathcal H}_\sigma)$ is a K\"ahler form.
\item Let $\mathcal Z\subseteq\mathcal X$ be an integral horizontal
subvariety.  Put $q:=\dim\mathcal Z$.  Then
\begin{equation*}
 \whdeg\!\left(\whc_1(\overline{\mathcal H}|_{\mathcal Z})^q\right)>0.
\end{equation*}
\end{enumerate}
The number in \textup{(A3)} is the height of $\mathcal Z$ with respect to
$\overline{\mathcal H}$.  A smooth Hermitian $\mathbb Q$-line bundle is
arithmetically ample if some positive integral multiple is an arithmetically
ample Hermitian line bundle.

Arithmetic ampleness supplies the model polarizations used in
Section~\ref{sec:controlled-numerical-bg} to apply Moriwaki's inequality
when $d\geq2$.

\noindent\textbf{Nef Hermitian line bundles.}
We use the definition in \cite[Definition~5.3(1), p.~1159]{YuanZhangHodge}.  Let
$\mathcal Y$ be a projective arithmetic model.  Let $\overline{\mathcal M}$
be a smooth Hermitian line bundle on $\mathcal Y$.  It is \emph{nef} if two
conditions hold.  First, its metric is semipositive.  Second, let
$\mathcal Z\subseteq\mathcal Y$ be any integral closed subvariety.  Put
$q:=\dim\mathcal Z$.  Then
\begin{equation*}
 \whdeg\!\left(\whc_1(\overline{\mathcal M}|_{\mathcal Z})^q\right)\geq0.
\end{equation*}
A smooth Hermitian $\mathbb Q$-line bundle is nef if some positive integral
multiple is nef.

\subsection{Controlled metric model systems}

\begin{definition}[BG-admissible model system]
\label{def:bg-admissible-model-system}
The pair $(\overline L,\overline\xi)$ is said to admit a
\emph{BG-admissible model system} if the following data and conditions are
given.

\begin{enumerate}[label=\textup{(BG\arabic*)},leftmargin=3.6em]
\item For every $m\geq1$, choose a regular, integral, projective, and flat
arithmetic model $\mathcal X_m$ of $X$.  Choose a vector bundle $\mathcal E_m$
on $\mathcal X_m$ with generic fibre $E$.  Equip $\mathcal E_m$ with a smooth
Hermitian metric $h_m$.  We require $h_m$ to be invariant under complex
conjugation.  Write
$\overline{\mathcal E}_m=(\mathcal E_m,h_m)$.  Put
\begin{equation}
 \pi_m:\mathcal Y_m=\mathbb P_{\mathcal X_m}(\mathcal E_m)
 \longrightarrow\mathcal X_m,
 \qquad
 \overline{\boldsymbol\xi}_m
 =\overline{\mathcal O}_{\mathcal Y_m}(1),
 \label{eq:bg-model-projective-bundle}
\end{equation}
Here $\overline{\boldsymbol\xi}_m$ carries the smooth quotient metric.  For
every $m$, fix isomorphisms
\begin{equation*}
 \iota_m:(\mathcal X_m)_K\xrightarrow{\sim}X,
 \qquad
 \varphi_m:(\mathcal E_m)_K\xrightarrow{\sim}\iota_m^*E.
\end{equation*}
They induce isomorphisms
\begin{equation*}
 j_m:(\mathcal Y_m)_K\xrightarrow{\sim}Y,
 \qquad
 \mathcal O_{\mathcal Y_m}(1)|_{(\mathcal Y_m)_K}
 \xrightarrow{\sim}j_m^*\xi.
\end{equation*}
We keep these maps fixed in all later comparisons.  Under the second induced
isomorphism, the quotient metric defines a model adelic metric
$\overline\xi_m$ on $\xi$.

\item If $d\geq2$, there is an arithmetically ample smooth Hermitian
$\mathbb Q$-line bundle $\overline{\mathcal L}_m$ on $\mathcal X_m$ whose
generic fibre is $L$.
The induced nef model adelic line bundles $\overline L_m$ form a Cauchy
sequence for the supremum norm.
This sequence converges to $\overline L$.

If $d=1$, no model polarization is required: the exponent $d-1$ in every
discriminant degree is zero.  In this case the regular arithmetic surfaces
$\mathcal X_m$ may be replaced, whenever two model systems are compared, by
regular common dominating models.  Pulling back the specified Hermitian
bundle and the line classes occurring below does not change their arithmetic
degrees.  We keep the chosen extension $\mathcal E_m$ fixed.

\item There are fixed $\mathbb Q$-line bundles $M^+$ and $M^-$ on $Y$ with
$\xi=M^+-M^-$ in $\Pic(Y)_{\mathbb Q}$.  For every $m$, choose a projective
arithmetic model
$\rho_m:\mathcal Y'_m\to\mathcal Y_m$.  The map $\rho_m$ induces the identity
on $Y$.  Choose smooth Hermitian $\mathbb Q$-line bundles
$\overline{\mathcal M}_m^\pm$ on $\mathcal Y'_m$.  They are nef.  Their generic
fibres are $M^\pm$.  We require
\begin{equation}
 \rho_m^*\overline{\boldsymbol\xi}_m
 =\overline{\mathcal M}_m^+-\overline{\mathcal M}_m^-
 \quad\text{in }\widehat{\Pic}(\mathcal Y'_m)_{\mathbb Q}.
 \label{eq:bg-fixed-nef-decomposition}
\end{equation}
For each sign, the associated nef model adelic line bundles
$\overline M_m^\pm$ form a Cauchy sequence for the supremum norm.
Write $\overline M^\pm$ for its limit.
Each limit $\overline M^\pm$ is nef by the definition of nefness.
Require $\overline\xi=\overline M^+-\overline M^-$.
The models
$\mathcal Y'_m$ and the models carrying the pullbacks of
$\overline L_m$ may be replaced by common dominating models in all
intersection computations.
\end{enumerate}
\end{definition}

The last condition is stronger than integrability of $\overline\xi$.  It is
the uniform control that permits every mixed term arising from
\cref{eq:bg-fixed-nef-decomposition} to pass to the limit.

\subsection{The tautological numerical discriminant}

\begin{definition}[Tautological discriminant number]
\label{def:bg-tautological-discriminant}
For a BG-admissible model system, define
\begin{equation}
\Delta^\tau_{\overline L}(E,\overline\xi)
:=(r+1)(p_1^*\overline\xi)^r(p_2^*\overline\xi)^r(q^*\overline L)^{d-1}
-2r\,\overline\xi^{\,r+1}(\pi^*\overline L)^{d-1}.
\label{eq:bg-tautological-discriminant}
\end{equation}
The first intersection is taken on $W$ and the second on $Y$.  When $d=1$,
the factors involving $\overline L$ are absent.
\end{definition}

We use the superscript $\tau$ because this number comes from the tautological
metric $\overline\xi$.

\begin{lemma}
\label{lem:bg-numerical-bridge}
Let $\overline E:=(E,\overline\xi)$ be the adelic vector bundle.  Here
$\overline\xi$ is the given integrable adelic metric on
$\xi=\mathcal O_Y(1)$.  Let $\widehat s_i(\overline E)$ be the numerical
tautological Segre class from \cref{def:segre}.  Let
$\widehat c_i(\overline E)$ be the corresponding numerical Chern class from
\cref{def:ring-characteristic-classes}.  In particular,
\begin{equation}
 \widehat c_1(\overline E)=\widehat s_1(\overline E),
 \qquad
 \widehat c_2(\overline E)
 =\widehat s_1(\overline E)^2-\widehat s_2(\overline E).
 \label{eq:bg-tautological-chern-polynomials}
\end{equation}
Then, in the numerical pairing of \cref{sec:numerical-adelic-chow-ring},
\begin{equation}
\left\langle 2r\widehat c_2(\overline E)
-(r-1)\widehat c_1(\overline E)^2,
\overline L^{d-1}\mid X\right\rangle
=\Delta^\tau_{\overline L}(E,\overline\xi).
\label{eq:bg-numerical-bridge}
\end{equation}
\end{lemma}

\begin{proof}
We omit $\overline E$ from the notation.  The universal relation in
\cref{eq:bg-tautological-chern-polynomials} gives
\begin{equation}
 2r\widehat c_2-(r-1)\widehat c_1^2
 =(r+1)\widehat s_1^2-2r\widehat s_2.
 \label{eq:bg-formal-discriminant}
\end{equation}
In \cref{eq:segre-monomial}, take $Z=X$ and $\lambda=(1,1)$.  Then
$P_{X,\lambda}=W$.  The two tautological line bundles are
$p_1^*\overline\xi$ and $p_2^*\overline\xi$.  Each tautological exponent is
$r-1+1=r$.  Therefore
\begin{equation*}
\left\langle \widehat s_1(\overline E)^2,\overline L^{d-1}\mid X\right\rangle
=\widehat{\deg}_W\!\left(\widehat c_1(p_1^*\overline\xi)^r\widehat c_1(p_2^*\overline\xi)^r\widehat c_1(q^*\overline L)^{d-1}\right).
\end{equation*}
In \cref{eq:segre-definition}, take $Z=X$ and $i=2$.  Then
\begin{equation*}
\left\langle \widehat s_2(\overline E),\overline L^{d-1}\mid X\right\rangle
=\widehat{\deg}_Y\!\left(\widehat c_1(\overline\xi)^{r+1}\widehat c_1(\pi^*\overline L)^{d-1}\right).
\end{equation*}
Apply the pairing with $\overline L^{d-1}$ to
\cref{eq:bg-formal-discriminant}.  Substitute the two identities above.  By
\cref{eq:bg-tautological-discriminant}, the resulting expression is
$\Delta^\tau_{\overline L}(E,\overline\xi)$.  This proves
\cref{eq:bg-numerical-bridge}.
\end{proof}

The model calculation below proves that the first two Bott--Chern
correction terms cancel in the discriminant.

\subsection{The model Bott--Chern identity}

Fix an index $m$ and temporarily omit it.  Thus $\mathcal X$ is regular,
$\overline{\mathcal E}=(\mathcal E,h)$ is a smooth Hermitian bundle of rank
$r$, and
\begin{equation}
 \widehat\pi:\mathcal Y=\mathbb P_{\mathcal X}(\mathcal E)
 \longrightarrow\mathcal X,
 \qquad
 \widehat\xi=\widehat c_1(\overline{\mathcal O}_{\mathcal Y}(1)).
 \label{eq:bg-fixed-model}
\end{equation}
Define the tautological arithmetic Segre classes on the model by
\begin{equation}
 \widehat s_i(\overline{\mathcal E})
 :=\widehat\pi_*(\widehat\xi^{\,r-1+i})
 \in\widehat{\mathrm{CH}}^i(\mathcal X)_{\mathbb Q}.
 \label{eq:bg-raw-segre}
\end{equation}
We use the Segre notation from \cref{def:segre} for these model classes.
When the adelic line bundles used in the pairing come from Hermitian line
bundles on $\mathcal X$, the arithmetic projection formula identifies
the normalized top intersections of $\widehat s_i(\overline{\mathcal E})$
with the evaluations on $X$ in \cref{def:segre} for the induced adelic metrics.
Let $F_\infty$ denote complex conjugation on $\mathcal X(\mathbb C)$.  For
$p\geq0$, let $A^{p,p}(\mathcal X)$ be the space of smooth real differential
forms $\alpha$ of type $(p,p)$ satisfying
$F_\infty^*\alpha=(-1)^p\alpha$.  These are the differential forms invariant
under complex conjugation in the real structure used here.  Define
\begin{equation*}
 \widetilde A^{p,p}(\mathcal X)
 :=A^{p,p}(\mathcal X)/(\operatorname{im}\partial+\operatorname{im}\overline\partial).
\end{equation*}
The two images in the denominator are taken in bidegree $(p,p)$.  For
$p\geq1$, define
\begin{equation*}
 a_p:\widetilde A^{p-1,p-1}(\mathcal X)
 \longrightarrow\widehat{\mathrm{CH}}^p(\mathcal X)_{\mathbb Q}.
\end{equation*}
For $\eta\in\widetilde A^{p-1,p-1}(\mathcal X)$, put
$a_p(\eta):=[(0,\eta)]$.  Define
\begin{equation*}
 \omega_p:\widehat{\mathrm{CH}}^p(\mathcal X)_{\mathbb Q}
 \longrightarrow A^{p,p}(\mathcal X).
\end{equation*}
If $x$ is represented by an arithmetic cycle $(Z,g_Z)$, put
\begin{equation*}
 \omega_p(x):=\delta_Z+dd^c g_Z.
\end{equation*}
Here $\delta_Z$ is the current of integration over $Z(\mathbb C)$.  The Green
current equation shows that $\omega_p(x)$ is a smooth differential form.  Let
$p,q\geq1$.  Let
$x\in\widehat{\mathrm{CH}}^p(\mathcal X)_{\mathbb Q}$.  Let
$\theta\in\widetilde A^{p-1,p-1}(\mathcal X)$.  Let
$\eta\in\widetilde A^{q-1,q-1}(\mathcal X)$.  Then
\begin{equation}
 x\,a_q(\eta)=a_{p+q}\bigl(\omega_p(x)\wedge\eta\bigr),\qquad
 a_p(\theta)a_q(\eta)=a_{p+q}\bigl(dd^c\theta\wedge\eta\bigr),
 \label{eq:bg-form-ideal-product}
\end{equation}
These identities are the product formulas in \cite[Section~4.2.11]{GS}.

Put $h_j=\sum_{\ell=1}^j\ell^{-1}$ and
$\kappa_r=\sum_{j=1}^{r-1}h_j$.  Here
$\widehat c_t^{\mathrm{GS}}(\overline{\mathcal E})$ is the Chern
polynomial defined in \cref{sec:bott-chern}.  Its coefficient of $t^i$ is the arithmetic
Chern class of degree $i$ defined by Gillet and Soul\'e.  The superscript
$\mathrm{GS}$ distinguishes these classes on $\mathcal X$ from the numerical
classes $\widehat c_i(\overline E)$ defined earlier.
Mourougane constructs corrected Segre classes
$\widehat s_i^{\mathrm{GS}}=\widehat s_i+a(R_i)$ for $i\geq1$.
Set $\widehat s_0^{\mathrm{GS}}=1$.
Define $\widehat s_t^{\mathrm{GS}}(\overline{\mathcal E}):=\sum_{i\geq0}\widehat s_i^{\mathrm{GS}}t^i$.
We use the dual metric on $\mathcal E^{\vee}$.  Then
$\widehat c_i^{\mathrm{GS}}(\overline{\mathcal E}^{\vee})
=(-1)^i\widehat c_i^{\mathrm{GS}}(\overline{\mathcal E})$.
We replace $t$ by $-t$ in \cite[Theorem~4 in Section~8.1]{Mourougane}.
The duality identity gives
\begin{equation}
 \widehat c_t^{\mathrm{GS}}(\overline{\mathcal E})
 \widehat s_{-t}^{\mathrm{GS}}(\overline{\mathcal E})=1.
 \label{eq:bg-mourougane-inverse}
\end{equation}
Mourougane computes the first two correction forms as
\begin{equation}
 R_1=-\kappa_r,
 \qquad
 R_2=-\left(1+\frac1r\right)\kappa_r c_1(\mathcal E,h);
 \label{eq:bg-mourougane-corrections}
\end{equation}
see \cite[Section~8.2]{Mourougane}.  We compare the coefficients of $t$ and $t^2$
in \cref{eq:bg-mourougane-inverse}.  We then substitute the correction forms
from \cref{eq:bg-mourougane-corrections}.  This gives
\begin{align}
 \widehat s_1
 &=\widehat c_1^{\mathrm{GS}}(\overline{\mathcal E})+a(\kappa_r),
 \label{eq:bg-raw-segre-one}\\
 \widehat s_2
 &=\widehat c_1^{\mathrm{GS}}(\overline{\mathcal E})^2
 -\widehat c_2^{\mathrm{GS}}(\overline{\mathcal E})
 +a\left(\left(1+\frac1r\right)
 \kappa_r c_1(\mathcal E,h)\right).
 \label{eq:bg-raw-segre-two}
\end{align}

\begin{proposition}[Bott--Chern cancellation]
\label{prop:bg-bott-chern-cancellation}
In $\widehat{\mathrm{CH}}^2(\mathcal X)_{\mathbb Q}$ one has
\begin{equation}
 (r+1)(\widehat s_1)^2-2r\widehat s_2
 =2r\widehat c_2^{\mathrm{GS}}(\overline{\mathcal E})
 -(r-1)\widehat c_1^{\mathrm{GS}}(\overline{\mathcal E})^2.
 \label{eq:bg-bott-chern-cancellation}
\end{equation}
\end{proposition}

\begin{proof}
Since $\kappa_r$ is constant, \cref{eq:bg-form-ideal-product} gives
$a(\kappa_r)^2=0$.
The map $\omega$ sends the first arithmetic Chern class to the first Chern form:
$$
 \omega\bigl(\widehat c_1^{\mathrm{GS}}(\overline{\mathcal E})\bigr)
 =c_1(\mathcal E,h).
$$
Apply the first identity in \cref{eq:bg-form-ideal-product} with
$p=q=1$, $x=\widehat c_1^{\mathrm{GS}}(\overline{\mathcal E})$, and
$\eta=\kappa_r$. This gives
\begin{equation}
 \widehat c_1^{\mathrm{GS}}(\overline{\mathcal E})a(\kappa_r)
 =a(\kappa_r c_1(\mathcal E,h)).
 \label{eq:bg-constant-form-product}
\end{equation}
Substitution of \cref{eq:bg-raw-segre-one,eq:bg-raw-segre-two} into the left
side of \cref{eq:bg-bott-chern-cancellation} shows that the secondary term
from the square has coefficient $2(r+1)\kappa_r$, while the secondary term
from $2r\widehat s_2$ has coefficient
\begin{equation}
 2r\left(1+\frac1r\right)\kappa_r=2(r+1)\kappa_r.
 \label{eq:bg-correction-coefficient}
\end{equation}
The signs are opposite.  The remaining primary terms equal
$$
(r+1)\widehat c_1^{\mathrm{GS}}(\overline{\mathcal E})^2
-2r\left(\widehat c_1^{\mathrm{GS}}(\overline{\mathcal E})^2
-\widehat c_2^{\mathrm{GS}}(\overline{\mathcal E})\right)
=2r\widehat c_2^{\mathrm{GS}}(\overline{\mathcal E})
-(r-1)\widehat c_1^{\mathrm{GS}}(\overline{\mathcal E})^2,
$$
which proves the proposition.
\end{proof}

\subsection{Fiber-square realization and the exact model identity}

Return to the $m$-th model.  Let
$\mathcal W_m:=\mathcal Y_m\times_{\mathcal X_m}\mathcal Y_m$, let
$p_{1,m},p_{2,m}:\mathcal W_m\to\mathcal Y_m$ be the projections, and let
$q_m:=\pi_mp_{1,m}=\pi_mp_{2,m}$.  Projective bundles are smooth over their
bases, so $\mathcal Y_m$ and $\mathcal W_m$ are regular.

\begin{lemma}
\label{lem:bg-fiber-square-realization}
With $\widehat\xi_m=\widehat c_1(\overline{\boldsymbol\xi}_m)$, one has
\begin{equation}
 (q_m)_*\left(
 (p_{1,m}^*\widehat\xi_m)^r(p_{2,m}^*\widehat\xi_m)^r
 \right)
 =(\widehat s_{1,m})^2,
 \qquad
 (\pi_m)_*(\widehat\xi_m^{\,r+1})=\widehat s_{2,m}.
 \label{eq:bg-fiber-square-realization}
\end{equation}
\end{lemma}

\begin{proof}
The second identity is the definition of $\widehat s_{2,m}$.  Put
$\widehat\alpha_m=\widehat\xi_m^r$.  Consider the Cartesian square
\[
\begin{array}{ccc}
 \mathcal W_m&\xrightarrow{\ p_{2,m}\ }&\mathcal Y_m\\
 {\scriptstyle p_{1,m}}\downarrow&&\downarrow{\scriptstyle\pi_m}\\
 \mathcal Y_m&\xrightarrow{\ \pi_m\ }&\mathcal X_m
\end{array}
\]
The morphism $\pi_m$ is projective and smooth.
Thus it is proper and flat.
We will prove the identity
\begin{equation}
 (p_{1,m})_*p_{2,m}^*(\widehat\alpha_m)
 =\pi_m^*(\pi_m)_*(\widehat\alpha_m)
 \quad\text{in }\widehat{\mathrm{CH}}^1(\mathcal Y_m)_{\mathbb Q}.
 \label{eq:bg-arithmetic-base-change}
\end{equation}
Choose an arithmetic cycle $(A,g_A)$ representing $\widehat\alpha_m$.
We first check the equality on the algebraic cycle $A$.
By linearity, it suffices to consider an integral closed subscheme
$Z\subset\mathcal Y_m$ of codimension $r$.
Let $V$ be the image $\pi_m(Z)$ with its reduced structure.
The image is closed because $\pi_m$ is proper.
Put
$$
 Z':=\mathcal Y_m\times_{\mathcal X_m}Z,\qquad
 V':=\mathcal Y_m\times_{\mathcal X_m}V.
$$
The morphisms $Z'\to Z$ and $V'\to V$ are projective bundles of relative
dimension $r-1$.
In particular, $Z'$ and $V'$ are integral.
The definition of flat pullback gives
$$
 p_{2,m}^*[Z]=[Z'],\qquad \pi_m^*[V]=[V'].
$$
The map $p_{1,m}$ sends $Z'$ onto $V'$.
If $\dim V<\dim Z$, then $\dim V'<\dim Z'$.
The definition of proper pushforward therefore gives
$(\pi_m)_*[Z]=0$ and $(p_{1,m})_*[Z']=0$.
Thus both sides of the required equality vanish.

Suppose now that $\dim V=\dim Z$.
Set $n=[k(Z):k(V)]$.
The definition of proper pushforward gives $(\pi_m)_*[Z]=n[V]$.
Choose a trivialization of $\mathcal E_m$ near the generic point of $V$.
For indeterminates $t_1,\ldots,t_{r-1}$, this identifies the induced
extension of function fields with
$$
 k(V')=k(V)(t_1,\ldots,t_{r-1})
 \ \subset\
 k(Z')=k(Z)(t_1,\ldots,t_{r-1}).
$$
A purely transcendental extension preserves the degree $n$.
Hence $(p_{1,m})_*[Z']=n[V']$, and
$$
 (p_{1,m})_*p_{2,m}^*[Z]
 =n[V']
 =\pi_m^*(\pi_m)_*[Z].
$$
This proves the equality on $A$.
The argument also applies when $Z$ lies in a fibre over a closed point of
$\operatorname{Spec}\mathcal O_K$.
Thus it includes vertical cycles.

We next check the equality on the Green current $g_A$.
This current has bidegree $(r-1,r-1)$.
Choose an open set $U\subset\mathcal X_m(\mathbb C)$ on which
$\mathcal E_m$ is trivial.
Put $P=\mathbb P^{r-1}(\mathbb C)$.
The trivialization gives
$$
 Y_U:=\pi_m^{-1}(U)\simeq U\times P,\qquad
 W_U:=q_m^{-1}(U)\simeq U\times P\times P.
$$
We use the same symbols for the induced maps of complex manifolds.
In these coordinates,
$$
 \pi_m(x,u)=x,\qquad
 p_{1,m}(x,u,v)=(x,u),\qquad
 p_{2,m}(x,u,v)=(x,v).
$$

First let $\eta$ be a smooth form of bidegree $(r-1,r-1)$ on $Y_U$.
For $x\in U$, define $i_x:P\to Y_U$ by $i_x(v)=(x,v)$.
The fibre of $p_{1,m}$ over $(x,u)$ is $\{x\}\times\{u\}\times P$.
The restriction of $p_{2,m}^*\eta$ to this fibre is $i_x^*\eta$.
Integration over this fibre therefore gives
$$
 \bigl((p_{1,m})_*p_{2,m}^*\eta\bigr)(x,u)
 =\int_P i_x^*\eta
 =\bigl((\pi_m)_*\eta\bigr)(x).
$$
The last expression does not depend on $u$.
It is the value of $\pi_m^*(\pi_m)_*\eta$ at $(x,u)$.
This proves the required equality for smooth forms.

We now pass to currents.
We identify each smooth form with its associated current.
Choose smooth forms $\eta_j$ of bidegree $(r-1,r-1)$ that converge to
$g_A|_{Y_U}$ in the space of currents.
Such forms exist by \cite[Section~1.1.3]{GS}.
The maps $p_{2,m}$ and $\pi_m$ are submersions.
Their pullbacks on currents are continuous.
The maps $p_{1,m}$ and $\pi_m$ are projections from products with the compact space $P$.
Hence they are proper.
Their pushforwards on currents are continuous.
These continuity properties follow from the definitions by duality in
\cite[Section~1.1.4]{GS}.
Apply the equality for smooth forms to each $\eta_j$.
Passing to the limit gives
$$
 (p_{1,m})_*p_{2,m}^*(g_A|_{Y_U})
 =\pi_m^*(\pi_m)_*(g_A|_{Y_U}).
$$
The open sets $U$ cover $\mathcal X_m(\mathbb C)$.
Hence the same equality holds on $\mathcal Y_m(\mathbb C)$.
By \cite[Theorem~3.6.1]{GS}, these operations respect the relations defining
arithmetic Chow groups.
Together with the equality on $A$, this proves
\cref{eq:bg-arithmetic-base-change}.

Functoriality, \cref{eq:bg-arithmetic-base-change}, and the arithmetic
projection formula now give
\begin{align}
 &(q_m)_*\left(
 p_{1,m}^*\widehat\alpha_m\,p_{2,m}^*\widehat\alpha_m
 \right)\notag\\
 &\quad=(\pi_m)_*\left(
 \widehat\alpha_m\,(p_{1,m})_*p_{2,m}^*\widehat\alpha_m
 \right)\notag\\
 &\quad=(\pi_m)_*\left(
 \widehat\alpha_m\,\pi_m^*(\pi_m)_*\widehat\alpha_m
 \right)
 =\left((\pi_m)_*\widehat\alpha_m\right)^2
 =(\widehat s_{1,m})^2.
 \label{eq:bg-fiber-square-proof}
\end{align}
\end{proof}

Define a class of codimension $d-1$ on $\mathcal X_m$ by
\begin{equation}
 \widehat\Lambda_m=
 \begin{cases}
 \widehat c_1(\overline{\mathcal L}_m)^{d-1},&d\geq2,\\
 1,&d=1.
 \end{cases}
 \label{eq:bg-model-polarization-factor}
\end{equation}
The arithmetic model $\mathcal X_m$ has dimension $d+1$.
Multiplying a class of codimension $2$ by $\widehat\Lambda_m$ gives a class
of codimension $2+(d-1)=d+1$.
We can therefore take the arithmetic degree of the product.
Define
\begin{equation}
 D_m:=\whdeg_K\left(
 \left(
 2r\widehat c_2^{\mathrm{GS}}(\overline{\mathcal E}_m)
 -(r-1)\widehat c_1^{\mathrm{GS}}(\overline{\mathcal E}_m)^2
 \right)\widehat\Lambda_m
 \right)\in\mathbb R.
 \label{eq:bg-model-discriminant}
\end{equation}
We obtain $\whdeg_K$ by dividing the arithmetic degree over
$\operatorname{Spec}\mathcal O_K$ by $[K:\mathbb Q]$.
Let $B_m$ be the model version of
\cref{eq:bg-tautological-discriminant}, namely
\begin{equation}
B_m:=(r+1)\whdeg_K\!\left((p_{1,m}^*\widehat\xi_m)^r
(p_{2,m}^*\widehat\xi_m)^r q_m^*\widehat\Lambda_m\right)
-2r\whdeg_K\!\left(\widehat\xi_m^{\,r+1}\pi_m^*\widehat\Lambda_m\right).
\label{eq:bg-model-tautological-number}
\end{equation}

\begin{proposition}[Exact model identity]
\label{prop:bg-exact-model-identity}
For every $m$, one has
\begin{equation}
 B_m=D_m.
 \label{eq:bg-exact-model-identity}
\end{equation}
\end{proposition}

\begin{proof}
Put
$$
 \widehat\gamma_m:=
 (p_{1,m}^*\widehat\xi_m)^r(p_{2,m}^*\widehat\xi_m)^r
 \in\widehat{\mathrm{CH}}^{2r}(\mathcal W_m)_{\mathbb Q}.
$$
The first identity in \cref{lem:bg-fiber-square-realization} gives
$(q_m)_*\widehat\gamma_m=(\widehat s_{1,m})^2$.
Apply the arithmetic projection formula to $q_m$.
We obtain
$$
 (q_m)_*(\widehat\gamma_m q_m^*\widehat\Lambda_m)
 =\bigl((q_m)_*\widehat\gamma_m\bigr)\widehat\Lambda_m
 =(\widehat s_{1,m})^2\widehat\Lambda_m.
$$
The second identity in \cref{lem:bg-fiber-square-realization} gives
$(\pi_m)_*(\widehat\xi_m^{\,r+1})=\widehat s_{2,m}$.
The arithmetic projection formula for $\pi_m$ therefore gives
$$
 (\pi_m)_*(\widehat\xi_m^{\,r+1}\pi_m^*\widehat\Lambda_m)
 =\bigl((\pi_m)_*\widehat\xi_m^{\,r+1}\bigr)\widehat\Lambda_m
 =\widehat s_{2,m}\widehat\Lambda_m.
$$
The classes inside these pushforwards have top codimension on
$\mathcal W_m$ and $\mathcal Y_m$, respectively.
Their images have codimension $d+1$ on $\mathcal X_m$.
Functoriality of pushforward to $\operatorname{Spec}\mathcal O_K$ shows
that these pushforwards preserve arithmetic degree.
Dividing every degree by the same number $[K:\mathbb Q]$ preserves these
equalities.
Taking normalized arithmetic degrees and using
\cref{eq:bg-model-tautological-number}, we obtain
$$
 B_m=(r+1)\whdeg_K\!\left((\widehat s_{1,m})^2\widehat\Lambda_m\right)
 -2r\whdeg_K\!\left(\widehat s_{2,m}\widehat\Lambda_m\right).
$$
The arithmetic degree is linear.
Hence
$$
 B_m=\whdeg_K\!\left(
 \bigl((r+1)(\widehat s_{1,m})^2-2r\widehat s_{2,m}\bigr)
 \widehat\Lambda_m\right).
$$
Apply \cref{eq:bg-bott-chern-cancellation} to
$\overline{\mathcal E}_m$ on $\mathcal X_m$.
Substitution into the preceding formula gives
$$
 B_m=\whdeg_K\!\left(
 \bigl(2r\widehat c_2^{\mathrm{GS}}(\overline{\mathcal E}_m)
 -(r-1)\widehat c_1^{\mathrm{GS}}(\overline{\mathcal E}_m)^2\bigr)
 \widehat\Lambda_m\right)=D_m.
$$
The last equality is the definition in \cref{eq:bg-model-discriminant}.
\end{proof}

This proposition is purely an identity.  Positivity is deliberately kept
separate: it comes from Moriwaki's arithmetic-surface theorem below when
$d=1$, and from his higher-dimensional theorem in Section~\ref{sec:controlled-numerical-bg} when $d\geq2$.

\subsection{Passage to the adelic limit}

\begin{proposition}
\label{prop:bg-controlled-convergence}
For every BG-admissible model system, one has
\begin{equation}
 \lim_{m\to\infty}B_m
 =\Delta^\tau_{\overline L}(E,\overline\xi).
 \label{eq:bg-controlled-limit}
\end{equation}
\end{proposition}

\begin{proof}
When $d=1$, omit all factors involving $\overline L_m$ or $\overline L$
in the formulas below.
No choice of $\overline L_m$ is needed in this case.
On $W$, put
$$
 I_m:=(p_1^*\overline\xi_m)^r(p_2^*\overline\xi_m)^r
 (q^*\overline L_m)^{d-1}.
$$
On $Y$, put
$$
 J_m:=\overline\xi_m^{\,r+1}(\pi^*\overline L_m)^{d-1}.
$$
For model adelic line bundles, these intersections equal the normalized
arithmetic degrees on the models that define them.
Thus \cref{eq:bg-model-tautological-number} gives
$$
 B_m=(r+1)I_m-2rJ_m.
$$
For each finite comparison of models, choose a common dominating model
carrying all the line bundles in that comparison.
The projection formula preserves $I_m$ and $J_m$ under these pullbacks.

Use the equality of arithmetic Picard classes in \textup{(BG3)} to replace
each occurrence of $\overline\xi_m$ in $I_m$ and $J_m$ by
$\overline M_m^+-\overline M_m^-$.
This substitution preserves the numerical intersections.
Expand each intersection by multilinearity.
The expansion of $I_m$ contains mixed intersections of
$p_1^*\overline M_m^\pm$, $p_2^*\overline M_m^\pm$, and
$q^*\overline L_m$.
The expansion of $J_m$ contains mixed intersections of
$\overline M_m^\pm$ and $\pi^*\overline L_m$.
Both expansions are finite, and their coefficients do not depend on $m$.

Condition \textup{(BG3)} gives Cauchy sequences of nef model metrics on the
fixed line bundles $M^+$ and $M^-$.
Their limits are $\overline M^+$ and $\overline M^-$.
When $d\geq2$, condition \textup{(BG2)} gives a Cauchy sequence of nef model
metrics $\overline L_m$ with limit $\overline L$.
Pullback preserves nefness.
After enlarging the finite set of places to extend the morphisms to the
fixed models, pullback preserves the model metrics outside that set.
Pullback does not increase the supremum of a logarithmic metric ratio.
It therefore preserves the bounds in \cref{eq:cauchy-boundary-bounds}.
For each intersection, use the common finite set of places and error bounds
constructed at the beginning of \cref{sec:curve-pb}.
These choices give simultaneous control of the factors in every mixed term.

Every mixed term on $W$ has $2r+d-1=\dim W+1$ factors.
Every mixed term on $Y$ has $r+d=\dim Y+1$ factors.
Apply \cref{eq:controlled-intersection-continuity} to each mixed term.
Then take the finite linear combinations of these limits.
Condition \textup{(BG3)} also gives
$$
 \overline\xi=\overline M^+-\overline M^-.
$$
Multilinearity therefore identifies the first limit as
$$
 \lim_{m\to\infty}I_m
 =(p_1^*\overline\xi)^r(p_2^*\overline\xi)^r(q^*\overline L)^{d-1}.
$$
It identifies the second limit as
$$
 \lim_{m\to\infty}J_m
 =\overline\xi^{\,r+1}(\pi^*\overline L)^{d-1}.
$$
Both limits exist.
We can therefore take the limit of their fixed linear combination.
By \cref{eq:bg-tautological-discriminant},
\begin{equation*}
 \lim_{m\to\infty}B_m
 =(r+1)\lim_{m\to\infty}I_m-2r\lim_{m\to\infty}J_m
 =\Delta^\tau_{\overline L}(E,\overline\xi).\qedhere
\end{equation*}
\end{proof}

\begin{proposition}
\label{prop:bg-controlled-transfer}
\label{cor:bg-equality-gap}
For every BG-admissible model system, one has
\begin{equation}
 \Delta^\tau_{\overline L}(E,\overline\xi)
 =\lim_{m\to\infty}B_m
 =\lim_{m\to\infty}D_m.
 \label{eq:bg-controlled-transfer}
\end{equation}
If $D_m\geq0$ for every $m$, then
\begin{equation}
 \Delta^\tau_{\overline L}(E,\overline\xi)\geq0.
 \label{eq:bg-transfer-nonnegative}
\end{equation}
Under the same nonnegativity hypothesis, equality holds in
\cref{eq:bg-transfer-nonnegative} if and only if $D_m\to0$.  If there are
$\varepsilon>0$ and $m_0$ such that $D_m\geq\varepsilon$ for every
$m\geq m_0$, then
\begin{equation}
 \Delta^\tau_{\overline L}(E,\overline\xi)
 \geq\varepsilon>0.
 \label{eq:bg-uniform-gap}
\end{equation}
\end{proposition}

\begin{proof}
\Cref{prop:bg-controlled-convergence} shows that $B_m$ converges to
$\Delta^\tau_{\overline L}(E,\overline\xi)$.
For every $m$, \cref{prop:bg-exact-model-identity} gives $D_m=B_m$.
Hence $D_m$ also converges.
Its limit satisfies
$$
 \lim_{m\to\infty}D_m
 =\lim_{m\to\infty}B_m
 =\Delta^\tau_{\overline L}(E,\overline\xi).
$$
This proves \cref{eq:bg-controlled-transfer}.

Assume that $D_m\geq0$ for every $m$.
The limit of a convergent sequence of nonnegative real numbers is nonnegative.
Therefore,
$$
 \Delta^\tau_{\overline L}(E,\overline\xi)
 =\lim_{m\to\infty}D_m\geq0.
$$
If $\Delta^\tau_{\overline L}(E,\overline\xi)=0$, the limit identity above
gives $D_m\to0$.
Conversely, suppose that $D_m\to0$.
The uniqueness of the limit gives
$\Delta^\tau_{\overline L}(E,\overline\xi)=0$.

Finally, assume that $D_m\geq\varepsilon>0$ for every $m\geq m_0$.
Then $D_m-\varepsilon\geq0$ for every $m\geq m_0$.
Removing the first $m_0-1$ terms does not change the limit.
Applying the same fact about nonnegative sequences gives
$$
 \Delta^\tau_{\overline L}(E,\overline\xi)
 =\varepsilon+\lim_{m\to\infty}(D_m-\varepsilon)
 \geq\varepsilon>0.
$$
\end{proof}

\subsection{Moriwaki's inequality on arithmetic surfaces}

We now assume $d=1$.
Write $C=X$.
The exponent $d-1$ in \cref{eq:bg-tautological-discriminant} is zero.
We therefore take each factor involving $\overline L$ to be the
multiplicative unit $1$ in the intersection product.
Thus $\Delta^\tau_{\overline L}(E,\overline\xi)$ does not depend on $\overline L$.
On each model, \cref{eq:bg-model-polarization-factor} sets $\widehat\Lambda_m=1$.
Hence the definitions of $B_m$ and $D_m$ require no Hermitian model line
bundle $\overline{\mathcal L}_m$.
Condition \textup{(BG2)} requires no such line bundle when $d=1$.
Conditions \textup{(BG1)} and \textup{(BG3)} do not involve $L$.
For a curve, we can therefore define a BG admissible model system without
choosing a polarization line bundle $L$.

\begin{theorem}[Moriwaki's inequality on arithmetic surfaces]
\label{thm:bg-moriwaki-surface-input}
Let $\mathcal C$ be regular, integral, projective, and flat over
$\mathcal O_K$, and let $\overline{\mathcal E}$ be a smooth Hermitian vector
bundle of rank $r$ on $\mathcal C$.  If the geometric generic fibre of
$\mathcal E$ is slope semistable, then
\begin{equation}
 \whdeg_K\left(
 2r\widehat c_2^{\mathrm{GS}}(\overline{\mathcal E})
 -(r-1)\widehat c_1^{\mathrm{GS}}(\overline{\mathcal E})^2
 \right)\geq0.
 \label{eq:bg-moriwaki-surface-input}
\end{equation}
\end{theorem}

\begin{proof}
This is the main theorem of \cite{MoriwakiSurface}, multiplied by $2r$ and
divided by the positive number $[K:\mathbb Q]$.  Regularity is imposed here
in order to use the Gillet--Soul\'e arithmetic Chow ring and the Bott--Chern
comparison above.  All sheaves used here are locally free.
\end{proof}

\begin{proposition}
\label{prop:bg-curve-model-nonnegativity}
Assume that $E_{\overline K}$ is slope semistable.  For every model in a BG admissible model system over a curve, one has
\begin{equation}
 B_m=D_m\geq0.
 \label{eq:bg-curve-model-nonnegative}
\end{equation}
\end{proposition}

\begin{proof}
Fix $m$.
Condition \textup{(BG1)} gives a regular, integral, projective, and flat
arithmetic model $\mathcal X_m$ of $C$.
Since $\dim C=1$, the model $\mathcal X_m$ has dimension $2$.
Condition \textup{(BG1)} also gives a smooth Hermitian vector bundle
$\overline{\mathcal E}_m$ on $\mathcal X_m$.
Its geometric generic fibre is $E_{\overline K}$.
Our hypothesis says that $E_{\overline K}$ is slope semistable.
Put
$$
 \widehat\beta_m
 :=2r\widehat c_2^{\mathrm{GS}}(\overline{\mathcal E}_m)
 -(r-1)\widehat c_1^{\mathrm{GS}}(\overline{\mathcal E}_m)^2
 \in\widehat{\mathrm{CH}}^{\,2}(\mathcal X_m)_{\mathbb Q}.
$$
Since $d=1$, \cref{eq:bg-model-polarization-factor} gives $\widehat\Lambda_m=1$.
Thus \cref{eq:bg-model-discriminant} gives
$D_m=\whdeg_K(\widehat\beta_m)$.
We apply \cref{thm:bg-moriwaki-surface-input} to $\overline{\mathcal E}_m$.
We obtain $D_m\geq0$.
Now \cref{prop:bg-exact-model-identity} gives $B_m=D_m\geq0$.

To compare models, let $f:\mathcal Z\to\mathcal X_m$ be a projective
morphism between regular arithmetic models of $C$.
Assume that $f_K=\mathrm{id}_C$ under the fixed isomorphisms on the generic fibres.
Equip $f^*\mathcal E_m$ with the metric $f^*h_m$.
Arithmetic Chern classes commute with pullback.
Hence the discriminant class of $f^*\overline{\mathcal E}_m$ is
$f^*\widehat\beta_m$.
The identity $f_K=\mathrm{id}_C$ shows that $f$ is smooth on the generic fibres.
It also shows that $f$ has generic degree one.
Consequently, $f_*(1)=1$.
The arithmetic projection formula
\cite[Theorem~4.4.3(7)]{GS} gives
$$
 f_*f^*\widehat\beta_m
 =\widehat\beta_m\,f_*(1)
 =\widehat\beta_m.
$$
We push this identity forward to $\operatorname{Spec}\mathcal O_K$.
We divide the resulting arithmetic degrees by $[K:\mathbb Q]$.
Functoriality of pushforward gives
$$
 \whdeg_K(f^*\widehat\beta_m)
 =\whdeg_K(\widehat\beta_m)
 =D_m.
$$
Thus pulling back the specified Hermitian bundle preserves $D_m$.
\end{proof}

Define the curve discriminant by
\begin{equation}
 \Delta_C^\tau(E,\overline\xi)
 :=\Delta^\tau_{\overline L}(E,\overline\xi);
 \label{eq:curve-numerical-discriminant}
\end{equation}
the right side is independent of $\overline L$ because $d-1=0$.

\begin{theorem}[Controlled adelic Bogomolov--Gieseker inequality over a curve]
\label{thm:curve-controlled-bg}
Let $C/K$ be smooth, projective, and geometrically connected.  Let $E$ have
rank $r\geq2$ and assume that $E_{\overline K}$ is slope semistable.  Let
$\overline\xi$ be an integrable adelic quotient metric on
$\mathcal O_{\mathbb P_C(E)}(1)$.  To use the notation of
\cref{def:bg-admissible-model-system}, choose an ample line bundle $L$ on $C$.
Equip $L$ with a nef adelic metric $\overline L$.  If
$(\overline L,\overline\xi)$ admits a BG admissible model system over $C$, then
\begin{equation}
 \Delta_C^\tau(E,\overline\xi)
 =\lim_{m\to\infty}D_m\geq0.
 \label{eq:curve-controlled-bg}
\end{equation}
Equivalently, the specialization to $d=1$ of the numerical pairing in
\cref{eq:bg-numerical-bridge} is nonnegative.
\end{theorem}

\begin{proof}
Choose a BG admissible model system as in the hypothesis.
\Cref{prop:bg-curve-model-nonnegativity} gives $D_m\geq0$ for every $m$.
\Cref{prop:bg-controlled-transfer} shows that the real sequence $(D_m)$
converges to $\Delta^\tau_{\overline L}(E,\overline\xi)$.
The definition in \cref{eq:curve-numerical-discriminant} identifies this
limit with $\Delta_C^\tau(E,\overline\xi)$.
The limit of a convergent sequence of nonnegative real numbers is nonnegative.
Therefore,
$$
 \Delta_C^\tau(E,\overline\xi)=\lim_{m\to\infty}D_m\geq0.
$$
For $d=1$, \cref{eq:bg-numerical-bridge} identifies this number with the
stated numerical pairing.
\end{proof}

\subsection{The Deligne line and its algebraic normalization}

Retain the assumptions of \cref{thm:curve-controlled-bg}.
We recall the adelic Deligne pairing needed for the curve refinement.
Let $V$ and $B$ be projective integral $K$-varieties, with $B$ normal.
Let $f:V\to B$ be projective and flat of pure relative dimension $n$.
For integrable adelic line bundles $\overline M_0,\ldots,\overline M_n$ on
$V$, Yuan and Zhang construct an integrable adelic line bundle
$$
 \left\langle\overline M_0,\ldots,\overline M_n\right\rangle_f
 \in\widehat{\Pic}(B)_{\intg,\RR};
$$
see \cite[Theorem~4.1.3 and Section~4.5]{YZ}.
These pairings are symmetric and multilinear.
They commute with base change to a normal integral base when the new total
space is integral.
The base change isomorphisms are isometries.
We use the following notation for the first Chern class of a Deligne pairing:
\begin{equation}
 f_*\!\left(\prod_{j=0}^{n}\widehat c_1(\overline M_j)\right)
 :=\widehat c_1\!\left(\left\langle\overline M_0,\ldots,\overline M_n\right\rangle_f\right).
 \label{eq:curve-deligne-pushforward}
\end{equation}
For top intersection numbers, we use the projection formula in
\cite[Lemma~4.6.1(1)]{YZ}.

Define
\begin{equation}
 \overline A
 :=\left\langle
 \underbrace{\overline\xi,\ldots,\overline\xi}_{r\ \mathrm{factors}}
 \right\rangle_{Y/C}
 \in\widehat{\Pic}(C)_{\intg,\RR},
 \label{eq:curve-deligne-line}
\end{equation}
and let $A$ be its underlying algebraic line bundle.

\begin{proposition}[Algebraic normalization of the Deligne line]
 \label{prop:curve-deligne-determinant}
 There is an isomorphism $A\simeq\det E$.  In particular, if
 $\det E\simeq\mathcal O_C$, then $\deg A=0$.
\end{proposition}

\begin{proof}
 We work in ordinary Chow groups with integral coefficients.
 The first Chern class formula for the algebraic Deligne pairing gives
 $$
  c_1(A)=\pi_*\!\left(c_1(\xi)^r\right)
  \quad\text{in }\mathrm{CH}^1(C);
 $$
 see \cite[Sections~4.1.2 and~4.2.1]{YZ}.
 The projective bundle formula for line quotients gives
 $$
  \pi_*\!\left(c_1(\xi)^r\right)=c_1(E)=c_1(\det E);
 $$
 see \cite[Chapter~3]{Fulton}.
 Thus $c_1(A)=c_1(\det E)$ in $\mathrm{CH}^1(C)$.
 Since $C$ is smooth, the first Chern class map is an isomorphism
 $$
  c_1:\Pic(C)\xrightarrow{\ \sim\ }\mathrm{CH}^1(C).
 $$
 Hence $A$ and $\det E$ have the same class in $\Pic(C)$.
 This proves $A\simeq\det E$.
 If $\det E\simeq\mathcal O_C$, then $A\simeq\mathcal O_C$ and $\deg A=0$.
\end{proof}

\subsection{Intersection on the fibre product}

Put $W=Y\times_CY$, let $p_1,p_2:W\to Y$ be the projections, and let
$q=\pi\circ p_1=\pi\circ p_2$.  We express the first intersection in the curve discriminant as
$\overline A^2$.

\begin{proposition}
 \label{prop:curve-fibre-square}
 One has
 \begin{equation}
 \overline A^2
 =\widehat{\deg}_W\left(
 \widehat c_1(p_1^*\overline\xi)^r
 \widehat c_1(p_2^*\overline\xi)^r
 \right).
 \label{eq:curve-fibre-square}
 \end{equation}
 Consequently,
 \begin{equation}
 \Delta_C^\tau(E,\overline\xi)
 =(r+1)\overline A^2-2r\,\overline\xi^{r+1}.
 \label{eq:curve-discriminant-deligne-form}
 \end{equation}
\end{proposition}

\begin{proof}
 The morphism $\pi:Y\to C$ is smooth and projective of relative dimension $r-1$.
 Both $C$ and $Y$ are normal and integral.
 We have $W\simeq\mathbb P_Y(\pi^*E)$, so $W$ is integral.
 Thus \cite[Theorem~4.1.3]{YZ} applies to base change by $\pi:Y\to C$.
 It gives an isometry
 \begin{equation}
 \pi^*\overline A
 \simeq
 \left\langle
 \underbrace{p_2^*\overline\xi,\ldots,
 p_2^*\overline\xi}_{r\ \mathrm{factors}}
 \right\rangle_{W/Y},
 \label{eq:curve-deligne-base-change}
 \end{equation}
 Here the pairing uses the morphism $p_1:W\to Y$.
 We use the pushforwards defined by Deligne pairings in
 \cref{eq:curve-deligne-pushforward}.
 The definition of $\overline A$ gives
 $$
 \pi_*\!\left(\widehat c_1(\overline\xi)^r\right)
 =\widehat c_1(\overline A).
 $$
 Apply \cref{eq:curve-deligne-pushforward} to
 \cref{eq:curve-deligne-base-change} to obtain
 $$
 (p_1)_*\!\left(\widehat c_1(p_2^*\overline\xi)^r\right)
 =\widehat c_1(\pi^*\overline A).
 $$
 All the line bundles in these formulas are integrable.
 Hence \cite[Proposition~4.1.1]{YZ} defines their top intersection numbers.

 Apply the projection formula in \cite[Lemma~4.6.1(1)]{YZ} to $\pi$.
 Substitute $\pi_*(\widehat c_1(\overline\xi)^r)=\widehat c_1(\overline A)$.
 We obtain
 $$
 \widehat{\deg}_Y\!\left(
 \widehat c_1(\pi^*\overline A)\widehat c_1(\overline\xi)^r\right)
 =\widehat{\deg}_C\!\left(\widehat c_1(\overline A)^2\right)
 =\overline A^2.
 $$
 Now apply the same projection formula to $p_1$.
 Substitute $(p_1)_*(\widehat c_1(p_2^*\overline\xi)^r)
 =\widehat c_1(\pi^*\overline A)$.
 This gives
 \begin{equation}
 \widehat{\deg}_W\!\left(\widehat c_1(p_1^*\overline\xi)^r
 \widehat c_1(p_2^*\overline\xi)^r\right)
 =\widehat{\deg}_Y\!\left(
 \widehat c_1(\pi^*\overline A)\widehat c_1(\overline\xi)^r\right)
 =\overline A^2.
 \label{eq:curve-fibre-square-proof}
 \end{equation}
 The last equality follows from the preceding calculation for $\pi$.
 This proves \cref{eq:curve-fibre-square}.

 The specialization to $d=1$ of
 \cref{eq:bg-tautological-discriminant} is
 \begin{equation}
 (r+1)\widehat{\deg}_W\left(
 \widehat c_1(p_1^*\overline\xi)^r
 \widehat c_1(p_2^*\overline\xi)^r
 \right)
 -2r\,\overline\xi^{r+1}.
 \label{eq:curve-specialized-fibre-discriminant}
 \end{equation}
 Substitution of \eqref{eq:curve-fibre-square} proves
 \eqref{eq:curve-discriminant-deligne-form}.
\end{proof}

\subsection{The curve refinement}

\begin{theorem}[Refinement for a projective bundle over a curve]
 \label{thm:curve-pb}
 In the setting above, assume in addition that
 $\det E\simeq\mathcal O_C$.  Then
 \begin{equation}
 \Delta_C^\tau(E,\overline\xi)\geq0,
 \qquad
 \overline A^2\leq0,
 \qquad
 \overline\xi^{r+1}\leq0.
 \label{eq:curve-three-signs}
 \end{equation}
 More precisely,
 \begin{equation}
 \overline\xi^{r+1}
 =\frac{r+1}{2r}\overline A^2
 -\frac{1}{2r}\Delta_C^\tau(E,\overline\xi).
 \label{eq:curve-main-identity}
 \end{equation}

 The following equality conditions are equivalent:
 \begin{enumerate}[label=\textup{(\roman*)}]
  \item $\overline\xi^{r+1}=0$;
  \item $\overline A^2=0$ and
  $\Delta_C^\tau(E,\overline\xi)=0$;
  \item a positive multiple of $\overline A$ is pulled back from
  $\operatorname{Spec}K$ and $D_m\to0$.
 \end{enumerate}
 In particular, $\overline\xi^{r+1}<0$ if and only if
 $\overline A^2<0$ or
 $\Delta_C^\tau(E,\overline\xi)>0$.

 If the chosen BG admissible model system has a uniform discriminant gap,
 namely if there are $\varepsilon>0$ and $m_0$ such that $D_m\geq\varepsilon$
 for every $m\geq m_0$, then
 \begin{equation}
 \Delta_C^\tau(E,\overline\xi)\geq\varepsilon,
 \qquad
 \overline\xi^{r+1}
 \leq-\frac{\varepsilon}{2r}<0.
 \label{eq:curve-uniform-gap}
 \end{equation}
\end{theorem}

\begin{proof}
 By \cref{thm:curve-controlled-bg}, one has
 $\Delta_C^\tau(E,\overline\xi)=\lim_mD_m\geq0$.  By
\cref{prop:curve-deligne-determinant}, the underlying algebraic
 line bundle of $\overline A$ has degree zero.  The adelic Hodge index theorem
 on $C$ therefore gives $\overline A^2\leq0$; see
 \cite[Theorem~3.2 and Section~3.2]{YuanZhangHodge}.  Solving
 \cref{eq:curve-discriminant-deligne-form} for
 $\overline\xi^{r+1}$ proves \eqref{eq:curve-main-identity}, and both terms on
 its right side are nonpositive.  This proves
 \eqref{eq:curve-three-signs}.

 The same observation shows that the right side of
 \eqref{eq:curve-main-identity} vanishes if and only if both of its summands
 vanish.  The equality statement in the curve case of the adelic Hodge index
 theorem \cite[Theorem~3.2 and Section~3.2]{YuanZhangHodge} says that
 $\overline A^2=0$ if and only if there are an integer
 $q>0$ and an adelic line bundle $\overline N$ on $\operatorname{Spec}K$ such
 that
 \begin{equation}
 q\overline A=\pi_C^*\overline N
 \quad\text{in }\widehat{\Pic}(C)_{\intg,\QQ},
 \label{eq:curve-hodge-equality}
 \end{equation}
 where $\pi_C:C\to\operatorname{Spec}K$ is the structure morphism.  The
 equality criterion in \cref{prop:bg-controlled-transfer} gives
 $\Delta_C^\tau(E,\overline\xi)=0$ if and only if $D_m\to0$.  This proves
 the three equivalent conditions and the strictness criterion.

 Finally, a uniform gap remains a gap after taking the limit, so
 $\Delta_C^\tau(E,\overline\xi)\geq\varepsilon$.  Substitution into
 \eqref{eq:curve-main-identity}, together with $\overline A^2\leq0$, gives
 \eqref{eq:curve-uniform-gap}.
\end{proof}

\begin{remark}
\label{rem:curve-pb-scope}
 We use $\det E\simeq\mathcal O_C$ and
 \cref{prop:curve-deligne-determinant} to obtain $A\simeq\mathcal O_C$.
 Hence $\deg A=0$.
 We then apply the adelic Hodge index theorem to obtain $\overline A^2\leq0$.
 \Cref{thm:curve-controlled-bg} proves
 $\Delta_C^\tau(E,\overline\xi)\geq0$ without assuming
 $\det E\simeq\mathcal O_C$.

 For the chosen BG admissible model system,
 \cref{prop:bg-controlled-transfer} gives
 $$
 \Delta_C^\tau(E,\overline\xi)=0
 \quad\Longleftrightarrow\quad
 \lim_{m\to\infty}D_m=0.
 $$
 The limit condition does not imply $D_m=0$ for any fixed $m$.

 Semistability and $\det E\simeq\mathcal O_C$ alone do not imply
 $\overline\xi^{r+1}<0$.
 To see this, take $E=\mathcal O_C^{\oplus r}$.
 Choose a regular arithmetic model $\mathcal C$ of $C$.
 For every $m$, take $\mathcal E_m=\mathcal O_{\mathcal C}^{\oplus r}$.
 Make the standard basis orthonormal at every complex embedding.
 Let $\overline\xi$ be the resulting model adelic quotient metric.
 The standard quotient line bundle $\overline{\boldsymbol\xi}_m$ is nef.
 In \textup{(BG3)} of \cref{def:bg-admissible-model-system}, take
 $$
 \overline{\mathcal M}_m^+=\overline{\boldsymbol\xi}_m,
 \qquad
 \overline{\mathcal M}_m^-=\overline{\mathcal O}_{\mathcal Y_m}.
 $$
 These constant nef sequences give a BG admissible model system.
 Each $\overline{\mathcal E}_m$ is an orthogonal direct sum of trivial
 Hermitian line bundles.
 The arithmetic Whitney formula gives
 $$
 \widehat c_i^{\mathrm{GS}}(\overline{\mathcal E}_m)=0,
 \qquad i=1,2.
 $$
 Thus \cref{eq:bg-model-discriminant} gives $D_m=0$ for every $m$.
 \Cref{thm:curve-controlled-bg} then gives
 $\Delta_C^\tau(E,\overline\xi)=0$.
 The quotient metric on $Y=C\times_K\mathbb P_K^{r-1}$ comes from
 $\mathbb P_K^{r-1}$.
 Base change for the Deligne pairing gives
 $\overline A\simeq\pi_C^*\overline N$ for an adelic line bundle
 $\overline N$ on $\operatorname{Spec}K$.
 Here $\pi_C:C\to\operatorname{Spec}K$ is the structure morphism.
 The projection formula gives $\overline A^2=0$.
 Equation~\eqref{eq:curve-main-identity} now gives $\overline\xi^{r+1}=0$.

 In general, \cref{eq:curve-main-identity} and the signs in
 \cref{eq:curve-three-signs} give
 $$
 \overline\xi^{r+1}<0
 \quad\Longleftrightarrow\quad
 \overline A^2<0\ \text{or}\ \Delta_C^\tau(E,\overline\xi)>0.
 $$
 For the chosen model system, \cref{thm:curve-controlled-bg} gives
 $D_m\to\Delta_C^\tau(E,\overline\xi)$.
 If $\Delta_C^\tau(E,\overline\xi)>0$, set
 $\varepsilon=\Delta_C^\tau(E,\overline\xi)/2$.
 Convergence gives an index $m_0$ such that $D_m\geq\varepsilon$ for every
 $m\geq m_0$.
 Conversely, suppose $D_m\geq\varepsilon>0$ for every $m\geq m_0$.
 Taking the limit gives $\Delta_C^\tau(E,\overline\xi)\geq\varepsilon>0$.
 Thus a positive lower bound for these $D_m$ suffices to prove
 $\overline\xi^{r+1}<0$.
 If $\overline A^2<0$, \cref{eq:curve-main-identity} already gives
 $\overline\xi^{r+1}<0$.
\end{remark}

\section{Controlled numerical adelic Bogomolov--Gieseker inequality
in higher dimension}
\label{sec:controlled-numerical-bg}

Retain the notation fixed at the beginning of \cref{sec:curve-pb}.
Fix a BG admissible model system as in \cref{def:bg-admissible-model-system}.
Assume throughout this section that $d\geq2$.
Assume in addition throughout this section that $E_{\overline K}$ is
slope semistable with respect to $L_{\overline K}$.
\Cref{prop:bg-exact-model-identity} proves $B_m=D_m$ for every $m$.
\Cref{prop:bg-controlled-transfer} gives the limit of the real sequence $(D_m)$
defined in \cref{eq:bg-model-discriminant}.
We now prove nonnegativity on each arithmetic model.
Its generic fibre has dimension $d\geq2$.

\begin{theorem}[Moriwaki's arithmetic inequality in higher dimension]
\label{thm:bg-moriwaki-input}
Let $\mathcal X$ be regular, integral, projective, and flat over
$\mathcal O_K$, with generic fibre of dimension $d\geq2$.  Let
$\overline{\mathcal H}$ be an arithmetically ample smooth Hermitian line
bundle and let $\overline{\mathcal E}$ be a smooth Hermitian vector bundle
of rank $r$.  Assume that $\mathcal E_\sigma$ is slope semistable with
respect to $c_1(\overline{\mathcal H}_\sigma)$ on every connected complex
component.  Then
\begin{equation}
 \whdeg_K\left(
 \left(2r\widehat c_2^{\mathrm{GS}}(\overline{\mathcal E})
 -(r-1)\widehat c_1^{\mathrm{GS}}(\overline{\mathcal E})^2\right)
 \widehat c_1(\overline{\mathcal H})^{d-1}
 \right)\geq0.
 \label{eq:bg-moriwaki-higher-input}
\end{equation}
\end{theorem}

\begin{proof}
This is the main theorem of \cite{MoriwakiHigherBG}, multiplied by $2r$ and
divided by the positive number $[K:\mathbb Q]$.  The regularity imposed here
allows us to use the Gillet--Soul\'e arithmetic Chow ring in the common model
identity of Section~\ref{sec:curve-pb}.
\end{proof}

\begin{remark}
\label{rem:bg-model-dimension-reduction}
Moriwaki proves \cref{thm:bg-moriwaki-input} by induction on the absolute
dimension of the arithmetic model.
The base case is his theorem for arithmetic surfaces,
\cref{thm:bg-moriwaki-surface-input}; see
\cite[Sections~3--6]{MoriwakiHigherBG}.
Fix an arithmetic model from this section.
Write $\overline{\mathcal E}=(\mathcal E,h)$.
Choose a section $s\in H^0(\mathcal X,\mathcal H^{\otimes N})$ as in
\cite[Section~6]{MoriwakiHigherBG}, with $N>0$ and $\|s\|_{\sup}<1$.
Write
$$
 \operatorname{div}(s)=\mathcal Z+\sum_i a_i\mathcal F_i.
$$
Here $\mathcal Z$ is the horizontal divisor.
Each $\mathcal F_i$ is a smooth fibre over a finite prime $\mathfrak p_i$.
Its multiplicity $a_i$ is positive.
For every embedding $\sigma:K\hookrightarrow\mathbb C$, the choice of $s$
makes $\mathcal Z_\sigma$ smooth.
It also makes $\mathcal E_\sigma|_{\mathcal Z_\sigma}$ semistable with
respect to $\mathcal H_\sigma|_{\mathcal Z_\sigma}$.
Set $q_i:=\#\kappa(\mathfrak p_i)$.
Define the algebraic class
$$
 \beta:=2r c_2(\mathcal E)-(r-1)c_1(\mathcal E)^2.
$$
For every embedding $\sigma:K\hookrightarrow\mathbb C$, put
$\omega_\sigma:=c_1(\overline{\mathcal H}_\sigma)$.
Define the form
$$
 \Phi_\sigma(h):=
 \left(2r c_2(\mathcal E_\sigma,h_\sigma)
 -(r-1)c_1(\mathcal E_\sigma,h_\sigma)^2\right)\omega_\sigma^{d-2}.
$$
Let $D_{\mathcal X}(h)$ denote the left side of
\cref{eq:bg-moriwaki-higher-input}.
Write $\overline{\mathcal E}_{\mathcal Z}$ and
$\overline{\mathcal H}_{\mathcal Z}$ for the restricted Hermitian bundles.
Define their arithmetic Chern number by
$$
 D_{\mathcal Z}(h):=\whdeg_K\!\left(
 \left(2r\widehat c_2(\overline{\mathcal E}_{\mathcal Z})
 -(r-1)\widehat c_1(\overline{\mathcal E}_{\mathcal Z})^2\right)
 \widehat c_1(\overline{\mathcal H}_{\mathcal Z})^{d-2}\right).
$$
We compute this arithmetic Chern number as in
\cite[Section~6]{MoriwakiHigherBG}.
Define
$$
 V(s):=\frac{1}{[K:\mathbb Q]}\sum_i a_i\log q_i\,
 \deg_{\kappa(\mathfrak p_i)}\!\left(
 \beta|_{\mathcal F_i}\,c_1(\mathcal H|_{\mathcal F_i})^{d-2}\right)
$$
and
$$
 I(s,h):=-\frac{1}{[K:\mathbb Q]}
 \sum_{\sigma:K\hookrightarrow\mathbb C}
 \int_{\mathcal X_\sigma(\mathbb C)}\log\|s_\sigma\|\,\Phi_\sigma(h).
$$
The sum defining $I(s,h)$ runs over all embeddings, including conjugate pairs.
The arithmetic intersection formula in \cite[Section~6]{MoriwakiHigherBG}
gives
$$
 N D_{\mathcal X}(h)=D_{\mathcal Z}(h)+V(s)+I(s,h).
$$
The term $V(s)$ records the contributions of the vertical divisors
$\mathcal F_i$.
We must retain this term when we separate the horizontal divisor
$\mathcal Z$ from $\operatorname{div}(s)$.
Moriwaki obtains $V(s)\geq0$ from constancy of the algebraic Chern numbers
on smooth fibres and the inequality on the generic fibre.
His induction gives $D_{\mathcal Z}(h)\geq0$.

The function $-\log\|s_\sigma\|^2$ is a Green function for
$\operatorname{div}(s_\sigma)$.
The condition $\|s\|_{\sup}<1$ makes $-\log\|s_\sigma\|$ positive away
from this divisor.
It does not imply that $\Phi_\sigma(h)$ is a nonnegative form.
Thus we cannot infer $I(s,h)\geq0$ from the condition on $s$ alone.
The heat flow estimates in \cite[Section~6]{MoriwakiHigherBG} give metrics
$h_t$ such that
$$
 D_{\mathcal X}(h)\geq D_{\mathcal X}(h_t),
 \qquad
 \liminf_{t\to\infty}I(s,h_t)\geq0.
$$
The induction also gives $D_{\mathcal Z}(h_t)\geq0$ for every $t$.
Together with $V(s)\geq0$, the comparison formula gives
$N D_{\mathcal X}(h_t)\geq I(s,h_t)$.
Taking the lower limit proves $D_{\mathcal X}(h)\geq0$.

Moriwaki's proof on arithmetic models uses his theorem for arithmetic
surfaces.
In this section, we first apply Moriwaki's theorem to the arithmetic models.
We then apply \cref{prop:bg-controlled-transfer} to take the adelic limit.
\end{remark}

\subsection{Nonnegativity on arithmetic models}

\begin{proposition}[Nonnegativity on arithmetic models in higher dimension]
\label{prop:bg-model-nonnegativity}
Assume that $E_{\overline K}$ is slope semistable with respect to
$L_{\overline K}$.
For every model in a BG admissible model system in dimension $d\geq2$, one
has
\begin{equation}
 B_m=D_m\geq0.
 \label{eq:bg-higher-model-nonnegative}
\end{equation}
\end{proposition}

\begin{proof}
\Cref{prop:bg-exact-model-identity} gives $B_m=D_m$.
Condition \textup{(BG1)} in \cref{def:bg-admissible-model-system} gives a
regular arithmetic model $\mathcal X_m$ of absolute dimension $d+1$.
By \textup{(BG2)}, choose a positive integer $N_m$ such that
$\overline{\mathcal H}_m:=N_m\overline{\mathcal L}_m$ is a smooth
Hermitian line bundle satisfying \textup{(A1)}, \textup{(A2)}, and
\textup{(A3)} stated at the beginning of Section~\ref{sec:curve-pb}.
Thus $\mathcal H_m$ is relatively ample over $\operatorname{Spec}\mathcal O_K$.
For every embedding $\sigma:K\hookrightarrow\mathbb C$, the form
$c_1(\overline{\mathcal H}_{m,\sigma})$ is a K\"ahler form.
For every integral horizontal subvariety $\mathcal Z\subseteq\mathcal X_m$,
put $q:=\dim\mathcal Z$.
The third condition gives
$$
 \whdeg\!\left(
 \whc_1(\overline{\mathcal H}_m|_{\mathcal Z})^q\right)>0.
$$
For every embedding $\sigma:K\hookrightarrow\mathbb C$, the fixed
isomorphisms on the generic fibres give
\begin{equation}
 (\mathcal E_m)_\sigma\simeq E_\sigma,
 \qquad
 (\mathcal L_m)_\sigma\simeq L_\sigma.
 \label{eq:bg-complex-fibre-identifications}
\end{equation}
Fix an embedding $\sigma:K\hookrightarrow\mathbb C$.
Extend it to an embedding $\overline K\hookrightarrow\mathbb C$.
Our hypothesis gives slope semistability of $E_{\overline K}$ with
respect to $L_{\overline K}$.
Apply the preservation of slope semistability under extension of the
base field; see \cite[Corollary~1.3.8 and Theorem~1.6.6]{HuybrechtsLehn}.
We obtain slope semistability of $E_\sigma$ with respect to $L_\sigma$.

Put $\omega_{m,\sigma}:=c_1(\overline{\mathcal L}_{m,\sigma})$.
The equality $N_m\omega_{m,\sigma}
=c_1(\overline{\mathcal H}_{m,\sigma})$ shows that
$\omega_{m,\sigma}$ is a K\"ahler form.
The second isomorphism in \cref{eq:bg-complex-fibre-identifications}
gives $[\omega_{m,\sigma}]=c_1(L_\sigma)$ in
$H^2(X_\sigma(\mathbb C),\mathbb R)$.

Let $G$ be a coherent sheaf on $X_\sigma$ without torsion and of positive rank.
Write $c_1(G):=c_1(\det G)$.
Define its algebraic slope by
$$
 \mu_{L_\sigma}(G):=
 \frac{\deg\!\left(c_1(G)c_1(L_\sigma)^{d-1}\right)}{\operatorname{rk}G}.
$$
In the integral below, represent $c_1(G)$ by any smooth closed form.
Define the slope with respect to $\omega_{m,\sigma}$ by
$$
 \mu_{\omega_{m,\sigma}}(G):=
 \frac{\int_{X_\sigma(\mathbb C)}c_1(G)\,\omega_{m,\sigma}^{d-1}}
 {\operatorname{rk}G}
 =\mu_{L_\sigma}(G).
$$
The equality follows from $[\omega_{m,\sigma}]=c_1(L_\sigma)$ and the
comparison of algebraic and complex Chern classes.
Let $F\subset E_\sigma$ be a coherent subsheaf with
$0<\operatorname{rk}F<r$.
The semistability of $E_\sigma$ gives
$$
 \mu_{\omega_{m,\sigma}}(F)=\mu_{L_\sigma}(F)
 \leq\mu_{L_\sigma}(E_\sigma)
 =\mu_{\omega_{m,\sigma}}(E_\sigma).
$$
The variety $X_\sigma$ is projective.
By GAGA, every coherent analytic subsheaf of $E_\sigma^{\mathrm{an}}$
comes from a coherent algebraic subsheaf of $E_\sigma$.
Thus the same inequality proves analytic slope semistability.
The first isomorphism in \cref{eq:bg-complex-fibre-identifications}
identifies $E_\sigma$ with $(\mathcal E_m)_\sigma$.
Since $X$ is geometrically integral, $X_\sigma(\mathbb C)$ is connected.
We have therefore proved the required semistability on each complex component.
Replacing $\omega_{m,\sigma}$ by $N_m\omega_{m,\sigma}$ multiplies both
sides of the slope inequality by $N_m^{d-1}>0$.
Hence $(\mathcal E_m)_\sigma$ is also semistable with respect to
$c_1(\overline{\mathcal H}_{m,\sigma})$.

Apply \cref{thm:bg-moriwaki-input} to
$\overline{\mathcal E}_m$ and $\overline{\mathcal H}_m$.
In Moriwaki's original normalization, the resulting inequality reads
\begin{equation}
 \frac{N_m^{d-1}[K:\mathbb Q]}{2r}\,D_m\geq0.
 \label{eq:bg-moriwaki-normalization}
\end{equation}
The factor $N_m^{d-1}$ comes from
$\widehat c_1(\overline{\mathcal H}_m)
=N_m\widehat c_1(\overline{\mathcal L}_m)$.
Moriwaki uses $\widehat c_2-(r-1)\widehat c_1^2/(2r)$, which gives
the factor $1/(2r)$.
He uses the arithmetic degree before division by $[K:\mathbb Q]$,
which gives the remaining factor $[K:\mathbb Q]$.
All three factors are positive, so $D_m\geq0$.
\end{proof}

\subsection{The adelic inequality in higher dimension}

\begin{theorem}[Controlled numerical adelic Bogomolov--Gieseker inequality
in higher dimension]
\label{thm:controlled-numerical-bg}
Let $X/K$ be geometrically integral, smooth, and projective of dimension
$d\geq2$.  Let $L$ be ample, let $E$ have rank $r\geq2$, and assume that
$E_{\overline K}$ is slope semistable with respect to $L_{\overline K}$.
Let $\overline L$ be nef and let $\overline\xi$ be an integrable
adelic metric on $\mathcal O_{\mathbb P_X(E)}(1)$.  If
$(\overline L,\overline\xi)$ admits a BG admissible model system, then
\begin{equation}
 \Delta^\tau_{\overline L}(E,\overline\xi)\geq0.
 \label{eq:bg-main-inequality}
\end{equation}
More precisely,
\begin{equation}
 \Delta^\tau_{\overline L}(E,\overline\xi)
 =\lim_{m\to\infty}D_m,
 \qquad D_m\geq0.
 \label{eq:bg-main-limit}
\end{equation}
Equivalently, the numerical pairing in
\cref{eq:bg-numerical-bridge} is nonnegative.  Equality holds if and only if
$D_m\to0$.  If $D_m\geq\varepsilon>0$ for every sufficiently large $m$,
then
\begin{equation}
 \Delta^\tau_{\overline L}(E,\overline\xi)
 \geq\varepsilon>0.
 \label{eq:bg-higher-uniform-gap}
\end{equation}
\end{theorem}

\begin{proof}
Fix a BG admissible model system for $(\overline L,\overline\xi)$.
Let $D_m$ be the real number defined in \cref{eq:bg-model-discriminant}.
Condition \textup{(BG1)} in \cref{def:bg-admissible-model-system} supplies
the regular arithmetic models $\mathcal X_m$ and the smooth Hermitian
bundles $\overline{\mathcal E}_m$ extending $E$.
Condition \textup{(BG2)} supplies the arithmetically ample model
polarizations $\overline{\mathcal L}_m$ extending $L$.
The theorem assumes $d\geq2$ and slope semistability of
$E_{\overline K}$ with respect to $L_{\overline K}$.
Thus every chosen model satisfies the hypotheses of
\cref{prop:bg-model-nonnegativity}.
That proposition gives $B_m=D_m\geq0$ for every $m$.

Apply \cref{prop:bg-controlled-transfer} to this model system.
It identifies $\lim_{m\to\infty}D_m$ with
$\Delta^\tau_{\overline L}(E,\overline\xi)$ and proves
\cref{eq:bg-main-inequality,eq:bg-main-limit}.
Its equality criterion gives the assertion that equality holds if and only
if $D_m\to0$.
Its uniform lower bound gives \cref{eq:bg-higher-uniform-gap}.

Finally, put $\overline E:=(E,\overline\xi)$.
Use its numerical Chern classes from \cref{lem:bg-numerical-bridge}.
That lemma identifies the numerical pairing with the limit just obtained:
$$
\left\langle 2r\widehat c_2(\overline E)-(r-1)\widehat c_1(\overline E)^2,\overline L^{d-1}\mid X\right\rangle=\Delta^\tau_{\overline L}(E,\overline\xi)=\lim_{m\to\infty}D_m\geq0.
$$
This proves the stated inequality for the numerical pairing.
\end{proof}

\begin{corollary}[All dimensions]
\label{cor:bg-all-dimensions}
The controlled numerical adelic Bogomolov--Gieseker inequality holds in every
dimension $d\geq1$: the case $d=1$ is
\cref{thm:curve-controlled-bg}, and the case $d\geq2$ is
\cref{thm:controlled-numerical-bg}.
\end{corollary}

\begin{remark}[Relation to classical inequalities]
\label{rem:bg-classical-shadow}
On a fixed arithmetic model, the identity $B_m=D_m$ identifies
$B_m\geq0$ with Moriwaki's arithmetic inequality for
$\overline{\mathcal E}_m$.  At every embedding
$\sigma:K\hookrightarrow\mathbb C$, geometric semistability gives the usual
complex Bogomolov--Gieseker inequality
\begin{equation}
 \int_{X_\sigma}
 \left(2r c_2(E_\sigma)-(r-1)c_1(E_\sigma)^2\right)
 c_1(L_\sigma)^{d-2}\geq0.
 \label{eq:bg-classical-shadow}
\end{equation}
For $d=2$ this is the classical surface inequality; for $d>2$ it follows by
successive semistable restriction to a general surface obtained as a complete intersection
\cite{HuybrechtsLehn,MehtaRamanathan}.
\end{remark}

\renewcommand{\refname}{References}

\end{document}